\documentclass[a4paper,11pt]{amsart}
\usepackage{amssymb, amsmath, amsthm, stmaryrd, mathrsfs}
\usepackage{mdwlist}

\newcommand\beqn{\begin{equation}}
\newcommand\eeqn{\end{equation}}
\newcommand\beqny{\begin{eqnarray}}
\newcommand\eeqny{\end{eqnarray}}
\newcommand\beqnyn{\begin{eqnarray*}}
\newcommand\eeqnyn{\end{eqnarray*}}

\usepackage{amsmath,amsfonts,amssymb,amsthm}
\newtheorem{theorem}{Theorem}[section]
\newtheorem{lemma}[theorem]{Lemma}

\newtheorem{corollary}[theorem]{Corollary}
\newtheorem{definition}[theorem]{Definition}
\newtheorem{remark}[theorem]{Remark}

\numberwithin{equation}{section}
\newcommand{\op}[1]{\operatorname{\text{\rm #1}}} 
\allowdisplaybreaks[1]

\begin{document}

\title[Regularity of minimal submanifolds and mean curvature flow]{Regularity of minimal submanifolds and \\mean curvature flows with a common free boundary}
\author{Brian Krummel} 
\thanks{Mathematics Department, The University of Texas at Austin, \\ \indent 
2515 Speedway Stop C1200, RLM 8.100 \\ \indent Austin, TX 78712, USA \\ \indent 
bkrummel@math.utexas.edu}

\maketitle

\begin{abstract} 
Let $N$ be a smooth $(n+l)$-dimensional Riemannian manifold.  We show that if $V$ is an area-stationary union of three or more $C^{1,\mu}$ $n$-dimensional submanifolds-with-boundary $M_k \subset N$ with a common boundary $\Gamma$, then $\Gamma$ is smooth and each $M_k$ is smooth up to $\Gamma$ (real-analytic in the case $N$ is real-analytic).  This extends a previous result of the author for codimension $l = 1$. 

We additionally show that if $\{V_t\}_{t \in (-1,1)}$ is a Brakke flow such that each time-slice $V_t$ is a union of three or more $n$-dimensional submanifolds-with-boundary $M_{k,t} \subset N$ with a common boundary $\Gamma_t$ and with parabolic $C^{2+\mu}$ regularity in time-space, then $\{\Gamma_t\}_{t \in (-1,1)}$ and $\{M_{k,t}\}_{t \in (-1,1)}$ are smooth (second Gevrey with real-analytic time-slices in the case $N$ is real-analytic). 
\end{abstract}

\setcounter{tocdepth}{1}

\section{Introduction}

We will consider the higher regularity of unions of three or more submanifolds with a common boundary that arise as singular minimal submanifolds and mean curvature flows.  In particular, we will prove the following result for singular minimal submanifolds, i.e. stationary integral varifolds: 

\noindent \textbf{Theorem A.}  \textit{Let $V$ be a stationary $n$-dimensional integral varifold in a smooth $(n+l)$-dimensional Riemannian manifold $N$ such that $V$ consists of $q \geq 3$ $C^{1,\mu}$ $n$-dimensional submanifolds-with-boundary $M_k \subset N$ with (constant) integer multiplicities and a common boundary $\Gamma$.  Suppose $M_k$ are not all tangent to the same plane at any point of $\Gamma$.  Then $\Gamma$ is smooth and each $M_k$ is smooth up to $\Gamma$.  In the case that $N$ is real-analytic, $\Gamma$ is real-analytic and each $M_k$ is real-analytic up to $\Gamma$.}

As will be discussed in Sections \ref{sec:preliminaries_sec} and \ref{sec:results_sec} below, the hypotheses of Theorem A, in particular that $V$ is stationary, are equivalent to each $M_k$ having zero mean curvature and the sum of the unit conormals of $M_k$ along $\Gamma$ equaling zero. 

Theorem A was previously proven by Kinderlehrer, Nirenberg, and Spruck~\cite{KNS} in the special case $l = 1$ and $q = 3$ and later proven by the author~\cite{Krum} in the case $l = 1$ and $q \geq 3$ is arbitrary.  Here we extend Theorem A to codimension $l > 1$.  We also prove an analogous result for singular mean curvature flows, i.e. Brakke flows:

\noindent \textbf{Theorem B.}  \textit{Let $\{V_t\}_{t \in (-1,1)}$ be a Brakke flow in a smooth $(n+l)$-dimensional Riemannian manifold $N$ such that each time-slice $V_t$ is a sum of three or more $n$-dimensional submanifolds-with-boundary $M_{k,t} \subset N$ with parabolic $C^{2+\mu}$ regularity (see Section \ref{sec:parabolic_subsec}), constant integer multiplicities, and a common boundary $\Gamma_t$.  Suppose $M_{k,t}$ are not all tangent to the same plane at any point of $\Gamma_t$.  Then $\{\Gamma_t\}_{t \in (-1,1)}$ and $\{M_{k,t}\}_{t \in (-1,1)}$ are smooth in time-space.  In the case that $N$ is real-analytic, $\{\Gamma_t\}_{t \in (-1,1)}$ and $\{M_{k,t}\}_{t \in (-1,1)}$ are second Gevrey in time-space and all the time-slices $\Gamma_t$ and $M_{k,t}$ are real-analytic.}

As will be discussed in Section \ref{sec:results_sec} below, the hypotheses of Theorem B, in particular that $\{V_t\}_{t \in (-1,1)}$ is a Brakke flow, are equivalent to each $M_{k,t}$ flowing by mean curvature and the sum of the unit conormals of $M_{k,t}$ along $\Gamma_t$ equaling zero.  Note that since $M_{k,t}$ solve a parabolic problem, namely mean curvature flow, $\Gamma_t$ and $M_{k,t}$ are not generally expected to be real-analytic in time-space.  (A well-known example due to Kowalevsky~\cite{Kow75} shows that the solution $u$ to the heat equation $\partial u/\partial t = \partial^2 u/\partial x^2$ on $\mathbb{R}^2$ with initial condition $u = (1-x)^{-1}$ at $t = 0$ is not real analytic in $(t,x)$.)

The primary motivation for proving Theorem A for codimension $l = 1$ in~\cite{Krum} was a strengthening of Wickramasekera's general regularity theorem of~\cite{Wic}.  Wickramasekera showed that if a codimension one, stationary, integral $n$-dimensional varifold $V$ of $N$ is stable on its regular part and is nowhere locally the union of three or more $C^{1,\mu}$ hypersurface-with-boundary with a common boundary, then $\op{sing} V$ is empty if $n \leq 6$, discrete if $n = 7$, and has Hausdorff dimension at most $n-7$ when $n \geq 8$.  By~\cite{Krum}, the same conclusion holds true if we instead assume that $V$ is nowhere locally the union of three or more smooth hypersurfaces-with-boundary $M_k$ with a common boundary ($M_k$ real-analytic in the case that $N$ is real-analytic). 

Since the publication of~\cite{Krum}, it has become clear that Theorem A has important applications when codimension $> 1$.  In particular, the result can be applied to the recent work of Hughes of~\cite{Hughes} on the regularity of minimal Lipschitz two-valued graphs.  Hughes showed that if a minimal Lipschitz two-valued graph $V$ of $\mathbb{R}^{n+l}$ is $L^2$ close to the union of four $n$-dimensional half-planes with a common boundary that is not a union of two intersecting $n$-dimensional planes, then $V$ is locally the union of four $C^{1,\mu}$ submanifolds-with-boundary $M_k$ with a common boundary $\Gamma$.  As a consequence of Theorem A, $M_k$ and $\Gamma$ are in fact real-analytic, see Corollary \ref{hughes_cor} below.  

Theorems A and B have applications to recent work of Schulze and White~\cite{SchultzeWhite} on mean curvature flows of submanifold clusters with triple edges in codimension $\geq 1$.  A submanifold cluster with triple edges consists of smoothly embedded $n$-dimensional submanifolds meeting along $(n-1)$-dimensional edges in triples at equal angles and possibly meeting at higher order junctions.  A static triple junction is a union of three $n$-dimensional half-planes meeting along a common boundary at equal angles that is static in time.  Schulze and White show that if a smooth mean curvature flows with triple edges $\mathcal{M}^j$ converge weakly to a static triple junction $J$, then $\mathcal{M}^j$ converge smoothly to $J$.  Theorem B implies second Gevrey regularity of $\mathcal{M}^j$ and real-analyticity of the time-slices of $\mathcal{M}^j$ for large $j$.  Schultze and White apply their regularity result to prove smooth short time existence of smooth mean curvature flows with triple edges without higher order junctions.  Their approach to short time existence uses Ilmanen's elliptic regularization scheme~\cite{EllipticRegularization} to construct a flat chain mod $3$, $F_{\varepsilon}$, that minimizes the area functional with respect to a conformally Euclidean metric on $\mathbb{R}^{N+1}$ and is a translating soliton for mean curvature flow in $\mathbb{R}^N$ with the standard metric.  Schultze and White need my result, together with~\cite[Corollary 2]{Sim93} of Simon, to show that $F_{\varepsilon}$ is smooth.  Rescaling $F_{\varepsilon}$ in the time direction and letting $\varepsilon \downarrow 0$, Schulze and White produce the desired mean curvature flow on $\mathbb{R}^N$. 

Both~\cite{KNS} and~\cite{Krum} proved special cases of Theorem A using a hodograph transformation to transform the free boundary problem to a boundary value problem on a half-space and then apply the elliptic regularity theory of Agmon, Douglis, and Nirenberg~\cite{ADK1}~\cite{ADK2} and Morrey~\cite{Morrey}.  The main challenge is checking the complementing condition for the boundary values.  In~\cite{KNS}, the boundary condition was that $M_1,M_2,M_3$ meet along $\Gamma$ at constant angles. ~\cite{Krum} used the boundary condition that the sum of the unit conormals to $M_k$ equals zero along $\Gamma$, which in codimension one is equivalent to the sum of the unit normals to $M_k$ equaling zero along $\Gamma$.  We will extend the use of the boundary condition in~\cite{Krum} to codimension $> 1$. 

The proof of our regularity result for Brakke flow, Theorem B, is similar to the proof of Theorem A except we apply the parabolic regularity theory of Solonnikov~\cite{Parabolic} and we apply~\cite{Friedman} to prove Gevrey regularity.  The arguments of~\cite{Friedman} require some modification to account for the derivatives in time and space being weighted differently, in particular treating the combinatorial computations using a key inequality, \eqref{parareg_combo} below.  Due to these changes, we state a general Gevrey regularity theorem for parabolic systems, Theorem \ref{parareg_thm}, and include its proof. 

\noindent \textbf{Organization of paper.}  In Section \ref{sec:preliminaries_sec} we will discuss notation and preliminaries, including the basic facts about integral varifolds and the formal definitions of parabolic $C^{k+\mu}$ regularity of functions and submanifolds.  Those not familiar with varifolds might want to focus on the special case of varifolds which are sums of submanifolds with a common boundary as in Theorems A and B.  See Subsection \ref{sec:sums_subm_sec} for a discussion of the basic facts about this special class of varifolds.  In Section \ref{sec:results_sec} we restate Theorems A and B in more precise and useful forms and we discuss our application to the work of Hughes~\cite{Hughes}.  Section \ref{sec:elliptic_section} contains the proof of Theorem A using the partial hodograph transformation and Section \ref{sec:parabolic_section} similarly contains Theorem B in the case that $N$ is smooth.  In Section \ref{sec:gevrey_section}, we state and prove the general Gevrey regularity result for parabolic systems, Theorem \ref{parareg_thm}, from which we obtain Theorem B in the case that $N$ is real-analytic.

\section{Preliminaries} \label{sec:preliminaries_sec}

\subsection{Notation} \label{sec:notation_subsec} Let $N$ be embedded as smooth (real analytic) $(n+l)$-dimensional submanifold of $\mathbb{R}^{n+m}$, where $n \geq 1$ and $1 \leq l \leq m$ are integers. 

We shall use coordinates $X = (x_1,x_2,\ldots,x_{n+m})$ on $\mathbb{R}^{n+m}$.  We use coordinates $x = (x_1,x_2,\ldots,x_n)$ on $\mathbb{R}^n$ and let $x' = (x_1,x_2,\ldots,x_{n-1})$ so that $x = (x',x_n)$.  We let $t$ denote a time parameter.  In the proofs of Theorems A and B, we will let $y$ denote coordinates on $\mathbb{R}^n$ and $\tau$ denote coordinates in time after applying the hodograph transform.

For any integer $k$, we let $B^k_r(x_0) = \{ x \in \mathbb{R}^k : |x-x_0| < r \}$ for each $x_0 \in \mathbb{R}^k$ and $r > 0$.  When $k = n$, we let $B_r(x_0) = B^n_r(x_0)$. 

For each positive integer $k$, $\mathcal{H}^k$ denotes the $k$-dimensional Hausdorff measure.

\subsection{Integral varifolds}  Below we present the basic facts about integral varifolds.  We refer the reader to~\cite{GMT}, in particular Chapters 4 and 8, for a more thorough discussion.  In Subsection \ref{sec:sums_subm_sec}, we will discuss the simpler special case of varifolds equal to sums of submanifolds with a common boundary.

Let $\mathcal{O}$ be an open subset of the $(n+l)$-dimensional submanifold $N$ of $\mathbb{R}^{n+m}$.  A subset $M$ of $\mathcal{O}$ is \textit{countably $n$-rectifiable} if $M \subseteq E_0 \cup \bigcup_{k=1}^{\infty} f_k(\mathbb{R}^n)$ for a subset $E_0 \subseteq \mathcal{O}$ with $\mathcal{H}^n(E_0) = 0$ and a countable collection of Lipschitz functions $f_k : \mathbb{R}^n \rightarrow \mathcal{O}$.  It is known that for any $\mathcal{H}^n$-measurable subset $M$ of $\mathcal{O}$ with locally finite $\mathcal{H}^n$-measure, $M$ is countably $n$-rectifiable if and only if at $\mathcal{H}^n$-a.e.~$Y \in M$ there exists a linear subspace $T_Y M$, called the \textit{approximate tangent plane} of $M$ at $Y$, such that 
\begin{equation*}
	\lim_{\rho \downarrow 0} \int_{\eta_{Y,\rho}(M)} \zeta \, d\mathcal{H}^n = \int_{T_Y M} \zeta \, d\mathcal{H}^n
\end{equation*}
for every $\zeta \in C_c^0(\mathcal{O})$, where $\eta_{Y,\rho}(X) = (X-Y)/\rho$ for each $\rho > 0$ and $X \in \mathbb{R}^n$.  For example, any $n$-dimensional $C^1$ submanifold $M$ of $\mathcal{O}$ is countably $n$-rectifiable and at each $Y \in M$ the approximate tangent plane of $M$ at $Y$ is the (usual) tangent plane of $M$ at $Y$. 

An \textit{$n$-dimensional integral varifold} $V = \mathbf{v}(M,\theta)$ of $\mathcal{O}$ is a equivalence class of pairs of a countably $n$-rectifiable set $M \subset \mathcal{O}$ and a multiplicity function $\theta : M \rightarrow \mathbb{N}$ that is locally $\mathcal{H}^n$-integrable on $M$ such that $\mathbf{v}(M_1,\theta_1) = \mathbf{v}(M_2,\theta_2)$ whenever $\mathcal{H}^n((M_1 \setminus M_2) \cup (M_2 \setminus M_1)) = 0$ and $\theta_1 = \theta_2$ $\mathcal{H}^n$-a.e.~on $M_1 \cap M_2$.  For example, when $M$ is a $n$-dimensional $C^1$ submanifold of $\mathcal{O}$, $|M| = \mathbf{v}(M,1)$ is an $n$-dimensional integral varifold.  To each integral varifold $V = \mathbf{v}(M,\theta)$ we associate a Radon measure $\|V\|$ on $\mathcal{O}$ such that for each Borel set $A \subseteq \mathcal{O}$, 
\begin{equation*}
	\|V\|(A) = \int_{M \cap A} \theta \, d\mathcal{H}^n
\end{equation*}
represents the $n$-dimensional area of $V$ in $A$.  (Note that~\cite{GMT} denotes $\|V\|$ by $\mu_V$.)

For each proper, Lipschitz map $f : \mathcal{O} \rightarrow \mathcal{O}'$ between two open subsets $\mathcal{O}$ and $\mathcal{O}'$ of Riemannian manifolds $N$ and $N'$ respectively, the \textit{pushforward} $f_{\#} V$ of the $n$-dimensional integral varifold $V = \mathbf{v}(M,\theta)$ is the $n$-dimensional integral varifold $f_{\#} V = \mathbf{v}(f(M),\widetilde{\theta})$, where $\widetilde{\theta}(y) = \sum_{x \in f^{-1}(y) \cap M} \theta(x)$ for each $y \in f(M)$. 

A sequence of $n$-dimensional integral varifolds $V_k = \mathbf{v}(M_k,\theta_k)$ converge to an integral varifold $V = \mathbf{v}(M,\theta_k)$ \textit{in the sense of varifolds} if and only if 
\begin{equation*}
	\int_{M \cap \mathcal{O}} \zeta(X,T_X M) \, \theta(X) \, d\mathcal{H}^n(X) 
	= \lim_{k \rightarrow \infty} \int_{M_k \cap \mathcal{O}} \zeta(X,T_X M_k) \, \theta_k(X) \, d\mathcal{H}^n(X)
\end{equation*}
for every $\zeta \in C^0_c(G_n(\mathcal{O}))$, where $G_n(\mathcal{O})$ is the Grassmannian 
\begin{equation*}
	G_n(\mathcal{O}) = \{ (X,S) : X \in \mathcal{O}, \, S \text{ is a $n$-dimensional linear subspace of } T_X N \}. 
\end{equation*}

Let $\zeta \in C^1_c(\mathcal{O};TN)$ be an arbitrary vector field and $f_t : \mathcal{O} \rightarrow \mathcal{O}$, $t \in (-1,1)$, be the one-parameter family of diffeomorphisms generated by $\zeta$.  The \textit{first variation of area} $\delta V : C^1_c(\mathcal{O};TN) \rightarrow \mathbb{R}$ of the integral varifold $V = \mathbf{v}(M,\theta)$ is the linear functional given by
\begin{equation*}
	\delta V(\zeta) = \left. \frac{d}{dt} \|f_{t\#} V\|(\op{spt} \zeta) \right|_{t=0} = \int_{M \cap \mathcal{O}} \op{div}_{T_X M} \zeta(X) \, d\|V\|(X), 
\end{equation*}
where $\op{div}_{T_X M} \zeta(X) = \sum_{i=1}^n \nabla_{\tau_i} \zeta(X) \cdot \tau_i$ for any orthonormal basis $\tau_1,\tau_2,\ldots,\tau_n$ of $T_X M$.  We say that an integrable varifold $V$ has \textit{generalized mean curvature} $H$ if $H$ is a locally $\|V\|$-integrable vector field on $\mathcal{O}$ such that 
\begin{equation*}
	\delta V(\zeta) = -\int_{\mathcal{O}} H(X) \cdot \zeta(X) \, d\|V\|(X) 
\end{equation*}
for all $\zeta \in C^1_c(\mathcal{O};TN)$.  We say that an integral varifold $V$ is \textit{stationary} if $\delta V(\zeta) = 0$ for all $\zeta \in C^1_c(\mathcal{O};TN)$. 

Let $V$ be a stationary integral varifold $V$ and $Y \in \op{spt} \|V\|$.  We say an $n$-dimensional integral varifold $C$ of $T_Y N$ is a \textit{tangent cone} to $V$ at $Y$ if $C = \lim_{j \rightarrow \infty} \eta_{Y,\rho_j \#} V$ in the sense varifolds of $\mathbb{R}^{n+m}$ for some $\rho_j \downarrow 0$, where $\eta_{Y,\rho}(X) = (X-Y)/\rho$ for each $\rho > 0$ and $X \in \mathbb{R}^{n+m}$.  There always exists at least one tangent cone of $V$ at $Y$, though it is unknown if in general there is a unique tangent cone independent of the sequence $\rho_j$.  Every tangent cone $C$ to a stationary integral varifold is itself a stationary integral varifold and is a cone in the sense that $C = \eta_{0,\rho \#} C$ for all $\rho > 0$. 

\subsection{Sums of submanifolds-with-boundary} \label{sec:sums_subm_sec} We will be primarily interested in $n$-dimensional integral varifolds $V$ in $\mathcal{O}$ of the form 
\begin{equation} \label{V_form}
	V = \sum_{k=1}^q \theta_k \, |M_k|
\end{equation}
consisting of $C^1$ submanifolds-with-boundary $M_k$ with integer multiplicity $\theta_k$ and a common boundary $\Gamma$ in $\mathcal{O}$.  Here the sum is taken by regarding integral varifolds as Radon measures on the Grassmannian so that $V$ is the integral varifold $V = \mathbf{v}(M_1 \cup M_2 \cup \cdots \cup M_q,\widetilde{\theta})$ with $\widetilde{\theta}(X) = \sum_{k \text{ s.t.}\, X \in M_k} \theta_k$ at $\mathcal{H}^n$-a.e.~$X \in M_1 \cup M_2 \cup \cdots \cup M_q$.  When $V$ is given by \eqref{V_form}, the area measure $\|V\|$ of $V$ is given by 
\begin{equation*}
	\|V\|(A) = \sum_{k=1}^q \theta_k \, \mathcal{H}^n(M_k \cap A)
\end{equation*}
for every Borel set $A \subseteq \mathcal{O}$. 

Suppose $M_k$ is a $C^2$ submanifold-with-boundary for all $k$.  Then by the divergence theorem, the first variation of the area of $V$ is given by 
\begin{equation*}
	\delta V(\zeta) 
	= \sum_{k=1}^q \int_{M_k} \op{div}_{T_X M_k} \zeta(X) \, \theta_k \, d\mathcal{H}^n(X)
	= \int_{\Gamma} \, \sum_{k=1}^q \eta_k \cdot \zeta \,  \theta_k \, d\mathcal{H}^{n-1} 
		- \sum_{k=1}^q \int_{M_k} H_k \cdot \zeta \, \theta_k \, d\mathcal{H}^n 
\end{equation*}
for all $\zeta \in C^1_c(N;TN)$, where $H_k$ is the mean curvature of $M_k$ and $\eta_k$ is the outward unit conormal vector to the boundary of $M_k$.  Thus $V$ has generalized mean curvature if and only if   
\begin{equation} \label{mainthm_tangents}
	\sum_{k=1}^q \theta_k \, \eta_k = 0 \text{ on } \Gamma . 
\end{equation}
$V$ is stationary if and only if $H_k = 0$ on $M_k$ for all $k$ and \eqref{mainthm_tangents} holds true. 

\subsection{Parabolic regularity}  \label{sec:parabolic_subsec}  For Theorem \ref{mainthm2}, our main regularity result for mean curvature flow, we need the following the notion of parabolic $C^{k+\mu}$ regularity based on~\cite{Parabolic}. 

\begin{definition} \label{parareg_defn}
Let $\Xi$ be an open subset of $\mathbb{R} \times \mathbb{R}^n$.  Let $f : \Xi \rightarrow \mathbb{R}^m$ be an arbitrary function.  For each $\mu \in (0,1]$, let 
\begin{align*}
	\langle f \rangle_{\mu,t;\Xi} &= \sup_{x \text{ s.t. } \Xi \cap \mathbb{R} \times \{x\} \neq \emptyset} 
		[D_t^{\alpha} D_x^{\beta} f(\,\cdot\,,x)]_{\mu;\Xi \cap \mathbb{R} \times \{x\}}, \\  
	\langle f \rangle_{\mu,x;\Xi} &= \sup_{t \text{ s.t. } \Xi \cap \{t\} \times \mathbb{R}^n \neq \emptyset} 
		[D_t^{\alpha} D_x^{\beta} f(t,\cdot\,)]_{\mu;\Xi \cap \{t\} \times \mathbb{R}^n}. 
\end{align*}
For each integer $k \geq 0$ and $\mu \in (0,1)$, we define
\begin{align*}
	\|f\|_{C^{k+\mu}_{\rm para}(\Xi)} = {}& \sum_{2\alpha + |\beta| \leq k} \sup_{\Xi} |D_t^{\alpha} D_x^{\beta} f| 
		+ \sum_{2\alpha + |\beta| = k-1} \langle D_t^{\alpha} D_x^{\beta} f \rangle_{(1+\mu)/2,t;\Xi} 
		\\&+ \sum_{2\alpha + |\beta| = k} (\langle D_t^{\alpha} D_x^{\beta} f \rangle_{\mu/2,t;\Xi} + \langle D_t^{\alpha} D_x^{\beta} f \rangle_{\mu,x;\Xi}). 
		\nonumber 
\end{align*}
We say that $f \in C^{k+\mu}_{\rm para}(\Xi;\mathbb{R}^m)$ if all derivatives $D_t^{\alpha} D_x^{\beta} f(t,x)$ exists on $\Xi$ for $2\alpha + |\beta| \leq k$ and $\|f\|_{C^{k+\mu}_{\rm para}(\Xi')} < \infty$ for every $\Xi' \subset \subset \Xi$.
\end{definition}

\begin{remark} {\rm
We will sometimes refer to $C^{k+\mu}_{\rm para}(\Xi;\mathbb{R}^m)$ as \textit{parabolic $C^{k+\mu}(\Xi;\mathbb{R}^m)$}.  We let $C^{k+\mu}_{\rm para}(\Xi) = C^{k+\mu}_{\rm para}(\Xi;\mathbb{R})$.  Similar conventions will also be used for the parabolic $C^{k+\mu}$ spaces defined in Definitions \ref{parabolic_side_defn}--\ref{parabolic_mnfd_defn} below. 
} \end{remark}

We want to consider parabolic $C^{k+\mu}$ regularity of functions up to portions of the boundary of $\Xi$.  In particular, we are interested in portions of the side of $\Xi$ where locally each time-slice of $\Xi$ has a submanifold boundary flowing regularly in time and we will not be interested in initial or final conditions.  For this we need the following. 

\begin{definition} \label{parabolic_side_defn}
Let $k \geq 1$ be an integer, $\mu \in (0,1)$, $\Xi$ be an open subset of $\mathbb{R} \times \mathbb{R}^n$, and $(t_0,x_0) \in \partial \Xi$.  When $n \geq 2$, we say that $\Xi$ is $C^{k+\mu}_{\rm para}$ near $(t_0,x_0)$ if there exists $\delta > 0$, a rotation $Q$ of $\mathbb{R}^n$, and $f \in C^{k+\mu}_{\rm para}((t_0-\delta,t_0+\delta) \times B^{n-1}_{\delta}(0))$ such that $f(t_0,0) = 0$, $D_{x'} f(t_0,0) = 0$, and 
\begin{equation*}
	\Xi \cap (\{t\} \times B_{\delta}(x_0)) = x_0 + Q \, \{ x \in B_{\delta}(0) : x_n < f(t,x') \} \text{ for all } t \in (t_0-\delta,t_0+\delta),
\end{equation*}
where we recall from Section \ref{sec:notation_subsec} that $x' = (x_1,\ldots,x_{n-1})$.  Similarly, when $n = 1$, we say $\Xi$ is $C^{k+\mu}_{\rm para}$ near $(t_0,x_0)$ if there exists $\delta > 0$ and a function $f : (t_0-\delta,t_0+\delta) \rightarrow \mathbb{R}$ in $C^{k/2,\mu/2}$ if $k$ is even and in $C^{(k-1)/2,(1+\mu)/2}$ if $k$ is odd such that $f(t_0) = x_0$ and 
\begin{equation*}
	\Xi \cap (\{t\} \times (x_0-\delta,x_0+\delta)) = \{ x \in (x_0-\delta,x_0+\delta) : x < f(t) \}  \text{ for all } t \in (t_0-\delta,t_0+\delta).
\end{equation*}

We say that $S \subseteq \partial \Xi$ is a $C^{k+\mu}_{\rm para}$ portion of the side of $\Xi$ if $\Xi$ is $C^{k+\mu}_{\rm para}$ near each point $(t_0,x_0) \in S$. 
\end{definition}

\begin{definition}
Let $\Xi$ be an open subset of $\mathbb{R} \times \mathbb{R}^n$ and let $S$ be a $C^{k+\mu}_{\rm para}$ portion of the side of $\Xi$.  Given an integer $k \geq 0$ and $\mu \in (0,1)$, we say a function $f : \Xi \cup S \rightarrow \mathbb{R}^m$ is in $C^{k+\mu}_{\rm para}(\Xi \cup S;\mathbb{R}^m)$ if all derivatives $D_t^{\alpha} D_x^{\beta} f(t,x)$ exists on $\Xi$ for $2\alpha + |\beta| \leq k$ and $\|f\|_{C^{k+\mu}_{\rm para}(\Xi')} < \infty$ for every open set $\Xi' \subseteq \Xi$ such that the closure of $\Xi'$ is relatively compact in $\Xi \cup S$. 
\end{definition}

\begin{definition} \label{parabolic_mnfd_defn}
Let $k \geq 1$ be an integer and $\mu \in (0,1)$.  Let $I$ be a bounded open interval of $\mathbb{R}$ and $\mathcal{O}$ be an open subset of $\mathbb{R}^{n+m}$.  We say a one-parameter family of $n$-dimensional submanifolds $\{M_t\}_{t \in I}$ of $\mathcal{O}$ is $C^{k+\mu}_{\rm para}$ if for each $t_0 \in I$ and $X_0 \in M_{t_0}$, there exists $\delta > 0$, a rotation $Q$ of $\mathbb{R}^{n+m}$, and a function $u \in C^{k+\mu}_{\rm para}((t_0-\delta,t_0+\delta) \times B_{\delta}(0); \mathbb{R}^m)$ such that 
\begin{equation*}
	M_t \cap B_{\delta}^{n+m}(X_0) = X_0 + Q \, (\{ (x,u(t,x)) : x \in B_{\delta}(0) \} \cap B^{n+m}_{\delta}(0)) \text{ whenever } |t-t_0| < \delta. 
\end{equation*}
Note that here we may allow $n = 0$ so that when $X_t$ is a point of $\mathcal{O}$ for each $t \in I$, the one parameter family of points $\{X_t\}_{t \in I}$ is $C^{k+\mu}_{\rm para}$ if the map $t \in I \mapsto X_t$ is $C^{k/2,\mu/2}$ if $k$ is even and is $C^{(k-1)/2,(1+\mu)/2}$ if $k$ is odd.

We say a one-parameter family of $n$-dimensional submanifolds-with-boundary $\{M_t\}_{t \in I}$ of $\mathcal{O}$ is $C^{k+\mu}_{\rm para}$ if for each $t_0 \in I$ and $X_0 \in \overline{M_{t_0}} \cap \mathcal{O}$, there exists $\delta > 0$, an open set $\Xi \subset (t_0-\delta,t_0+\delta) \times B_{\delta}(0)$ such that $S = \partial \Xi \cap (t_0-\delta,t_0+\delta) \times B_{\delta}(0)$ is empty or a $C^{k+\mu}_{\rm para}$ portion of the side of $\Xi$, a rotation $Q$ of $\mathbb{R}^{n+m}$, and a function $u \in C^{k+\mu}_{\rm para}(\Xi \cup S; \mathbb{R}^m)$ such that 
\begin{equation*}
	\overline{M_t} \cap B^{n+m}_{\delta}(X_0) = X_0 + Q \, (\{ (x,u(t,x)) : (t,x) \in \Xi \cup S \} \cap B^{n+m}_{\delta}(0)) \text{ whenever } |t-t_0| < \delta. 
\end{equation*}
\end{definition}

\section{Statement of main results and applications} \label{sec:results_sec}

\subsection{Main results} We restate our main regularity result for stationary varifolds, Theorem A of the introduction, as follows:

\begin{theorem} \label{mainthm1}
Let $(N,g)$ be an $(n+l)$-dimensional, smooth (real-analytic), Riemannian manifold, $Z \in N$, and $\mathcal{O}$ be an open neighborhood of $Z$ in $N$.  Let $\mu \in (0,1)$ and $q \geq 3$.  Let $V$ be an $n$-dimensional integral varifold of the form \eqref{V_form} for positive integer multiplicities $\theta_k$ and distinct submanifolds $C^{1,\mu}$ embedded $n$-dimensional submanifold-with-boundary $M_k$ of $\mathcal{O}$ with common boundary $\Gamma$ containing $Z$.  Suppose that the interior of each $M_k$ is a minimal submanifold, \eqref{mainthm_tangents} holds true, and the submanifolds $M_k$ are not all tangent to the same $n$-dimensional plane at $Z$.  Then for some open neighborhood $\mathcal{O}' \subset \mathcal{O}$ of $Z$, $M_k$ are smooth (real-analytic) submanifolds-with-boundary of $\mathcal{O}'$ and $\Gamma$ is a smooth (real-analytic) $(n-1)$-dimensional submanifold of $\mathcal{O}'$. 
\end{theorem}

\begin{remark} {\rm
Observe that when $V$ is of the form \eqref{V_form} and $M_k$ are $C^{1,\mu}$ submanifolds-with-boundary, $V$ is stationary if and only if each $M_k$ is minimal and \eqref{mainthm_tangents} holds true.  To see this, suppose $V$ is stationary.  Then clearly the interior of each $M_k$ is a smooth minimal submanifold.  Moreover, at each point $Y \in \Gamma$, $M_k$ is tangent to a stationary sum of $n$-dimensional half-planes with a common boundary, so \eqref{mainthm_tangents} holds true.  Conversely, if each $M_k$ is a minimal submanifold and \eqref{mainthm_tangents} holds true, by Theorem \ref{mainthm1} each $M_k$ is a smooth submanifold-with-boundary, so by the discussion in Section 2.3, $V$ is stationary. 
} \end{remark}

Our second main result concerns Brakke flows.  Let $I$ be a open interval in $\mathbb{R}$ and $\mathcal{O}$ be a connected open subset of $N$.  A one-parameter family of integral $n$-dimensional varifolds $\{V_t\}_{t \in I}$ is said to be a \textit{Brakke flow} in $\mathcal{O}$ if 
\begin{align} \label{brakkeflowdefn}
	&\int_{\mathcal{O}} \phi(t_1,X) \, d\|V_{t_1}\|(X) - \int_N \phi(t_0,X) \, d\|V_{t_0}\|(X) \\
	&\hspace{10mm} \leq \int_{t_0}^{t_1} \mathcal{B}(V_t,\phi(t,\cdot\,)) \, dt + \int_{t_0}^{t_1} \int_N \frac{\partial \phi}{\partial t}(t,X) \, d\|V_t\|(X) \, dt \nonumber 
\end{align}
for all $t_0 < t_1$ in $I$ and all $\phi \in C^1(I \times \mathcal{O};\mathbb{R}_+)$ such that $\phi(t,\cdot\,)$ has compact support for all $t \in I$, where 
\begin{equation*}
	\mathcal{B}(V_t,\phi) = \int_{\mathcal{O}} (-|H_t|^2 \, \phi + H_t \cdot \nabla \phi) d\|V_t\|
\end{equation*}
for all $\phi \in C^1(\mathcal{O})$ whenever $V_t$ has generalized mean curvature $H_t \in L^2_{\text{loc}}(\|V_t\|)$ and $\mathcal{B}(V_t,\phi) = -\infty$ otherwise.  Note that \eqref{brakkeflowdefn} implies that $V_t$ has generalized mean curvature $H_t \in L^2_{\text{loc}}(\|V_t\|)$ for a.e.~$t \in I$.  We now restate our main regularity result for Brakke flows, Theorem B of the introduction, as follows:

\begin{theorem} \label{mainthm2}
Let $(N,g)$ be an $(n+l)$-dimensional, smooth (real-analytic), Riemannian manifold, $Z \in N$, $\mathcal{O}$ be an open neighborhood of $Z$ in $N$, and $I$ be an open interval in $\mathbb{R}$ containing the origin.  Let $\mu \in (0,1)$ and $q \geq 3$.  Let $\{V_t\}_{t \in I}$ be a one parameter family of $n$-dimensional integral varifolds of the form 
\begin{equation} \label{V_form_mcf}
	V_t = \sum_{k=1}^q \theta_k \, |M_{k,t}|
\end{equation}
for positive integer multiplicities $\theta_k$ and some submanifolds-with-boundary $M_{k,t}$ in $\mathcal{O}$ such that $\{M_{k,t}\}_{t \in I}$ is $C^{2+\mu}_{\rm para}$ and $M_{k,t}$ have a common boundary $\Gamma_t$ with $Z \in \Gamma_0$.  Suppose the interior of $M_{k,t}$ is a smooth mean curvature flow and 
\begin{equation} \label{mainthm2_tangents}
	\sum_{k=1}^q \theta_k \, \eta_{k,t} = 0 \text{ on } \Gamma_t \text{ for all } t \in I, 
\end{equation}
where $\eta_{k,t}$ denote the outward unit conormal vector to the boundary of $M_{k,t}$.  Further suppose $M_{0,k}$ are not all tangent to the same $n$-dimensional plane at $Z$.  Then for some open neighborhoods $I' \subset I$ of $0$ and $\mathcal{O}' \subset \mathcal{O}$ of $Z$, $\{M_{k,t}\}_{t \in I'}$ and $\{\Gamma_t\}_{t \in I'}$ are smooth (second Gevrey) in $I' \times \mathcal{O}'$.  Moreover, for each $t \in I'$, the time-slices $M_{k,t}$ and $\Gamma_t$ are smooth (real analytic) in $\mathcal{O}'$. 
\end{theorem}

\begin{remark} {\rm
Observe that if $\{V_t\}_{t \in I}$ is of the form \eqref{V_form_mcf}, then $\{V_t\}_{t \in I}$ is a Brakke flow if and only if the interior of $\{M_{k,t}\}_{t \in I}$ is a smooth mean curvature flow and \eqref{mainthm2_tangents} holds true.  To see this, suppose $\{V_t\}_{t \in I}$ is a Brakke flow.  Then clear $\{M_{k,t}\}_{t \in I}$ flows by mean curvature.  Moreover, since $V_t$ has generalized mean curvature for a.e.~$t \in I$, \eqref{mainthm2_tangents} holds true.  Conversely, we want to show that if $\{M_{k,t}\}_{t \in I}$ is a smooth mean curvature flow and \eqref{mainthm2_tangents} holds true then 
\begin{equation} \label{mcf_remark_eqn1}
	\frac{d}{dt} \int_{\mathcal{O}} \phi \, d\|V_t\|(X) 
	= \sum_{k=1}^q \int_{M_{k,t}} \left( \frac{\partial \phi}{\partial t} + \nabla \phi \cdot H_{k,t} - \phi \, |H_{k,t}|^2 \right) d\mathcal{H}^n 
\end{equation}
for every $\phi \in C^1(I \times \mathcal{O};\mathbb{R}_+)$ such that $\phi(t,\cdot\,)$ has compact support $K_t$ for all $t \in I$, where $H_{k,t}$ denotes the mean curvature of $M_{k,t}$.  \eqref{brakkeflowdefn} will then follow by integrating \eqref{mcf_remark_eqn1} over $t \in (t_0,t_1)$.  Observe that by Theorem \ref{mainthm2}, $\{M_{k,t}\}_{t \in I}$ is a family of smooth submanifolds-with-boundary.  Thus, by using a partition of unity to localize \eqref{mcf_remark_eqn1} and rescaling, it suffices to prove \eqref{mcf_remark_eqn1} in the special case where $I = (-1,1)$, $\mathcal{O} = B^{n+m}_1(0)$, and there exists a smooth family of embeddings $X_k : I \times (B_1(0) \cap \{x_n \geq 0\}) \rightarrow \mathbb{R}^{n+m}$ such that 
\begin{gather*}
	M_{k,t} = \op{Image}(X_k(t,\cdot\,)) \cap \mathcal{O}, \quad \Gamma_t = \op{Image}(X_k(t,\cdot\,) |_{\{x_n=0\}}) \cap \mathcal{O}, \\
	X_1(t,x) = X_2(t,x) = \cdots = X_q(t,x) \text{ in } B_1(0) \cap \{x_n=0\}
\end{gather*} 
for all $t \in I$.  Define the velocity vector field $\zeta_{k,t} : M_{k,t} \rightarrow \mathbb{R}^{n+m}$ by $\zeta_{k,t}(X_{k,t}(x)) = (\partial X_k/\partial t)(x)$ for all $k = 1,2,\ldots,q$, $t \in I$, and $x \in B_1(0) \cap \{x_n \geq 0\}$ and note that 
\begin{equation} \label{mcf_remark_eqn2}
	\zeta_{1,t}(x',0) = \zeta_{2,t}(x',0) = \cdots = \zeta_{q,t}(x',0) \text{ for all } t \in I, \, x' \in B^{n-1}(0). 
\end{equation}
By the first variational formula for area and the divergence theorem, 
\begin{align*}
	&\frac{d}{dt} \int_{\mathcal{O}} \phi(t,X) \, d\|V_t\|(X) 
	= \sum_{k=1}^q \int_{M_{k,t}} \left( \frac{\partial \phi}{\partial t} + \nabla \phi \cdot \zeta_{k,t} + \phi \op{div}_{M_{k,t}} \zeta_{k,t} \right) \theta_k \, d\mathcal{H}^n 
	\\& \hspace{15mm} = \sum_{k=1}^q \int_{M_{k,t}} \left( \frac{\partial \phi}{\partial t} + \nabla \phi \cdot \zeta_{k,t}^{\perp} 
		+ \op{div}_{M_{k,t}}(\phi \, \zeta_{k,t}) \right) \theta_k \, d\mathcal{H}^n 
	\\& \hspace{15mm} = \sum_{k=1}^q \int_{M_{k,t}} \left( \frac{\partial \phi}{\partial t} + \nabla \phi \cdot \zeta_{k,t}^{\perp} 
		- \phi \, H_{k,t} \cdot \zeta_{k,t} \right) \theta_k \, d\mathcal{H}^n 
		+ \int_{\Gamma_t} \sum_{k=1}^q \phi \, \eta_{k,t} \cdot \zeta_{k,t} \, \theta_k \, d\mathcal{H}^{n-1}, 
\end{align*}
for all $t \in I$, where $\zeta_{k,t}^{\perp}$ denotes the projection of $\zeta_{k,t}$ onto the normal bundle of $M_{k,t}$.  Since $\{M_{k,t}\}_{t \in I}$ flows by mean curvature, $\zeta_{k,t}^{\perp} = H_{k,t}$ on $M_{k,t}$ for all $t \in I$.  By \eqref{mainthm2_tangents} and \eqref{mcf_remark_eqn2}, $\sum_{k=1}^q  \eta_{k,t} \cdot \zeta_{k,t} \, \theta_k = 0$ on $\Gamma_t$ for all $t \in I$.  Therefore \eqref{mcf_remark_eqn1} holds true. 
} \end{remark}

\begin{remark} {\rm 
For parabolic problems such as mean curvature flow, we do not expect solutions to be real-analytic in time-space.  For instance, a well-known example due to Kowalevsky~\cite{Kow75} shows that the solution $u$ to the heat equation $\partial u/\partial t = \partial^2 u/\partial x^2$ on $\mathbb{R}^2$ with initial condition $u = (1-x)^{-1}$ at $t = 0$ is not real analytic in $(t,x)$.  Hence we do not generally expect $\Gamma_t$ and $M_{k,t}$ to be real-analytic in time-space. 
} \end{remark}

\begin{remark} {\rm 
A standard approach to proving regularity for minimal submanifolds and Brakke flows is to first use blow-up methods to establish $C^{1+\mu}$ regularity, see for instance~\cite{DeGiorgiReg}, ~\cite{AllardReg}, and~\cite{BrakkeReg1}.  For minimal submanifolds, starting from $C^{1,\mu}$ regularity, one can use the fact that the minimal surface system is in divergence form to establish $C^{2,\mu}$ regularity and then establish higher regularity via elliptic regularity, see the proof of Theorem \ref{mainthm1} below.  For Brakke flows, going from parabolic $C^{1+\mu}$ regularity to parabolic $C^{2+\mu}$ regularity tends to be more involved, see for instance~\cite{BrakkeReg2}.  In the special case $n = 1$, it was pointed out to us by Tonegawa that the mean curvature flow equation in $\mathbb{R}^{1+m}$ can be written in divergence form as 
\begin{equation*}
	D_t u = D_x (\arctan(D_x u))
\end{equation*}
and thus one can generalize the estimates of~\cite{Parabolic} to parabolic equations in divergence form using the ideas from~\cite[Lemma 9.1]{ADK1} and then apply our approach to show Theorem \ref{mainthm2} holds true when $M_{t,k}$ is parabolic $C^{1+\mu}$ up to $\Gamma$.  We will not address this issue further here; rather, we will simply assume parabolic $C^{2+\mu}$ regularity. 
} \end{remark}

\subsection{Application to branched minimal submanifolds}  An important corollary of Theorem \ref{mainthm1} arises from the work of Hughes of~\cite{Hughes}, which examined the structure of stationary Lipschitz two-valued graphs. 

Let $\mathcal{A}_2(\mathbb{R}^m)$ denote the space of unordered pairs $\{a_1,a_2\}$ for $a_1,a_2 \in \mathbb{R}^m$ not necessarily distinct.  We equip $\mathcal{A}_2(\mathbb{R}^m)$ with a metric 
\begin{equation*}
	\mathcal{G}(\{a_1,a_2\},\{b_1,b_2\}) = \min\{|a_1-b_1|+|a_2-b_2|, |a_1-b_2|+|a_2-b_1|\} 
\end{equation*}
for $a_1,a_2,b_1,b_2 \in \mathbb{R}^m$. 

Let $\Omega$ be an open subset of $\mathbb{R}^n$.  A two-valued function is a map $u : \Omega \rightarrow \mathcal{A}_2(\mathbb{R}^m)$ so that each each $x \in \Omega$, $u(x) = \{u_1(x),u_2(x)\}$ as an unordered pair.  We say a two-valued function $u : \Omega \rightarrow \mathcal{A}_2(\mathbb{R}^m)$ is Lipschitz if 
\begin{equation*}
	\mathcal{G}(u(x),u(y)) \leq L |x-y| \text{ for all } x,y \in \Omega 
\end{equation*}
for some constant $L \in [0,\infty)$.  The graph of a Lipschitz two-valued function $u : \Omega \rightarrow \mathcal{A}_2(\mathbb{R}^m)$ can be regarded as an integral varifold $V$ whose support is the rectifiable set $M$ consisting of all points $(x,u_1(x))$ and $(x,u_2(x))$ such that $x \in \Omega$ and whose multiplicity function $\theta : M \rightarrow \mathbb{Z}_+$ given by $\theta(x,u_1(x)) = \theta(x,u_2(x)) = 1$ if $u_1(x) \neq u_2(x)$ and $\theta(x,u_1(x)) = 2$ if $u_1(x) = u_2(x)$. 

In~\cite{Hughes}, Hughes considered the structure of a stationary graph $V$ of a Lipschitz two-valued function $u : B_2(0) \rightarrow \mathcal{A}_2(\mathbb{R}^m)$ such that $V$ is close to a stationary cone $C$ of one of three types: 
\begin{enumerate}
\item[(i)] $C$ is a sum of two $n$-dimensional planes whose intersection is an $(n-2)$-dimensional subspace, 
\item[(ii)] $C$ is a sum of two $n$-dimensional planes whose intersection is an $(n-1)$-dimensional subspace, and 
\item[(iii)] $C$ is a sum of four $n$-dimensional half-planes meeting along a common $(n-1)$-dimensional boundary axis but $C$ is not the sum of two $n$-dimensional planes intersecting along its axis, i.e.~ $m \geq 2$ and after an orthogonal change of coordinates 
\begin{equation*}
	C = (L_{0,-\theta} + L_{0,\theta} + L_{\phi,\pi-\theta} + L_{\phi,\pi+\theta}) \times \mathbb{R}^{n-1} \times \{0\}^{m-2}
\end{equation*}
for some $\phi \in (0,\pi)$ and $\theta \in (0,\pi/2)$, where 
\begin{equation*}
	L_{\phi,\theta} = |\{ (t \cos \theta, t \cos \phi \sin \theta, t \sin \phi \sin \theta) : t > 0 \}|
\end{equation*}
for each $\phi, \theta \in \mathbb{R}$. 
\end{enumerate}
The main results in each case (i), (ii), and (iii) is covered as Theorems 1, 2, and 3 of~\cite{Hughes} respectively.  In particular,~\cite[Theorem 3]{Hughes} states that in case (iii), the graph of $u$ is the union of four $C^{1,\mu}$ submanifolds-with-boundary $M_k$ meeting along a common boundary $\Gamma$ and each $M_k$ close to one of the four half-planes of $C$.  Observe that this conclusion is not true in cases (i) or (ii) since in codimension $> 1$ the graph of $u$ could be the sum two mutually disjoint, smoothly embedded minimal submanifolds, each close to one of the two planes of $C$.  As a consequence of~\cite[Theorem 3]{Hughes} and our Theorem \ref{mainthm1} above, we have the following: 

\begin{corollary} \label{hughes_cor}
Let $C$ be a minimal cone such that $C$ sum of four $n$-dimensional halfplanes meeting along a common boundary axis $A$ but $C$ is not the sum of two intersecting $n$-dimensional planes.  There exists $\varepsilon = \varepsilon(n,m) > 0$ such that if $V$ is stationary $n$-dimensional integral varifold in $\mathbb{R}^{n+m}$ represented as the graph of a Lipschitz two-valued function $u : B_2(0) \rightarrow \mathcal{A}_2(\mathbb{R}^m)$ such that 
\begin{equation*}
	\int_{B^{n+m}_2(0)} \op{dist}(X,\op{spt} \|C\|)^2 d\|V\|(X) + \int_{B^{n+m}_2(0) \cap \{\op{dist}(X,A) < 1/8\}} \op{dist}(X,\op{spt} \|C\|)^2 d\|V\|(X) < \varepsilon
\end{equation*}
then $V = \sum_{k=1}^4 |M_k|$ in $B_{1/2}(0)$ for some locally real-analytic $n$-dimensional submanifolds-with-boundary $M_k$ meeting along a common locally real-analytic boundary $\Gamma$ in $B_{1/2}(0)$. 
\end{corollary}
\begin{proof}
By~\cite[Theorem 3]{Hughes}, each $M_k$ is a $C^{1,\mu}$ submanifold-with-boundary.  Thus Theorem \ref{mainthm1} applies to conclude $M_k$ and $\Gamma$ are real-analytic. 
\end{proof}

\section{Regularity for minimal submanifolds} \label{sec:elliptic_section}

\subsection{Setup}  \label{sec:elliptic_setup_subsec}  In this section we will prove our main regularity result for minimal submainfolds, Theorem \ref{mainthm1}, by showing that $\Gamma$ and $M_k$ are smooth (real analytic) in some open neighborhood of $Z$.  Without loss of generality we can assume the following.  Recall that $N$ is smoothly embedded $(n+l)$-dimensional submanifold of $\mathbb{R}^{n+m}$.  Assume that $0 \in N$ and $N$ is tangent to $\mathbb{R}^{n+l} \times \{0\}$ at the origin.  Let $\Omega$ be a connected open set in $\mathbb{R}^n$ containing the origin.  Let $\gamma$ be an $(n-1)$-dimensional $C^1$ submanifold of $\Omega$ such that $0 \in \gamma$, $\gamma$ is tangent to $\mathbb{R}^{n-1} \times \{0\}$ at $0$, and $\Omega \setminus \gamma$ has exactly two connected components, $\Omega_+$ and $\Omega_-$, such that $(0,0,\ldots,0,1)$ points into $\Omega_+$ and out of $\Omega_-$ at the origin.  Let $1 \leq s \leq q$ be integers.  Let $M_k = \op{graph} u_k \subset N$ where $u_k \in C^1(\Omega_+ \cup \gamma;\mathbb{R}^m)$ for $k = 1,2,\ldots,s$ and $u_k \in C^1(\Omega_- \cup \gamma;\mathbb{R}^m)$ for $k = s+1,s+2,\ldots,q$.  Assume that 
\begin{equation} \label{coincide}
	u_1 = u_2 = \cdots = u_q 
\end{equation}
on $\gamma$ so that $M_k$ have a common boundary $\Gamma = \op{graph} u_1 |_{\gamma}$.  Since $0 \in \Gamma$ and $\Gamma$ is $C^1$, at the origin we may assume that 
\begin{equation} \label{gradatorigin}
	u_k(0) = 0, \hspace{10mm} D_{x_i} u_k(0) = 0 \text{ for } i = 1,\ldots,n-1,  
\end{equation}
for $k = 1,2,\ldots,q$.  Let $V$ be the $n$-dimensional integral varifold of the form \eqref{V_form} for some positive integers $\theta_k$ and for $M_k$ and $\Gamma$ as above.  We shall let $Z = 0$ and $\mathcal{O} = N \cap (\Omega \times \mathbb{R}^m)$ and assume that $N$, $V$, $M_k$, and $\Gamma$ satisfy the hypotheses of Theorem \ref{mainthm1}, in particular that each $M_k$ is a minimal submanifold for each $k$, \eqref{mainthm_tangents} holds true, and the submanifolds $M_k$ are not all tangent to the same $n$-dimensional plane at the origin. 

For each $X \in N$, let $A_X = (A_X^1,A_X^2,\ldots,A_X^{n+m}) : T_X N \rightarrow (T_X N)^{\perp} \subset \mathbb{R}^{n+m}$ denote the second fundamental form of $N$ at $X$ when $l < m$ and $A_X = 0$ when $l = m$.  Letting $H_k$ denote the mean curvature of $M_k$ as a submanifold of $\mathbb{R}^{n+m}$, $M_k$ being a minimal submanifold is equivalent to 
\begin{equation} \label{mse0} 
	H_k(X) = \sum_{i=1}^n A_X(\tau_i,\tau_i) 
\end{equation}
for all $X \in M_k$ and any orthonormal basis $\tau_1,\ldots,\tau_n$ for $T_X M_k$.  Let $G(p) = (G_{ij}(p))_{i,j=1,\ldots,n}$ be the $n \times n$ matrix given by 
\begin{equation*}
	G_{ij}(p) = \delta_{ij} + \sum_{\lambda=1}^m p_i^{\lambda} p_j^{\lambda} 
\end{equation*}
for $p \in \mathbb{R}^{mn}$ and $i,j = 1,\ldots,n$ and $G(p)^{-1} = (G^{ij}(p))_{i,j=1,\ldots,n}$.  Extend $A_X$ to a symmetric bilinear form on $\mathbb{R}^{n+m}$ such that $A_X(v,w) = 0$ whenever $v$ is normal to $N$.  Extend $A_{(x,z)}$ to a symmetric bilinear form for every $(x,z)$ in a neighborhood of the origin in $\mathbb{R}^{n+m}$ by letting $A_{(x,z)} = A_{(x,z+w)}$ for $(x,z) \in N$ and $w \in \{0\} \times \mathbb{R}^{m-l}$.  Observe that if $l < m$, $N$ is the graph of some smooth (real-analytic) function $f$ over a neighborhood of the origin in $\mathbb{R}^{n+l} \times \{0\}$ and so $A_{(x,z)}$ is well-defined and can be expressed in terms of $f$, $Df$, and $D^2 f$ at $(x,z_1,\ldots,z_l)$.  Define 
\begin{align*}
	\mathscr{H}^{\kappa}(x,z,p) =& \sum_{i,j=1}^n G^{ij}(p) \, A_{(x,z)}^{\kappa}((e_i,p_i),(e_j,p_j)) 
\end{align*}
for $(x,z) \in \mathbb{R}^{n+m}$ near the origin, $p \in \mathbb{R}^{mn}$, and $\kappa = 1,2,\ldots,n+m$, where $e_1,\ldots,e_n$ denotes the standard basis for $\mathbb{R}^n$.  We can rewrite \eqref{mse0} as 
\begin{equation} \label{mse} 
	\sum_{i,j=1}^n D_{x_i} \left( \sqrt{\det G(Du_k)} \, G^{ij}(Du_k) \, D_{x_j} u_k^{\kappa} \right) = \sqrt{\det G(Du_k)} \, \mathscr{H}^{n+\kappa}(x,u_k,Du_k) \end{equation}
on its domain $\Omega_+$ of $k \leq s$ and $\Omega_-$ if $k > s$, where $\kappa = 1,2,\ldots,m$.

Since \eqref{mainthm_tangents} holds true and $M_k$ are not all tangent to the same $n$-dimensional plane at the origin, we may assume that $q \geq 2$ and the unit normals to $M_1$ and $M_2$ are linearly independent at the origin.  After an orthogonal change of coordinates, we may assume that $s \geq 2$, 
\begin{gather}
	D_{x_n} u_1^1(0) > D_{x_n} u_2^1(0), \label{twoplaneassumption} \\
	D_{x_n} u_1^{\kappa}(0) = D_{x_n} u_2^{\kappa}(0) = 0 \text{ for } \kappa = 2,3,\ldots,m.  \label{twoplaneassumption2}
\end{gather}
By \eqref{mainthm_tangents}, $s < q$.

We want to express \eqref{mainthm_tangents} in terms of $u_1,u_2,\ldots,u_q$.  In the special case $l = 1$, this could be accomplished by using the fact that \eqref{mainthm_tangents} is equivalent to 
\begin{equation} \label{tangents_codim1}
	\sum_{k=1}^s \theta_k \nu_k - \sum_{k=s+1}^q \theta_k \nu_k = 0 \text{ on } \Gamma,  
\end{equation}
where $\nu_k$ is the unit normal to $M_k$ with $\nu_k^{n+1} > 0$ (when $l=m=1$, $\nu_k = (-Du_k,1)/\sqrt{1+|Du_k|^2}$).  When $l > 1$, the situation is more complicated.  Suppose that $\gamma = \{ (x',\psi(x',0)) : (x',0) \in S \}$ is the graph of some $C^1$ function $\psi : S \rightarrow \mathbb{R}$, where $S \subset \mathbb{R}^{n-1} \times \{0\}$ is an open neighborhood of the origin and we recall from Section \ref{sec:notation_subsec} that $x' = (x_1,\ldots,x_{n-1})$.  We will determine $\psi$ in Subsection \ref{sec:hodograph_subsec} below.  Then  \eqref{mainthm_tangents} is equivalent to 
\begin{equation} \label{tangents}
	\sum_{k=1}^s \theta_k \frac{(-D_{x'} \psi,1,D_{(-D_{x'} \psi,1)} u_k)}{\sqrt{1+|D_{x'} \psi|^2+|D_{(-D_{x'} \psi,1)} u_k|^2}} 
		- \sum_{k=s+1}^q \theta_k \frac{(-D_{x'} \psi,1,D_{(-D_{x'} \psi,1)} u_k)}{\sqrt{1+|D_{x'} \psi|^2+|D_{(-D_{x'} \psi,1)} u_k|^2}} = 0 
\end{equation}
at each $x = (x',\psi(x',0)) \in \gamma$, where $D\psi$ is evaluated at $(x',0)$.  By \eqref{coincide}, for $x = (x',\psi(x',0)) \in \gamma$ near the origin, $(-D\psi(x',0),1,$ $D_{(-D\psi(x',0),1)} u_k(x))$ lies in the $(m+1)$-dimensional subspace orthogonal to $\Gamma$ at $(x,u_1(x))$.  By the definition of $\gamma$ and \eqref{gradatorigin}, the orthogonal projection of the subspace orthogonal to $\Gamma$ at $(x,u_1(x))$ onto $\{0\} \times \mathbb{R}^{1+m}$ is bijective near the origin.  Thus by taking the $n,n+1,\ldots,n+m$ components of both sides of \eqref{tangents}, \eqref{tangents} is equivalent to 
\begin{equation} \label{tangents2}
	\sum_{k=1}^s \theta_k \frac{1}{\sqrt{1+|D_{x'} \psi|^2+|D_{(-D_{x'} \psi,1)} u_k|^2}} 
		- \sum_{k=s+1}^q \theta_k \frac{1}{\sqrt{1+|D_{x'} \psi|^2+|D_{(-D_{x'} \psi,1)} u_k|^2}} = 0 
\end{equation}
and 
\begin{equation} \label{tangents3}
	\sum_{k=1}^s \theta_k \frac{D_{(-D_{x'} \psi,1)} u_k^{\kappa}}{\sqrt{1+|D_{x'} \psi|^2+|D_{(-D_{x'} \psi,1)} u_k|^2}} 
		- \sum_{k=s+1}^q \theta_k \frac{D_{(-D_{x'} \psi,1)} u_k^{\kappa}}{\sqrt{1+|D_{x'} \psi|^2+|D_{(-D_{x'} \psi,1)} u_k|^2}} = 0 
\end{equation}
for $\kappa = 1,2,\ldots,m$ at each $x = (x',\psi(x',0)) \in \gamma$ near the origin.  By replacing $\Omega$ with a smaller neighborhood of the origin if necessary, assume \eqref{tangents2} and \eqref{tangents3} holds at every $x \in \gamma$. 

\subsection{Partial hodograph transformation}  \label{sec:hodograph_subsec}  Our goal is prove that $u_1,u_2,\ldots,u_q$ are smooth (real-analytic) functions up to the boundary $\gamma$ and $\gamma$ is a smooth (real-analytic) $(n-2)$-dimensional submanifold in $\Omega$.  We will use the partial hodograph transformation of Kinderlehrer, Nirenberg, and Spuck~\cite{KNS}.  Let $w = u_1^1 - u_2^1$.  Consider the transformation 
\begin{equation*}
	y_i = x_i \text{ for } i = 1,\ldots,n-1, \quad y_n = w(x) \quad \text{for } x \in \Omega_+ \cup \gamma.  
\end{equation*}
Let $U$ and $S$ denote the images of $\Omega_+$ and $\gamma$ respectively under this transformation and observe that $S \subseteq \{ y : y_n = 0\}$ by \eqref{coincide}.  By \eqref{twoplaneassumption}, $x \mapsto (x',w(x))$ is invertible near the origin and thus we may assume that $x \mapsto (x',w(x))$ is invertible on $\Omega_+ \cup \gamma$ with inverse transformation given by 
\begin{equation*}
	x_i = y_i \text{ for } i = 1,\ldots,n-1, \quad x_n = \psi(y) \quad \text{for } y \in U \cup S
\end{equation*}
for some function $\psi \in C^{1,\mu}(U \cup S) \cap C^{\infty}(U)$.  We also consider the transformation 
\begin{equation*}
	x_i = y_i \text{ for } i = 1,\ldots,n-1, \quad x_n = \psi(y) - Cy_n \quad \text{for } y \in U \cup S
\end{equation*}
for some constant $C > 0$ such that $D_{y_n} \psi < C$ on $U$.  By replacing $\Omega$ with a smaller open neighborhood of the origin if necessary, we may assume that $y \mapsto (y',\psi(y))$ is a bijection from $U \cup S$ to $\Omega_+ \cup \gamma$ and $y \mapsto (y',\psi(y) - Cy_n)$ is a bijection from $U \cup S$ to $\Omega_- \cup \gamma$.  

Let 
\begin{equation*}
	\phi_k(y) = u_k(y',\psi(y)) \text{ on } U \cup S \text{ for } k \leq s, \quad 
	\phi_k(y) = u_k(y',\psi(y)-Cy_n) \text{ on } U \cup S \text{ for } k < s. 
\end{equation*}
Observe that $\gamma = \{ (y,\psi(y)) : y \in S \}$.  Thus we may assume that \eqref{tangents2} and \eqref{tangents3} hold true with $\psi$ as in the transformation.  Moreover, Theorem \ref{mainthm1} will be proven if we can show that $\psi$ and $\phi_k$ are smooth (real-analytic) up to $S$ near the origin.

By the chain rule, using $x_i = y_i$ for $i = 1,\ldots,n-1$, $x_n = \psi(x)$, and $y_n = w(x)$, for $x \in \Omega_+ \cup \gamma$ we compute that 
\begin{align*}
	&D_{y_i} = D_{x_i} + D_{y_i} \psi \, D_{x_n} \text{ for } i = 1,\ldots,n-1, 
	&&D_{y_n} = D_{y_n} \psi \, D_{x_n}, 
\end{align*}
and so 
\begin{align} \label{transform_deriv}
	&D_{x_i} = D_{y_i} - \frac{D_{y_i} \psi}{D_{y_n} \psi} \, D_{y_n} \text{ for } i = 1,\ldots,n-1, 
	&&D_{x_n} = \frac{1}{D_{y_n} \psi} \, D_{y_n}, \\
	&D_{x_i} w = - \frac{D_{y_i} \psi}{D_{y_n} \psi} \text{ for } i = 1,\ldots,n-1, 
	&&D_{x_n} w  = \frac{1}{D_{y_n} \psi}. \nonumber
\end{align}
Similarly, using $x_i = y_i$ for $i = 1,\ldots,n-1$ and $x_n = \psi(y) - Cy_n$, $x \in \Omega_- \cup \gamma$ we compute that 
\begin{align*}
	&D_{y_i} = D_{x_i} + D_{y_i} \psi \, D_{x_n} \text{ for } i = 1,\ldots,n-1, 
	&&D_{y_n} = (D_{y_n} \psi - C) \, D_{x_n}, 
\end{align*}
and so 
\begin{align} \label{transform_deriv2}
	&D_{x_i} = D_{y_i} - \frac{D_{y_i} \psi}{D_{y_n} \psi - C} \, D_{y_n}
		\text{ for } i = 1,\ldots,n-1, 
	&&D_{x_n} = \frac{1}{D_{y_n} \psi - C} \, D_{y_n}. 
\end{align}

By \eqref{transform_deriv} and \eqref{transform_deriv2}, under the partial hodograph transformation \eqref{mse} transforms to a differential system in $\psi, \phi_1,$ $\phi_2, \phi_3, \ldots, \phi_q$ of the form 
\begin{align} \label{mse_trans} 
	&\sum_{i=1}^n D_{y_i} F_{1,\kappa}^i(D\psi, D\phi_2) + F_{1,\kappa}^0(y, \psi, \phi_2, D\psi, D\phi_2) = 0 \text{ in } U, \\
	&\sum_{i=1}^n D_{y_i} F_{k,\kappa}^i(D\psi, D\phi_k) + F_{k,\kappa}^0(y, \phi_k, D\psi, D\phi_k) = 0 \text{ in } U \text{ if } k = 2,3,\ldots,q, \nonumber 
\end{align}
for $\kappa = 1,2,\ldots,m$ for some smooth (real-analytic) functions $F_{k,\kappa}^i$ for $i = 0,1,2,\ldots,n$, $k = 1,2,\ldots,q$, and $\kappa = 1,2,\ldots,m$.  \eqref{coincide} transforms to 
\begin{gather}
	\phi_2^1 = \phi_3^1 = \cdots = \phi_q^1 \label{coincide_trans} \\
	\phi_1^{\kappa} = \phi_2^{\kappa} = \phi_3^{\kappa} = \cdots = \phi_q^{\kappa} \text{ for } \kappa = 2,3,\ldots,m, \nonumber
\end{gather}
on $S$.  \eqref{tangents2} and \eqref{tangents3} transform to 
\begin{align} \label{tangents_trans}
	\Phi_{\kappa}(D\psi,D\phi_2,D\phi_3,\ldots,D\phi_q) = 0 
\end{align}
on $S$ for $\kappa = 1,2,\ldots,m+1$ for some smooth (real-analytic) functions $\Phi_1,\ldots,\Phi_{m+1}$. 

\subsection{General elliptic systems and the complementing condition}  Consider the general differential system in functions $v_1,v_2\ldots,v_Q$ of the form 
\begin{align} \label{system}
	\sum_{|\alpha| \leq s_k-l} D^{\alpha} F_k^{\alpha}(y,\{D^{\beta} v_j\}_{j=1,\ldots,Q,|\beta| \leq t_j+l}) &= 0 
		\text{ weakly in } U \text{ for } k = 1,2,\ldots,Q \text{ such that } s_k > l, \nonumber \\
	F_k(y,\{D^{\beta} v_j\}_{j=1,\ldots,Q,|\beta| \leq s_k+t_j}) &= 0 
		\text{ in } U \text{ for } k = 1,2,\ldots,Q \text{ such that } s_k \leq l, \nonumber \\
	\Phi_h(y,\{D^{\beta} v_j\}_{j=1,\ldots,Q,|\beta| \leq r_h+t_j}) &= 0 \text{ on } S \text{ for } h = 1,2,\ldots,M, 
\end{align}
where $F_k^{\alpha}$, $F_k$, and $\Phi_h$ are smooth real-valued functions, $l \leq 0$ is an integer, and $s_1,\ldots,s_Q$, $t_1,\ldots,t_Q$, and $r_1,\ldots,r_M$ are integer weights such that $\max_k s_k = 0$, $\min_j t_j \geq -l$, $\min_{k,j} (s_k+t_j) \geq 0$, and $\min_{j,h} (r_h+t_j) \geq 0$.  The linearization of \eqref{system} consists of linear operators of functions $\overline{v}_1,\ldots,\overline{v}_Q$ given by  
\begin{align*}
	\sum_{j=1}^Q \sum_{|\alpha| \leq s_k-l} \sum_{|\beta| \leq t_j+l} D^{\alpha} (a_{kj}^{\alpha \beta}(y) D^{\beta} \overline{v}_j) 
		&= \left. \frac{d}{dt} \sum_{|\alpha| \leq s_k-l} D^{\alpha} F_k^{\alpha}(y,\{D^{\beta} v_j + t D^{\beta} \overline{v}_j\}) \right|_{t=0} 
		\text{ in } U \text{ if } s_k > l, \\
	\sum_{j=1}^Q \sum_{|\beta| \leq s_k+t_j} a_{kj}^{\beta}(y) D^{\beta} \overline{v}_j 
		&= \left. \frac{d}{dt} \, F_k(y,\{D^{\beta} v_j + t D^{\beta} \overline{v}_j \}) \right|_{t=0} 
		\text{ in } U \text{ if } s_k \leq l, \\
	\sum_{j=1}^Q \sum_{|\beta| \leq r_h+t_j} b_{hj}^{\beta}(y) D^{\beta} \overline{v}_j 
		&= \left. \frac{d}{dt} \, \Phi_r(y,\{D^{\beta} v_j + t D^{\beta} \overline{v}_j\}) \right|_{t=0} \text{ on } S, 
\end{align*}
for $k = 1,2,\ldots,Q$ and $h = 1,2,\ldots,M$, where $a_{kj}^{\alpha \beta}$ and $a_{kj}^{\beta}$ are real-valued functions on $U$ and $b_{rj}^{\alpha}$ are real-valued functions on $S$.  Let 
\begin{align*}
	L'_{kj}(y,D) &= \sum_{|\alpha| = s_k-l} \sum_{|\beta| = t_j-l} a_{kj}^{\alpha \beta}(y) D^{\alpha+\beta} 
		\text{ for } y \in U \text{ if } s_k > l, \\
	L'_{kj}(y,D) &= \sum_{|\beta| = s_k+t_j} a_{kj}^{\beta}(y) D^{\beta} 
		\text{ for } y \in U \text{ if } s_k \leq l, \\
	B'_{hj}(y,D) &= \sum_{|\beta| = r_h+t_j} b_{hj}^{\beta}(y) D^{\beta} \text{ for } y \in S, 
\end{align*}
for $j = 1,2,\ldots,Q$, $k = 1,2,\ldots,Q$, and $h = 1,2,\ldots,M$ so that $\sum_{j=1}^Q L'_{kj}(y,D) \, \overline{v}_j$ and $\sum_{j=1}^Q B'_{hj}(y,D) \, \overline{v}_j$ are the principle parts of the linearization of \eqref{system}.  We say \eqref{system} is \textit{elliptic} at $y = y_0$ if the linear system 
\begin{equation*}
	\sum_{j=1}^Q L'_{kj}(y_0,D) \, \overline{v}_j = 0 \text{ in } \mathbb{R}^n \text{ for } k = 1,2,\ldots,Q 
\end{equation*}
has no nontrivial complex-valued solutions of the form $\overline{v}_j = c_j \, e^{i\xi \cdot y}$ for some $\xi \in \mathbb{R}^n \setminus \{0\}$ and $c_j \in \mathbb{C}$ for $j = 1,2,\ldots,Q$.  Assuming \eqref{system} is elliptic at the $y = y_0$, we say \eqref{system} satisfies the \textit{complementing condition} at $y = y_0$ if $\sum_{j=1}^Q (s_j + t_j) = 2M$ and the system 
\begin{align*}
	\sum_{j=1}^Q L'_{kj}(y_0,D) \, \overline{v}_j &= 0 \text{ in } \{y : y_n > 0\} \text{ for } k = 1,2,\ldots,Q, \\
	\sum_{j=1}^Q B'_{hj}(y_0,D) \, \overline{v}_j &= 0 \text{ on } \{y : y_n = 0\} \text{ for } h = 1,2,\ldots,M, 
\end{align*}
has no nontrivial, complex-valued solutions $\overline{v}_j(y',y_n) = e^{i\xi' \cdot y'} \, \overline{v}_j(0,y_n)$ that are exponentially decaying as $y_n \rightarrow +\infty$ for some $\xi' \in \mathbb{R}^{n-1}$.

\subsection{Checking ellipticity and the complementing condition}  \label{sec:checking_mse_subsec}  Now consider the differential system in $\psi$ and $\phi_k^{\kappa}$ with $(k,\kappa) \neq (1,1)$ given by \eqref{mse_trans}, \eqref{tangents_trans}, and 
\begin{gather}
	\phi_k^1 - \phi_2^1 = 0 \text{ for } k = 3,4,\ldots,q, \label{coincide_trans2} \\
	\phi_k^{\kappa} - \phi_1^{\kappa} = 0 \text{ for } k = 2,3,4,\ldots,q, \, \kappa = 2,3,\ldots,m, \nonumber
\end{gather} 
on $S$ with weights $l = 1$, $s = 0$ for the equations of \eqref{mse_trans}, $t = 2$ for the functions $\psi$ and $\phi_k^{\kappa}$ with $(k,\kappa) \neq (1,1)$, $r = -1$ for the equations of \eqref{tangents_trans}, and $r = -2$ for the equations of \eqref{coincide_trans2}.  In order to apply elliptic regularity to prove Theorem \ref{mainthm1}, we must show this differential system is elliptic and satisfies the complementing condition at the origin. 

Let $a_k = D_{x_n} u_k(0)$ for $k = 1,2,\ldots,q$.  By \eqref{gradatorigin}, \eqref{twoplaneassumption}, \eqref{twoplaneassumption2}, and \eqref{transform_deriv}, $D_{y_i} \psi(0) = 0$ for $i = 1,2,\ldots,n-1$ and $|a_1-a_2| = 1/D_{y_n} \psi(0)$, which together with \eqref{gradatorigin}, \eqref{transform_deriv}, and \eqref{transform_deriv2} yields 
\begin{gather} \label{gradatorigin_trans}
	D_{y_i} \psi(0) = D_{y_i} \phi_2(0) = D_{y_i} \phi_3(0) = \cdots = D_{y_i} \phi_q(0) = 0 \text{ for } i = 1,2,\ldots,n-1, \\
	|a_1 - a_2| = \frac{1}{D_{y_n} \psi(0)}, \hspace{4mm} 
	a_k = \frac{D_{y_n} \phi_k(0)}{D_{y_n} \psi(0)} \text{ if } k  \leq s, \hspace{4mm} 
	a_k = \frac{D_{y_n} \phi_k(0)}{D_{y_n} \psi(0) - C} \text{ if } k > s.  \nonumber 
\end{gather}

We want to linearize and take the principle part of \eqref{mse_trans} at the origin.  Consider the equation for $k = 2$ in \eqref{mse_trans}.  We can rewrite the minimal surface equation for $u_2$ from \eqref{mse} as 
\begin{align*}
	&\sum_{i,j=1}^n G^{ij}(Du_2) \, D_{x_i x_j} u_2^{\kappa} 
	+ \sum_{i,j,k,l=1}^n \sum_{\lambda=1}^m (G^{ij}(Du_2) \, G^{kl}(x,u_2,Du_2) - G^{ik}(Du_2) \, G^{lj}(Du_2) \\& - G^{il}(Du_2) \, G^{kj}(Du_2)) 
		\, D_{x_j} u_2^{\kappa} \, D_{x_k} u_2^{\lambda} \, D_{x_i x_l} u_2^{\lambda} = \mathscr{H}^{n+\kappa}(x,u_2,Du_2) \text{ in } \Omega_+, 
\end{align*}
using the fact that $u_2 \in C^{\infty}(\Omega_+)$.  By \eqref{gradatorigin}, $G_{ii}(Du_2(0)) = 1$ for $i = 1,2,\ldots,n-1$, $G_{nn}(Du_2(0)) = 1+|a_2|^2$, and $G_{ij}(Du_2(0)) = 0$ for $i \neq j$.  Thus linearizing and taking the principle part of the equation for $k = 2$ in \eqref{mse_trans} yields 
\begin{equation*}
	\sum_{\lambda=1}^m \left( \delta_{\kappa \lambda} - \frac{a_2^{\kappa} \, a_2^{\lambda}}{1+|a_2|^2} \right) 
	\left( \sum_{i=1}^{n-1} \overline{D_{x_i x_i} u_2^{\lambda}} + \frac{1}{1+|a_2|^2} \, \overline{D_{x_n x_n} u_2^{\lambda}} \right) = 0 
	\text{ on } \{ y : y_n > 0 \} 
\end{equation*}
for $\kappa = 1,2,\ldots,m$, where for $i = 1,2,\ldots,n$ we let $\overline{D_{x_i x_i} u_2^{\lambda}}$ denote the result of rewriting $D_{x_i x_i} u_2^{\lambda}$ as a function of $y$ and then computing its linearization and second order principle part at the origin.  Since the matrix $(\delta_{\kappa \lambda} - a_2^{\kappa} \, a_2^{\lambda}/(1+|a_2|^2))_{\kappa,\lambda=1,2,\ldots,m}$ is invertible (because it has eigenvalue $1/(1+|a_2|^2)$ with eigenvector $a_2$ and eigenvalue $1$ with multiplicity $m-1$), 
\begin{equation} \label{mse_lin0}
	\sum_{i=1}^{n-1} \overline{D_{x_i x_i} u_2} + \frac{1}{1+|a_2|^2} \, \overline{D_{x_n x_n} u_2} = 0 \text{ on } \{ y : y_n > 0 \}
\end{equation}
for $\kappa = 1,2,\ldots,m$.  By \eqref{transform_deriv} and \eqref{gradatorigin_trans}, 
\begin{align} \label{u_lin}
	\overline{D_{x_i x_i} u_2} &= \overline{\left( D_{y_i} - \frac{D_{y_i} \psi}{D_{y_n} \psi} \, D_{y_n} \right)^2 \phi_2} 
		= \overline{D_{y_i} \left( D_{y_i} - \frac{D_{y_i} \psi}{D_{y_n} \psi} \, D_{y_n} \right) \phi_2} \\
		&= D_{y_i y_i} \overline{\phi}_2 - D_{y_i y_i} \overline{\psi} \, \frac{D_{y_n} \phi_2(0)}{D_{y_n} \psi(0)} 
		= D_{y_i y_i} (\overline{\phi}_2 - \overline{\psi} \, a_2) \nonumber \\
	\overline{D_{x_n x_n} u_2} &= \overline{\left( \frac{1}{D_{y_n} \psi} \, D_{y_n} \right)^2 \phi_2}
		= \frac{1}{D_{y_n} \psi(0)^2} \, D_{y_n y_n} \overline{\phi}_2 - D_{y_n y_n} \overline{\psi} \, \frac{D_{y_n} \phi_2(0)}{D_{y_n} \psi(0)^3} \nonumber \\
		&= |a_1-a_2|^2 \, D_{y_n y_n} (\overline{\phi}_2 - \overline{\psi} \, a_2) \nonumber 
\end{align}
for functions $\overline{\psi}$ and $\overline{\phi}_2$, which substituting into \eqref{mse_lin0} yields  
\begin{equation*}
	(1+|a_2|^2) \sum_{i=1}^{n-1} D_{y_i y_i} (\overline{\phi}_2 - \overline{\psi} a_2) + |a_1-a_2|^2 D_{y_n y_n} (\overline{\phi}_2 - \overline{\psi} a_2) = 0 
	\text{ in } \{y : y_n > 0\}. 
\end{equation*}
By similar computations, we can linearize and take the principle part of the equations in \eqref{mse_trans} for every $k \in \{1,2,\ldots,q\}$ using \eqref{gradatorigin}, \eqref{transform_deriv}, \eqref{transform_deriv2}, and \eqref{gradatorigin_trans} to obtain the differential system in $\overline{\psi}, \overline{\phi}_2, \overline{\phi}_3, \ldots, \overline{\phi}_q$ of 
\begin{gather}
	(1+|a_1|^2) \sum_{i=1}^{n-1} D_{y_i y_i} (\overline{\phi}_2^1 - a_1^1 \, \overline{\psi}) + |a_1-a_2|^2 \, D_{y_n y_n} (\overline{\phi}_2^1 - a_1^1 \, \overline{\psi}) = 0, \label{mse_lin} \\
	(1+|a_k|^2) \sum_{i=1}^{n-1} D_{y_i y_i} (\overline{\phi}_k^{\kappa} - a_k^{\kappa} \, \overline{\psi}) 
		+ |a_1-a_2|^2 \, D_{y_n y_n} (\overline{\phi}_k^{\kappa} - a_k^{\kappa} \, \overline{\psi}) = 0 \text{ if } (k,\kappa) \neq (1,1), \, k \leq s, \nonumber \\
	(1+|a_k|^2) \sum_{i=1}^{n-1} D_{y_i y_i} (\overline{\phi}_k^{\kappa} - a_k^{\kappa} \, \overline{\psi}) 
		+ \frac{|a_1 - a_2|^2}{(1 - C \, |a_1 - a_2|)^2} \, D_{y_n y_n} (\overline{\phi}_k^{\kappa} - a_k^{\kappa} \, \overline{\psi}) = 0 \text{ if } k > s, \nonumber 
\end{gather}
in $\{y : y_n > 0\}$.  \eqref{mse_lin} is obviously an elliptic system in $\overline{\phi}_2^1 - a_1^1 \, \overline{\psi}$ and $\overline{\phi}_k^{\kappa} - a_k^{\kappa} \, \overline{\psi}$ for $(k,\kappa) \neq (1,1)$. 

To check the complementing condition, it suffices to consider solutions to \eqref{mse_lin} of the form 
\begin{equation} \label{barpsi_eqn}
	\overline{\phi}_2^1 - a_1^1 \, \overline{\psi} = c_1^1 \, e^{i\xi' \cdot y' - \lambda_1^1 y_n}, \hspace{10mm} 
	\overline{\phi}_k^{\kappa} - a_k^{\kappa} \, \overline{\psi} = c_k^{\kappa} \, e^{i\xi' \cdot y' - \lambda_k^{\kappa} y_n} \text{ for } (k,\kappa) \neq (1,1), 
\end{equation}
where $\xi' \in \mathbb{R}^{n-1}$, $c_k^{\kappa} \in \mathbb{C}$, and $\lambda_k^{\kappa} > 0$.  It is readily computed that  
\begin{align} \label{nu_eqn}
	\lambda_k^{\kappa} &= \frac{\sqrt{1+|a_k|^2} \, |\xi'|}{|a_1-a_2|} \text{ if } k \leq s, \\
	\lambda_k^{\kappa} &= \frac{(C \, |a_1 - a_2| - 1) \, \sqrt{1+|a_k|^2} \, |\xi'|}{|a_1-a_2|} \text{ if } k > s. \nonumber 
\end{align}
(note that $C \, |a_1-a_2| - 1 > 0$ since $C > D_{y_n} \psi(0) = 1/|a_1-a_2|$).  Since $\lambda_k^{\kappa} > 0$, $\xi' \neq 0$. 

The linearization of \eqref{coincide_trans} simply yields
\begin{gather} 
	\overline{\phi}_2^1 = \overline{\phi}_3^1 = \overline{\phi}_4^1 = \cdots = \overline{\phi}_q^1, \label{coincide_lin} \\
	\overline{\phi}_1^{\kappa} = \overline{\phi}_2^{\kappa} = \overline{\phi}_3^{\kappa} = \cdots = \overline{\phi}_q^{\kappa} 
		\text{ for } \kappa = 2,3,\ldots,m, \nonumber 
\end{gather}
on $\{ y : y_n = 0 \}$.  By \eqref{barpsi_eqn}, $\overline{\phi}_2^1 - a_k^1 \overline{\psi} = c_k^1 e^{i\xi' \cdot y'}$ on $\{ y : y_n = 0 \}$ for $k = 1,2$, so solving for $\overline{\psi}$ and $\overline{\phi}_2^1$, 
\begin{equation} \label{barpsi_and_barphi12}
	\overline{\psi} = \frac{c_2^1 - c_1^1}{a_1^1 - a_2^1} \, e^{i\xi' \cdot y'}, \quad 
	\overline{\phi}_2^1 = \frac{a_1^1 \, c_2^1 - a_2^1 \, c_1^1}{a_1^1 - a_2^1} \, e^{i\xi' \cdot y'}, 
\end{equation}
on $\{ y : y_n = 0 \}$.  Hence by \eqref{barpsi_eqn},  \eqref{coincide_lin}, and \eqref{barpsi_and_barphi12}, 
\begin{align*}
	c_k^1 \, e^{i\xi' \cdot y'} 
	&= \overline{\phi}_2^1 - a_k^1 \, \overline{\psi} 
	= \frac{(a_1^1 \, c_2^1 - a_2^1 \, c_1^1) - a_k^1 \, (c_2^1 - c_1^1)}{a_1^1 - a_2^1} \, e^{i\xi' \cdot y'}, \\
	c_k^{\kappa} \, e^{i\xi' \cdot y'} 
	&= \overline{\phi}_1^{\kappa} - a_k^{\kappa} \, \overline{\psi} 
	= \left( c_1^{\kappa} - \frac{a_k^{\kappa} \, (c_2^1 - c_1^1)}{a_1^1 - a_2^1} \right)  e^{i\xi' \cdot y'} \text{ for } \kappa = 2,3,\ldots,m, 
\end{align*}
on $\{ y : y_n = 0 \}$.  Cancelling $e^{i\xi' \cdot y'}$ and using $a_1^{\kappa} = a_2^{\kappa} = 0$ for $\kappa = 2,3,\ldots,m$ and $a_1^1 - a_2^1 = |a_1-a_2|$ by \eqref{twoplaneassumption} and \eqref{twoplaneassumption2}, 
\begin{equation} \label{coincide_lin2}
	c_k = c_1^1 \, \frac{a_k - a_2}{|a_1 - a_2|} + c_2^1 \, \frac{a_1 - a_k}{|a_1 - a_2|} + \widehat{c}_1 \text{ for } k = 1,2,3,\ldots,q, 
\end{equation} 
where $\widehat{c}_1 = (c_1^2,c_1^3,\ldots,c_1^m)$. 

Next we want to linearize and take the principle part of \eqref{tangents_trans} at the origin.  Let $\overline{D_{x_i} u_k^{\kappa}}$ denote the result of rewriting $D_{x_i} u_k^{\kappa}$ as a function of $y$ and then computing the first order principle part of its linearization at the origin.  Linearizing and taking the principle part of \eqref{tangents2} and \eqref{tangents3} using \eqref{gradatorigin_trans} yields 
\begin{gather} 
	\sum_{k=1}^s \theta_k \, \frac{a_k \cdot \overline{D_{x_n} u_k}}{(1+|a_k|^2)^{3/2}} 
		- \sum_{k=s+1}^q \theta_k \, \frac{a_k \cdot \overline{D_{x_n} u_k}}{(1+|a_k|^2)^{3/2}} = 0, \label{tangents_lin0} \\
	\sum_{k=1}^s \theta_k \left( \frac{\overline{D_{x_n} u_k}}{(1+|a_k|^2)^{1/2}} - \frac{a_k \cdot \overline{D_{x_n} u_k} \, a_k}{(1+|a_k|^2)^{3/2}} \right)
		- \sum_{k=s+1}^q \theta_k \left( \frac{\overline{D_{x_n} u_k}}{(1+|a_k|^2)^{1/2}} - \frac{a_k \cdot \overline{D_{x_n} u_k} \, a_k}{(1+|a_k|^2)^{3/2}} 
		\right) = 0, \nonumber
\end{gather}
on $\{ y : y_n = 0 \}$.  By \eqref{transform_deriv}, \eqref{transform_deriv2}, and \eqref{gradatorigin_trans}, 
\begin{align*} 
	\overline{D_{x_n} u_1^1} 
		&= \frac{1}{D_{y_n} \psi(0)} D_{y_n} \overline{\phi}_2^1 - \frac{1+D_{y_n} \phi_2^1(0)}{D_{y_n} \psi(0)^2} D_{y_n} \overline{\psi} 
		= |a_1-a_2| D_{y_n} (\overline{\phi}_2^1 - a_1^1 \overline{\psi}), \\
	\overline{D_{x_n} u_k^{\kappa}} 
		&= \frac{1}{D_{y_n} \psi(0)} D_{y_n} \overline{\phi}_k^{\kappa} - \frac{D_{y_n} \phi_k^{\kappa}(0)}{D_{y_n} \psi(0)^2} D_{y_n} \overline{\psi} 
		= |a_1-a_2| D_{y_n} (\overline{\phi}_k^{\kappa} - a_k^{\kappa} \overline{\psi}) \text{ if } k \leq s, \, (k,\kappa) \neq (1,1), \nonumber \\
	\overline{D_{x_n} u_k^{\kappa}} 
		&= \frac{1}{D_{y_n} \psi(0)-C} D_{y_n} \overline{\phi}_k^{\kappa} - \frac{D_{y_n} \phi_k^{\kappa}(0)}{(D_{y_n} \psi(0)-C)^2} D_{y_n} \overline{\psi} 
		= \frac{|a_1-a_2|}{1 - C |a_1-a_2|} D_{y_n} (\overline{\phi}_k^{\kappa} - a_k^{\kappa} \overline{\psi}) \text{ if } k > s. \nonumber
\end{align*}
which substituting into \eqref{tangents_lin0} yields 
\begin{align} \label{tangents_lin}
	&\theta_1 \, \frac{a_1^1 D_{y_n} (\overline{\phi}_2^1 - a_1^1 \, \overline{\psi})}{(1+a_1^2)^{3/2}} 
		+ \sum_{k \leq s, \, (k,\kappa) \neq (1,1)} \theta_k \, \frac{a_k^{\kappa} D_{y_n} (\overline{\phi}_k^{\kappa} - a_k^{\kappa} \, \overline{\psi})
		}{(1+a_k^2)^{3/2}} \\&\hspace{10mm} - \sum_{k=s+1}^q \theta_k \frac{a_k \cdot D_{y_n} (\overline{\phi}_k - a_k \, \overline{\psi})
		}{(1 - C |a_1-a_2|) (1+a_k^2)^{3/2}} = 0, \nonumber \\
	&\theta_1 \, \frac{D_{y_n} (\overline{\phi}_2^1 - a_1^1 \, \overline{\psi}) \, e_1}{(1+|a_1|^2)^{3/2}} 
		+ \sum_{k \leq s, \, (k,\kappa) \neq (1,1)} \theta_k \left( \frac{D_{y_n} (\overline{\phi}_k^{\kappa} - a_k^{\kappa} \, \overline{\psi}) \, e_{\kappa}
		}{(1+|a_k|^2)^{1/2}} - \frac{a_k^{\kappa} D_{y_n} (\overline{\phi}_k^{\kappa} - a_k^{\kappa} \, \overline{\psi}) \, a_k}{(1+|a_k|^2)^{3/2}} \right) \nonumber \\
		&\hspace{10mm} - \sum_{k=s+1}^q \theta_k \left( \frac{D_{y_n} (\overline{\phi}_k - \overline{\psi} \, a_k)}{(1 - C |a_1-a_2|) (1+|a_k|^2)^{1/2}} 
		- \frac{a_k \cdot D_{y_n} (\overline{\phi}_k - \overline{\psi} \, a_k) \, a_k}{(1 - C |a_1-a_2|) (1+|a_k|^2)^{3/2}} \right) = 0, \nonumber 
\end{align}
on $\{ y : y_n = 0 \}$, where $e_1,e_2,\ldots,e_m$ is the standard basis for $\mathbb{R}^m$.  By \eqref{barpsi_eqn} and \eqref{nu_eqn}, 
\begin{align*} 
	&D_{y_n} (\overline{\phi}_2^1 - a_1^1 \, \overline{\psi}) = -\lambda_1^1 \, c_1^1 \, e^{i\xi' \cdot y'} 
		= \frac{\sqrt{1+|a_1|^2} \, |\xi'|}{|a_1-a_2|} \, c_1^1 \, e^{i\xi' \cdot y'}, \\
	&D_{y_n} (\overline{\phi}_k^{\kappa} - a_k^{\kappa} \, \overline{\psi}) = -\lambda_k^{\kappa} \, c_k^{\kappa} \, e^{i\xi' \cdot y'} 
		= \frac{\sqrt{1+|a_k|^2} \, |\xi'|}{|a_1-a_2|} \, c_k^{\kappa} \, e^{i\xi' \cdot y'} \text{ if } (k,\kappa) \neq (1,1), \, k \leq s, \\
	&D_{y_n} (\overline{\phi}_k^{\kappa} - a_k^{\kappa} \, \overline{\psi}) = -\lambda_k^{\kappa} \, c_k^{\kappa} \, e^{i\xi' \cdot y'} 
		= \frac{(C |a_1 - a_2| - 1) \, \sqrt{1+|a_k|^2} \, |\xi'|}{|a_1-a_2|} \, c_k^{\kappa} e^{i\xi' \cdot y'} \text{ if } k > s, 
\end{align*}
on $\{ y : y_n = 0 \}$, which when substituted into \eqref{tangents_lin} yields  
\begin{align}
	 \label{tangents_lin2a} &\sum_{k=1}^q \theta_k \frac{a_k \cdot c_k}{1+|a_k|^2} = 0, \\
	\label{tangents_lin2b} &\sum_{k=1}^q \theta_k \left( c_k - \frac{a_k \cdot c_k \, a_k}{1+|a_k|^2} \right) = 0. 
\end{align}
Let $\widehat{a}_k = (a_k^2,a_k^3,\ldots,a_k^m)$ for $k = 1,2,\ldots,q$ and recall that $\widehat{c}_1 = (c_1^2,c_1^3,\ldots,c_1^m)$.  By substituting \eqref{coincide_lin2} into \eqref{tangents_lin2a}, 
\begin{equation} \label{tangents_lin3a}
	\sum_{k=1}^q \frac{\theta_k}{1+|a_k|^2} \left( \frac{a_k \cdot (a_k-a_2)}{|a_1-a_2|} \, c_1^1 + \frac{a_k \cdot (a_1-a_k)}{|a_1-a_2|} \, c_2^1 
		+ \widehat{a}_k \cdot \widehat{c}_1 \right) = 0.
\end{equation}
By breaking up \eqref{tangents_lin2b} into its first and remaining components 
\begin{equation}
	\sum_{k=1}^q \theta_k \, \frac{(1+|a_k|^2) \, c_k^1 - a_k \cdot c_k \, a_k^1}{1+|a_k|^2} = 0, \quad 
	\sum_{k=1}^q \theta_k \, \frac{(1+|a_k|^2) \, \widehat{c}_k - a_k \cdot c_k \, \widehat{a}_k^1}{1+|a_k|^2} = 0 
\end{equation}
and substituting \eqref{coincide_lin2} using $\widehat{a}_1 = \widehat{a}_2 = 0$ by \eqref{twoplaneassumption2} and 
\begin{equation*}
	(1+|a_k|^2) \, (a_k-a_i) - a_k \cdot (a_k-a_i) \, a_k = a_k - a_i - |\widehat{a}_k|^2 \, a_i + a_i \cdot a_k \, \widehat{a}_k \text{ for } i = 1,2, 
\end{equation*}
we obtain 
\begin{align} \label{tangents_lin3b}
	&\sum_{k=1}^q \frac{\theta_k}{1+|a_k|^2} \left( \frac{a_k^1 - a_2^1 - |\widehat{a}_k|^2 a_2^1}{|a_1-a_2|} \, c_1^1 
		+ \frac{a_1^1 - a_k^1 + |\widehat{a}_k|^2 a_1^1}{|a_1-a_2|} \, c_2^1 - a_k^1 \, \widehat{a}_k \cdot \widehat{c}_1 \right) = 0, \\
	&\sum_{k=1}^q \frac{\theta_k \, ((1 + a_2 \cdot a_k) \, c_1^1 - (1 + a_1 \cdot a_k) \, c_2^1) \, \widehat{a}_k}{(1+|a_k|^2) \, |a_1-a_2|} 
		+ \sum_{k=1}^q \theta_k \left( \widehat{c}_1 - \frac{\widehat{a}_k \cdot \widehat{c}_1 \, \widehat{a}_k}{1+|a_k|^2} \right) = 0. \nonumber 
\end{align}
In order to solve \eqref{tangents_lin3a} and \eqref{tangents_lin3b}, after an orthogonal change of coordinates of $\mathbb{R}^m$ we may suppose that $c_1^{\kappa} = 0$ for $\kappa = 3,4,\ldots,m$.  Then \eqref{tangents_lin3a} and \eqref{tangents_lin3b} imply that $c_1^1, c_2^1, c_1^2$ satisfy 
\begin{align} \label{tangents_lin4}
	&\sum_{k=1}^q \frac{\theta_k}{1+|a_k|^2} \left( \frac{a_k \cdot (a_k-a_2)}{|a_1-a_2|} \, c_1^1 + \frac{a_k \cdot (a_1-a_k)}{|a_1-a_2|} \, c_2^1 
		+ a_k^2 \, c_1^2 \right) = 0, \\
	&\sum_{k=1}^q \frac{\theta_k}{1+|a_k|^2} \left( \frac{a_k^1 - a_2^1 - |\widehat{a}_k|^2 a_2^1}{|a_1-a_2|} \, c_1^1 
		+ \frac{a_1^1 - a_k^1 + |\widehat{a}_k|^2 a_1^1}{|a_1-a_2|} \, c_2^1 - a_k^1 \, a_k^2 \, c_1^2 \right) = 0, \nonumber \\
	&\sum_{k=1}^q \frac{\theta_k \, ((1 + a_2 \cdot a_k) \, c_1^1 - (1 + a_1 \cdot a_k) \, c_2^1) \, a_k^2}{(1+|a_k|^2) \, |a_1-a_2|} 
		+ \sum_{k=1}^q \theta_k \left( 1 - \frac{(a_k^2)^2}{1+|a_k|^2} \right) \, c_1^2 = 0 \nonumber
\end{align}
To simplify notation, $\widetilde{\theta}_k = \theta_k/(1+|a_k|^2)$ and $\widetilde{a}_k = (a_k^3,a_k^4,\ldots,a_k^m)$ for $k = 1,2,\ldots,q$.  We compute the determinant $D$ of the linear system of \eqref{tangents_lin4}  by first using elementary row operations (add column 1 to column 2, then add $a_2^1 \cdot$ column 2 to column 1) to simply: 
\begin{align*}
	D &= \frac{1}{|a_1-a_2|^2} \left| \begin{array}{ccc} 
		\sum \widetilde{\theta}_k \, a_k \cdot (a_k-a_2) & \sum \widetilde{\theta}_k \, a_k \cdot (a_1-a_k) & \sum \widetilde{\theta}_k \, a_k^2 \\ 
		\sum \widetilde{\theta}_k (a_k^1-a_2^1-|\widehat{a}_k|^2 \, a_2^1) & \sum \widetilde{\theta}_k (a_1^1-a_k^1+|\widehat{a}_k|^2 \, a_1^1) 
			& -\sum \widetilde{\theta}_k \, a_k^1 \, a_k^2 \\ 
		\sum \widetilde{\theta}_k (1+a_2 \cdot a_k) \, a_k^2 & -\sum \widetilde{\theta}_k (1+a_1 \cdot a_k) \, a_k^2 
			& \sum \widetilde{\theta}_k (1+(a_k^1)^2 + |\widetilde{a}_k|^2)  
	\end{array} \right| \\
	&= \frac{1}{|a_1-a_2|} \left| \begin{array}{ccc} 
		\sum \widetilde{\theta}_k \, a_k \cdot (a_k-a_2) & \sum \widetilde{\theta}_k \, a_k^1 & \sum \widetilde{\theta}_k \, a_k^2 \\ 
		\sum \widetilde{\theta}_k (a_k^1-a_2^1-|\widehat{a}_k|^2 \, a_2^1) & \sum \widetilde{\theta}_k (1+|\widehat{a}_k|^2) 
			& -\sum \widetilde{\theta}_k \, a_k^1 \, a_k^2 \\ 
		\sum \widetilde{\theta}_k (1+a_2 \cdot a_k) \, a_k^2 & -\sum \widetilde{\theta}_k \, a_k^1 \, a_k^2 
			& \sum \widetilde{\theta}_k (1+(a_k^1)^2 + |\widetilde{a}_k|^2)  
	\end{array} \right| \\
	&= \frac{1}{|a_1-a_2|} \left| \begin{array}{ccc} 
		\sum \widetilde{\theta}_k \, |a_k|^2 & \sum \widetilde{\theta}_k \, a_k^1 & \sum \widetilde{\theta}_k \, a_k^2 \\ 
		\sum \widetilde{\theta}_k \, a_k^1 & \sum \widetilde{\theta}_k (1+|\widehat{a}_k|^2) & -\sum \widetilde{\theta}_k \, a_k^1 \, a_k^2 \\ 
		\sum \widetilde{\theta}_k \, a_k^2 & -\sum \widetilde{\theta}_k \, a_k^1 \, a_k^2 & \sum \widetilde{\theta}_k (1+(a_k^1)^2 + |\widetilde{a}_k|^2)  
	\end{array} \right| \\
	&= \frac{1}{|a_1-a_2|} \left| \begin{array}{ccc} 
		\sum \widetilde{\theta}_k \, ((a_k^1)^2+(a_k^2)^2+|\widetilde{a}_k|^2) & \sum \widetilde{\theta}_k \, a_k^1 & \sum \widetilde{\theta}_k \, a_k^2 \\ 
		\sum \widetilde{\theta}_k \, a_k^1 & \sum \widetilde{\theta}_k (1+(a_k^2)^2+|\widetilde{a}_k|^2) & -\sum \widetilde{\theta}_k \, a_k^1 \, a_k^2 \\ 
		\sum \widetilde{\theta}_k \, a_k^2 & -\sum \widetilde{\theta}_k \, a_k^1 \, a_k^2 & \sum \widetilde{\theta}_k (1+(a_k^1)^2 + |\widetilde{a}_k|^2)  
	\end{array} \right| .
\end{align*}
Expanding the $3 \times 3$ determinate,
\begin{align} \label{D_eqn1}
	D &= \frac{1}{|a_1-a_2|} \left( \sum \widetilde{\theta}_k \, ((a_k^1)^2+(a_k^2)^2+|\widetilde{a}_k|^2) 
		\cdot \sum \widetilde{\theta}_k (1+(a_k^2)^2+|\widetilde{a}_k|^2) \cdot \sum \widetilde{\theta}_k (1+(a_k^1)^2 + |\widetilde{a}_k|^2) \right. \\&
		\hspace{5mm} - 2 \sum \widetilde{\theta}_k \, a_k^1 \cdot \sum \widetilde{\theta}_k \, a_k^2 \cdot \sum \widetilde{\theta}_k \, a_k^1 \, a_k^2 
		- \sum \widetilde{\theta}_k (1+(a_k^1)^2 + |\widetilde{a}_k|^2) \cdot \left( \sum \widetilde{\theta}_k \, a_k^1 \right)^2 \nonumber \\& \hspace{5mm} \left. 
		- \sum \widetilde{\theta}_k (1+(a_k^2)^2 + |\widetilde{a}_k|^2) \cdot \left( \sum \widetilde{\theta}_k \, a_k^2 \right)^2 
		- \sum \widetilde{\theta}_k \, |a_k|^2 \cdot \left( \sum \widetilde{\theta}_k \, a_k^1 \, a_k^2 \right)^2 \right) \nonumber
\end{align}
Let 
\begin{equation*}
	S = \sum \widetilde{\theta}_k \, ((a_k^1)^2+(a_k^2)^2+|\widetilde{a}_k|^2) 
		\cdot \sum \widetilde{\theta}_k \, (1+(a_k^2)^2+|\widetilde{a}_k|^2) \cdot \sum \widetilde{\theta}_k \, (1+(a_k^1)^2 + |\widetilde{a}_k|^2). 
\end{equation*}
By expanding $S$ while grouping terms with factors $\left( \sum_k \widetilde{\theta}_k \right)^p$ for similar powers of $p$, 
\begin{align*}
	S &= \left( \sum \widetilde{\theta}_k \right)^2 \left( \widetilde{\theta}_k \, (a_k^1)^2 + \widetilde{\theta}_k \, (a_k^2)^2 
		+ \widetilde{\theta}_k \, |\widetilde{a}_k|^2 \right) + \sum \widetilde{\theta}_k \cdot \sum \widetilde{\theta}_k \, (a_k^1)^2 \cdot 
		\sum \widetilde{\theta}_k ((a_k^1)^2 + |\widetilde{a}_k|^2) 
		\\&\hspace{10mm} + \sum \widetilde{\theta}_k \cdot \sum \widetilde{\theta}_k \, (a_k^2)^2 \cdot \sum \widetilde{\theta}_k ((a_k^2)^2 + |\widetilde{a}_k|^2) 
		\\&\hspace{10mm} + 2 \sum \widetilde{\theta}_k \cdot \sum \widetilde{\theta}_k ((a_k^1)^2 + |\widetilde{a}_k|^2) \cdot 
		\sum \widetilde{\theta}_k ((a_k^2)^2 + |\widetilde{a}_k|^2) \\&\hspace{14mm} + \sum \widetilde{\theta}_k \, |a_k|^2 
		\cdot \sum \widetilde{\theta}_k \, ((a_k^1)^2 + |\widetilde{a}_k|^2) \cdot \sum \widetilde{\theta}_k \, ((a_k^2)^2 + |\widetilde{a}_k|^2). 
\end{align*}
Then using 
\begin{equation*}
	\sum_k \widetilde{\theta}_k \, ((a_k^1)^2 + |\widetilde{a}_k|^2) \cdot \sum_k \widetilde{\theta}_k \, ((a_k^2)^2 + |\widetilde{a}_k|^2) 
	= \sum_k \widetilde{\theta}_k \, (a_k^1)^2 \cdot \sum_k \widetilde{\theta}_k \, (a_k^2)^2 
	+ \sum_k \widetilde{\theta}_k \, |\widetilde{a}_k|^2 \cdot \sum_k \widetilde{\theta}_k \, |a_k|^2, 
\end{equation*}
we obtain 
\begin{align*}
	S &= \left( \sum \widetilde{\theta}_k \right)^2 \left( \widetilde{\theta}_k \, (a_k^1)^2 + \widetilde{\theta}_k \, (a_k^2)^2 
		+ \widetilde{\theta}_k \, |\widetilde{a}_k|^2 \right) + \sum \widetilde{\theta}_k \cdot \sum \widetilde{\theta}_k \, (a_k^1)^2 \cdot 
		\sum \widetilde{\theta}_k ((a_k^1)^2 + |\widetilde{a}_k|^2) 
		\\&\hspace{5mm} + \sum \widetilde{\theta}_k \cdot \sum \widetilde{\theta}_k \, (a_k^2)^2 \cdot \sum \widetilde{\theta}_k ((a_k^2)^2 + |\widetilde{a}_k|^2) 
		+ 2 \sum \widetilde{\theta}_k \cdot \sum \widetilde{\theta}_k \, (a_k^1)^2 \cdot \sum \widetilde{\theta}_k \, (a_k^2)^2 
		\\&\hspace{5mm} + 2 \sum \widetilde{\theta}_k \cdot \sum \widetilde{\theta}_k \, |\widetilde{a}_k|^2 \cdot \sum \widetilde{\theta}_k |a_k|^2  
		+ \sum \widetilde{\theta}_k \, |a_k|^2 \cdot \sum \widetilde{\theta}_k \, (a_k^1)^2 \cdot \sum \widetilde{\theta}_k \, (a_k^2)^2 
		\\&\hspace{5mm} + \sum \widetilde{\theta}_k \, |\widetilde{a}_k|^2 \cdot \left( \sum \widetilde{\theta}_k \, |a_k|^2 \right)^2. 
\end{align*}
Finally, by regrouping terms, 
\begin{align} \label{D_eqn2}
	S &= \sum \widetilde{\theta}_k \cdot \sum \widetilde{\theta}_k \, (a_k^1)^2 \cdot \sum \widetilde{\theta}_k (1 + (a_k^1)^2 + |\widetilde{a}_k|^2) 
		\\&\hspace{5mm} + \sum \widetilde{\theta}_k \cdot \sum \widetilde{\theta}_k \, (a_k^2)^2 
		\cdot \sum \widetilde{\theta}_k (1 + (a_k^2)^2 + |\widetilde{a}_k|^2) 
		+ 2 \sum \widetilde{\theta}_k \cdot \sum \widetilde{\theta}_k \, (a_k^1)^2 \cdot \sum \widetilde{\theta}_k \, (a_k^2)^2 \nonumber 
		\\&\hspace{5mm} + \sum \widetilde{\theta}_k \, |a_k|^2 \cdot \sum \widetilde{\theta}_k \, (a_k^1)^2 \cdot \sum \widetilde{\theta}_k \, (a_k^2)^2 
		+ \sum \widetilde{\theta}_k \, |\widetilde{a}_k|^2 \cdot \left( \sum \widetilde{\theta}_k (1+|a_k|^2) \right)^2. \nonumber
\end{align}
Hence, by substituting \eqref{D_eqn2} into \eqref{D_eqn1} and regrouping terms, 
\begin{align*}
	D &= \frac{1}{|a_1-a_2|} \left( \sum \widetilde{\theta}_k \, (1+(a_k^1)^2+|\widetilde{a}_k|^2) \cdot \left( \sum \widetilde{\theta}_k 
		\cdot \sum \widetilde{\theta}_k \, (a_k^1)^2 - \left( \sum \widetilde{\theta}_k \, a_k^1 \right)^2 \right) \right. \\&\hspace{5mm}
		+ \sum \widetilde{\theta}_k \, (1+(a_k^2)^2+|\widetilde{a}_k|^2) \cdot \left( \sum \widetilde{\theta}_k 
		\cdot \sum \widetilde{\theta}_k \, (a_k^2)^2 - \left( \sum \widetilde{\theta}_k \, a_k^2 \right)^2 \right) \\&\hspace{5mm}
		+ \sum \widetilde{\theta}_k \, |a_k|^2 \left( \sum \widetilde{\theta}_k \, (a_k^1)^2 
		\cdot \sum \widetilde{\theta}_k \, (a_k^2)^2 - \left( \sum \widetilde{\theta}_k \, a_k^1 \, a_k^2 \right)^2 \right) \\&\hspace{5mm}
		+ 2 \left( \sum \widetilde{\theta}_k \cdot \sum \widetilde{\theta}_k \, (a_k^1)^2 \cdot \sum \widetilde{\theta}_k \, (a_k^2)^2
		- \sum\widetilde{\theta}_k \, a_k^1 \cdot \sum \widetilde{\theta}_k \, a_k^2 \cdot \sum \widetilde{\theta}_k \, a_k^1 \, a_k^2 
		\right) \\&\hspace{5mm} \left. + \sum \widetilde{\theta}_k \, |\widetilde{a}_k|^2 \cdot \left( \sum \widetilde{\theta}_k \, (1+|a_k|^2) 
		\right)^2 \right), 
\end{align*}
By Cauchy-Schwartz $D \geq 0$ with $D = 0$ if and only if $a_1^1 = a_2^1 = \cdots a_q^1$, $a_1^2 = a_2^2 = \cdots a_q^2$, and $\widetilde{a}_1 = \widetilde{a}_2 = \cdots = \widetilde{a}_q = 0$.  Since $a_1 = a_2 = \cdots = a_q$ would contradict the assumption that $M_k$ are not all tangent to the same plane at the origin, $D > 0$.  Hence \eqref{tangents_lin4} implies that $c_1^1 = c_2^1 = c_1^2 = 0$.  Thus $c_k^{\kappa} = 0$ for all $k = 1,2,\ldots,q$ and $\kappa = 1,2,\ldots,m$.  Therefore the system \eqref{mse_lin}, \eqref{coincide_lin}, and \eqref{tangents_lin} satisfies the complementing condition in $\overline{\phi}_2^1 - a_1 \, \overline{\psi}$ and $\overline{\phi}_k^{\kappa} - a_k^{\kappa} \, \overline{\psi}$ for $(k,\kappa) \neq (1,1)$.  Equivalently, the system \eqref{mse_lin}, \eqref{coincide_lin}, and \eqref{tangents_lin} satisfies the complementing condition in $\overline{\psi}$ and $\overline{\phi}_k^{\kappa}$ for $(k,\kappa) \neq (1,1)$.  Consequently the differential system given by \eqref{mse_trans}, \eqref{tangents_trans}, and \eqref{coincide_trans2} is elliptic and satisfies the complementing condition at the origin. 

\begin{proof}[Proof of Theorem \ref{mainthm1}]
Recall that $\psi$ and $\phi_k^{\kappa}$ with $(k,\kappa) \neq (1,1)$ are $C^{1,\mu}$ on $U \cup S$ and solve a system of the form \eqref{system} that is elliptic and satisfies the complementing condition near the origin.  As was pointed out in~\cite{KNS}, we can establish a Schauder estimate for linear systems of the form \eqref{system} that is analogous to~\cite[Lemma 9.1]{ADK1} using a similar proof and ideas from~\cite{ADK2} and then apply this Schauder estimate in a standard difference quotient argument to show that $\psi, \phi_1, \phi_2,\ldots, \phi_q$ are $C^{2,\mu}$ functions in a relatively open neighborhood of the origin in $U \cup S$.  By Theorem 6.8.2 of~\cite{Morrey}, if $(N,g)$ is a smooth (real-analytic) Riemannian manifold then $\psi, \phi_1, \phi_2,\ldots, \phi_q$ are smooth (real-analytic) functions near the origin in $U \cup S$.  It follows that $u_1,\ldots,u_s$ are smooth (real-analytic) on $\Omega_+ \cup \gamma$ near the origin, $u_{s+1},u_{s+2},\ldots,u_q$ are smooth (real-analytic) on $\Omega_- \cup \gamma$ near the origin, and $\Gamma = \{(y',\psi(y',0),\phi_2(y',0)) : y' = (y_1,\ldots,y_{n-1}) \in S \}$ is a smooth (real-analytic) $(n-1)$-dimensional submanifold near the origin. 
\end{proof}

\section{Regularity for mean curvature flow} \label{sec:parabolic_section}

\subsection{Setup}  In this section we will prove our main regularity result for mean curvature flow, Theorem \ref{mainthm2}.  Let $N$, $\{V_t\}_{t \in I}$, $\{M_{k,t}\}_{t \in I}$, and $\{\Gamma_t\}_{t \in I}$ be as in the statement of Theorem \ref{mainthm2}.  We want to use a setup similar to Subsection \ref{sec:elliptic_setup_subsec} to represent the submanifolds $M_{k,t}$ in the statement of Theorem \ref{mainthm2} as graphs of solutions $u_{k,t}$ of the mean curvature flow system; however, we need to modify this setup slightly since $M_{k,t}$ and $\Gamma_t$ are evolving with time.  Recall that $N$ is smoothly embedded $(n+l)$-dimensional submanifold of $\mathbb{R}^{n+m}$.  Assume that $0 \in N$ and $N$ is tangent to $\mathbb{R}^{n+l} \times \{0\}$ at the origin.  Let $\Omega$ be a connected open set in $\mathbb{R}^n$ containing the origin.  Let $\{\gamma_t\}_{t \in I}$ be a $C^{2+\mu}_{\rm para}$ one-parameter family of $(n-1)$-dimensional submanifolds of $\Omega$ such that $0 \in \gamma_0$, $\gamma_0$ is tangent to $\mathbb{R}^{n-1} \times \{0\}$ at $0$, and for each $t \in I$ the open set $\Omega \setminus \gamma_t$ has exactly two connected components, $\Omega_{t,+}$ and $\Omega_{t,-}$.  Assume that $\Omega_{t,+}$ and $\Omega_{t,-}$ are continuous in time (as a family of $C^2$ domains of $\Omega$) and $(0,0,\ldots,0,1)$ points into $\Omega_{0,+}$ and out of $\Omega_{0,-}$ at the origin.  Let $2 \leq s < q$ be integers.  Represent each $M_{k,t}$ as the graph of a function $u_k(t,\cdot\,)$ on a domain $\Omega_{t,+}$ for $k = 1,2,\ldots,s$ and on a domain $\Omega_{t,-}$ for $k = s+1,\ldots,q$ such that $u_k$ is in $C^{2+\mu}_{\rm para}$ on its domain in time-space.  We will assume that $u_k = u_k(t,\cdot\,)$ satisfies \eqref{coincide} for all $t \in I$ so that $M_{k,t}$ have a common boundary $\Gamma_t = \op{graph} u_1(t,\cdot\,) |_{\gamma_t}$.  Letting $Z = 0$ and $\mathcal{O} = N \cap (\Omega \times \mathbb{R}^m)$, we will assume that $N$, $V_t$, $M_{k,t}$, and $\Gamma_t$ satisfy the hypotheses of Theorem \ref{mainthm2}.  In particular, we will assume that $M_{k,t}$ flow by mean curvature in $N$, which is equivalent to 
\begin{equation} \label{mcf0} 
	\left(\frac{\partial X}{\partial t}\right)^{\perp} = H_{k,t}(X) - \sum_{i=1}^n A_X(\tau_i,\tau_i) 
\end{equation}
on $M_{k,t}$ for all $t \in I$, where $(\cdot)^{\perp}$ denotes the orthogonal projection onto the orthogonal complement of $T_X M_{k,t}$ in $\mathbb{R}^{n+m}$, $H_{k,t}$ is the mean curvature of $M_{k,t}$ as a submanifold of $\mathbb{R}^{n+m}$, and $A_X$ is the second fundamental form of $N$ from Subsection \ref{sec:elliptic_setup_subsec}.  By rewriting \eqref{mcf0} using $M_{k,t} = \op{graph} u_k(t,\cdot\,)$ and taking the inner product of both sides of \eqref{mcf0} with $(-Du^{\kappa},e_{\kappa})$, which is normal to $M_{k,t}$, we find that \eqref{mcf0} is equivalent to $u_k$ satisfying the minimal curvature flow system, 
\begin{equation} \label{mcf} 
	D_t u_k^{\kappa} = \sum_{i,j=1}^n G^{ij}(Du_k) D_{x_i x_j} u_k^{\kappa} 
		+ \sum_{i=1}^n \mathscr{H}^i(x,u_k,Du_k) D_{x_i} u_k^{\kappa} - \mathscr{H}^{n+\kappa}(x,u_k,Du_k) 
\end{equation}
on $\Omega_{t,+}$ for $k = 1,2,\ldots,s$ and on $\Omega_{t,-}$ for $k = s+1,\ldots,q$ for all $t \in I$ and $\kappa = 1,2,\ldots,m$, where $G(p) = (G_{ij}(p))_{i,j=1,\ldots,n}$, $G(p)^{-1} = (G^{ij}(p))_{i,j=1,\ldots,n}$, and $\mathscr{H}^{\kappa}(x,z,p)$ are as defined in Subsection \ref{sec:elliptic_setup_subsec}.  We will assume that $M_k = M_{k,t}$ satisfy \eqref{mainthm_tangents} for all $t \in I$ and the submanifolds $M_{0,k}$ are not all tangent to the same $n$-dimensional plane at the origin.  In place of \eqref{gradatorigin}, \eqref{twoplaneassumption}, and \eqref{twoplaneassumption2} we will assume that 
\begin{gather}
	u_k(0,0) = 0, \hspace{10mm} D_{x_i} u_k(0,0) = 0 \text{ for } i = 1,\ldots,n-1 \hspace{10mm} \text{for all } k = 1,2,\ldots,q, \label{gradatorigin_mcf} \\
	D_{x_n} u_1^1(0,0) > D_{x_n} u_2^1(0,0), \label{twoplaneassumption_mcf} \\
	D_{x_n} u_1^{\kappa}(0,0) = D_{x_n} u_2^{\kappa}(0,0) = 0 \text{ for } \kappa = 2,3,\ldots,m. \label{twoplaneassumption2_mcf} 
\end{gather}
We will assume that $u_k = u_k(t,\cdot\,)$ satisfies \eqref{tangents2} and \eqref{tangents3} for all $t \in I$ for a function $\psi$ which will be determined in Subsection \ref{sec:hodograph_mcf_subsec} below.

\subsection{Partial hodograph transformation}  \label{sec:hodograph_mcf_subsec}  We define a partial hodograph transformation similar to the one in Section \ref{sec:elliptic_section} by 
\begin{equation*}
	\tau = t, \quad y_i = x_i \text{ for } i = 1,\ldots,n-1, \quad y_n = w(t,x) = u^1_1(t,x) - u^1_2(t,x)
\end{equation*}
mapping $\Omega_{t,+}$ and $\gamma_t$ into $\{ y \in \mathbb{R}^n : y_n > 0 \}$ and $\{ y \in \mathbb{R}^n : y_n = 0 \}$ respectively.  For some $\varepsilon > 0$, the image of $\Omega_{t,+}$ and $\gamma_t$ contain $U = \{ y \in B_{\varepsilon}(0) : y_n > 0 \}$ and $S = \{ y \in B_{\varepsilon}(0) : y_n = 0 \}$ respectively for all $t \in (-\varepsilon,\varepsilon) \subseteq I$.  Reducing $I$ if necessary, assume $I = (-\varepsilon,\varepsilon)$.  Let 
\begin{equation*}
	t = \tau, \quad x_i = y_i \text{ for } i = 1,\ldots,n-1, \quad x_n = \psi(\tau,y) \quad \text{for } \tau \in I, \, y \in U \cup S 
\end{equation*}
denote the inverse transformation of $\tau = t$, $y_i = x_i$ for $i = 1,\ldots,n-1$, and $y_n = w(t,x)$, which by \eqref{twoplaneassumption_mcf} exists provided $\varepsilon$ is sufficiently small.  We also define the transformation 
\begin{equation*}
	t = \tau, \quad x_i = y_i \text{ for } i = 1,\ldots,n-1, \quad x_n = \psi(\tau,y) - Cy_n 
\end{equation*}
mapping $U$ into $\Omega_{\tau,-}$ for all $\tau \in I$, where $C > 0$ is a constant such that $D_{y_n} \psi < C$ on $I \times U$ and we assume that $\varepsilon$ is sufficiently small.  

Let 
\begin{align*}
	\phi_k(\tau,y) &= u_k(\tau,y',\psi(\tau,y)) \text{ on } I \times (U \cup S) \text{ for } k = 1,2,3,\ldots,s, \\
	\phi_k(\tau,y) &= u_k(\tau,y',\psi(\tau,y)-Cy_n) \text{ on } I \times (U \cup S) \text{ for } k = s+1,\ldots,q. 
\end{align*}
Since $\Gamma$ is the graph of $\psi$ over $S$, \eqref{mainthm2_tangents} implies that $u = u(t,\cdot\,)$ and $\psi = \psi(t,\cdot\,)$ satisfy \eqref{tangents2} and \eqref{tangents3} for all $t \in I$.  Moreover, Theorem \ref{mainthm2} will be proven if we can show that $\psi, \phi_1,\phi_2, \ldots,\phi_q$ are smooth (second Gervey with real analytic time slices).  

We computing as in Subsection \ref{sec:hodograph_subsec}, 
\begin{align} \label{paratransform_deriv}
	&D_t = D_{\tau} - \frac{D_{\tau} \psi}{D_{y_n} \psi} D_{y_n}, 
	&&D_{x_i} = D_{y_i} - \frac{D_{y_i} \psi}{D_{y_n} \psi} D_{y_n} \text{ for } i < n, 
	&&D_{x_n} = \frac{1}{D_{y_n} \psi} D_{y_n}, \\
	&D_t w = - \frac{D_{\tau} \psi}{D_{y_n} \psi}, 
	&&D_{x_i} w = - \frac{D_{y_i} \psi}{D_{y_n} \psi} \text{ for } i < n, 
	&&D_{x_n} w = \frac{1}{D_{y_n} \psi}, \nonumber 
\end{align}
if $t \in I$ and $x \in \Omega_+ \cup \gamma$ and 
\begin{align} \label{paratransform_deriv2}
	&D_t = D_{\tau} - \frac{D_{\tau} \psi}{D_{y_n} \psi - C} D_{y_n}, 
	&&D_{x_i} = D_{y_i} - \frac{D_{y_i} \psi}{D_{y_n} \psi - C} D_{y_n} \text{ for } i < n, 
	&&D_{x_n} = \frac{1}{D_{y_n} \psi - C} D_{y_n}, 
\end{align}
if $t \in I$ and $x \in \Omega_- \cup \gamma$. 

Under the partial hodograph transformation, \eqref{mcf} transforms to a differential system in $\psi, \phi_1,$ $\phi_2, \ldots, \phi_q$ of the form 
\begin{align} \label{mcf_trans} 
	&F_{1,\kappa}(y, \phi_2, D_{(\tau,y)} \psi, D_{(\tau,y)} \phi_2, D_y^2 \psi, D_y^2 \phi_2) = 0, \\
	&F_{k,\kappa}(y, \phi_k, D_{(\tau,y)} \psi, D_{(\tau,y)} \phi_k, D_y^2 \psi, D_y^2 \phi_k) = 0 \text{ for } k = 2,3,\ldots,q, \nonumber 
\end{align}
in $I \times U$ for $\kappa = 1,2,\ldots,m$ for some smooth (real-analytic) functions $F_{k,\kappa}$ for $k = 1,2,\ldots,q$ and $\kappa = 1,2,\ldots,m$.  \eqref{coincide} on $\bigcup_{t \in I} \{t\} \times \gamma_t$ transforms to \eqref{coincide_trans} on $I \times S$.  \eqref{tangents2} and \eqref{tangents3} on $\bigcup_{t \in I} \{t\} \times \gamma_t$ transform to \eqref{tangents_trans} on $I \times S$. 

\subsection{General parabolic systems and the complementing condition} Consider the general differential system in functions $v_1,v_2\ldots,v_Q$ of the form 
\begin{align} \label{parasystem}
	F_k(\tau,y,\{D_{\tau}^{\alpha} D_y^{\beta} v_j\}_{j=1,\ldots,Q,2b\alpha+|\beta| \leq s_k+t_j}) &= 0 
		\text{ in } I \times U \text{ for } k = 1,2,\ldots,Q, \\ 
	\Phi_h(\tau,y,\{D_{\tau}^{\alpha} D_y^{\beta} v_j\}_{j=1,\ldots,Q,2b\alpha+|\beta| \leq r_h+t_j}) &= 0 \text{ on } I \times S \text{ for } h = 1,2,\ldots,M, \nonumber 
\end{align}
where $F_k$ and $\Phi_h$ are smooth real-valued functions, $b \geq 1$ is an integer, and $s_1,\ldots,s_Q$, $t_1,\ldots,t_Q$, and $r_1,\ldots,r_M$ are integer weights such that $\max_k s_k = 0$, $\min_{k,j} (s_k+t_j) \geq 0$, and $\min_{j,h} (r_h+t_j) \geq 0$.  The linearization of \eqref{parasystem} consists of linear operators of functions $\overline{v}_1,\ldots,\overline{v}_Q$ given by  
\begin{align*}
	\sum_{j=1}^Q \sum_{2b\alpha+|\beta| \leq s_k+t_j} a_{kj}^{\alpha \beta}(\tau,y) D_{\tau}^{\alpha} D_y^{\beta} \overline{v}_j 
		&= \left. \frac{d}{dt} \, F_k(\tau,y,\{D_{\tau}^{\alpha} D_y^{\beta} v_j + t D_{\tau}^{\alpha} D_y^{\beta} \overline{v}_j\}) \right|_{t=0} \text{ in } I \times U, \\
	\sum_{j=1}^Q \sum_{2b\alpha+|\beta| \leq 2b+r_h} b_{hj}^{\alpha \beta}(\tau,y) D_{\tau}^{\alpha} D_y^{\beta} \overline{v}_j 
		&= \left. \frac{d}{dt} \, \Phi_r(\tau,y,\{D_y^{\beta} v_j + t D_y^{\beta} \overline{v}_j\}) \right|_{t=0} \text{ on } I \times S, 
\end{align*}
for $k = 1,2,\ldots,Q$ and $h = 1,2,\ldots,M$, where $a_{kj}^{\alpha \beta}$ are real-valued functions on $U$ and $b_{hj}^{\alpha \beta}$ are real-valued functions on $S$.  Let 
\begin{align*}
	L'_{kj}(\tau,y,D_{\tau},D_y) &= \sum_{2b\alpha+|\beta| = s_k+t_j} a_{kj}^{\alpha \beta}(\tau,y) D_{\tau}^{\alpha} D_y^{\beta} 
		\text{ for } (\tau,y) \in I \times U, \\
	B'_{hj}(\tau,y,D_{\tau},D_y) &= \sum_{2b\alpha+|\beta| = 2b+r_h} b_{hj}^{\alpha \beta}(\tau,y) D_{\tau}^{\alpha} D_y^{\beta} 
		\text{ for } (\tau,y) \in I \times S, 
\end{align*}
for $j = 1,2,\ldots,Q$, $k = 1,2,\ldots,Q$, and $h = 1,2,\ldots,M$ so that $\sum_{j=1}^Q L'_{kj}(\tau,y,D_{\tau},D_y) \, \overline{v}_j$ and $\sum_{j=1}^Q B'_{hj}(\tau,y,D_{\tau},D_y) \, \overline{v}_j$ are the principle parts of the linearization of \eqref{parasystem}.  Following~\cite{Parabolic}, in particular Chapter 1 and Theorem 2.1, we define parabolic systems and the corresponding complementing condition as follows.  We say \eqref{parasystem} is \textit{parabolic} at $(\tau,y) = (\tau_0,y_0)$ if there exists $\delta > 0$ such that the linear system 
\begin{equation*}
	\sum_{j=1}^Q L'_{kj}(\tau_0,y_0,D_{\tau},D_y) \, \overline{v}_j = 0 \text{ in } \mathbb{R} \times \mathbb{R}^n \text{ for } k = 1,2,\ldots,Q 
\end{equation*}
has no nontrivial complex-valued solutions of the form $\overline{v}_j = c_j \, e^{\rho \tau + i\xi \cdot y}$ for some $\xi \in \mathbb{R}^n$ and $\rho \in \mathbb{C}$ with $\op{Re}(\rho) \leq -\delta |\xi|^{2b}$ and $c_j \in \mathbb{C}$ for $j = 1,2,\ldots,Q$.  It follows from the definition of a parabolic system that $\sum_{j=1}^Q (s_j+t_j) = 2bR$ for some integer $R \geq 1$.  Assuming \eqref{parasystem} is parabolic at the $(\tau,y) = (\tau_0,y_0)$, we say \eqref{parasystem} satisfies the \textit{complementing condition} at $(\tau,y) = (\tau_0,y_0)$ if $M = bR$ and there exists $\delta_1 \in (0,\delta)$ such that the system 
\begin{align*}
	\sum_{j=1}^Q L'_{kj}(\tau_0,y_0,D_{\tau},D_y) \, \overline{v}_j &= 0 \text{ in } \mathbb{R} \times \{y : y_n > 0\} \text{ for } k = 1,2,\ldots,Q, \\
	\sum_{j=1}^Q B'_{hj}(\tau_0,y_0,D_{\tau},D_y) \, \overline{v}_j &= 0 \text{ on } \mathbb{R} \times \{y : y_n = 0\} \text{ for } h = 1,2,\ldots,bR, 
\end{align*}
has no nontrivial, complex-valued solutions $\overline{v}_j(\tau,y',y_n) = e^{\rho \tau + i\xi' \cdot y'} \, \overline{v}_j(0,0,y_n)$ that is exponentially decaying as $y_n \rightarrow +\infty$ for some $\rho \in \mathbb{C}$ and $\xi' \in \mathbb{R}^{n-1}$ with $(\rho,\xi') \neq (0,0)$ and $\op{Re}(\rho) \geq -\delta_1 |\xi'|^{2b}$. 

\subsection{Checking parabolicity and the complementing condition}  Now consider the differential system in $\psi$ and $\phi_k^{\kappa}$ with $(k,\kappa) \neq (1,1)$ given by \eqref{mcf_trans} on $I \times U$, \eqref{tangents_trans} on $I \times S$, and \eqref{coincide_trans2} on $I \times S$ with weights $b = 1$, $s = 0$ for the equations of \eqref{tangents_trans}, $t = 2$ for the functions $\psi$ and $\phi_k^{\kappa}$ with $(k,\kappa) \neq (1,1)$, $r = -1$ for the equations of \eqref{tangents_trans}, and $r = -2$ for the equations of \eqref{coincide_trans2}.  In order to apply parabolic regularity to prove Theorem \ref{mainthm2}, we must show that this differential system is parabolic and satisfies the complementing condition at the origin. 

Let $a_k = D_{x_n} u_k(0)$ for $k = 1,2,\ldots,q$.  Recall that \eqref{gradatorigin_trans} holds true with $0$ denoting the origin in time-space.  First, we want to linearize and take the principle part of \eqref{mcf_trans} at the origin.  Consider the equation for $k = 2$ in \eqref{mcf_trans}.  Recall that by \eqref{gradatorigin}, $G_{ii}(Du_2(0)) = 1$ for $i = 1,2,\ldots,n-1$, $G_{nn}(Du_2(0)) = 1+|a_2|^2$, and $G_{ij}(Du_2(0)) = 0$ for $i \neq j$.  Thus linearizing and taking the principle part of \eqref{mcf} yields 
\begin{equation} \label{mcf_lin0}
	\overline{D_t u_k^{\kappa}} = \sum_{i=1}^{n-1} \overline{D_{x_i x_i} u_k^{\kappa}} + \frac{1}{1+|a_2|^2} \, \overline{D_{x_n x_n} u_k^{\kappa}} = 0 
	\text{ on } \mathbb{R} \times \{ y : y_n > 0 \} 
\end{equation}
for $k = 1,2,\ldots,q$ and $\kappa = 1,2,\ldots,m$, where $\overline{D_t u_k^{\kappa}}$ and $\overline{D_{x_i x_i} u_k^{\kappa}}$ for $i = 1,2,\ldots,n$ we let  denote the result of rewriting $D_t u_k^{\kappa}$ and $D_{x_i x_i} u_k^{\kappa}$ respectively as functions of $y$ and then computing their linearization and principle part at the origin.  Using \eqref{paratransform_deriv}, \eqref{paratransform_deriv2}, and \eqref{gradatorigin_trans}, we compute that 
\begin{equation*}
	\overline{D_t u_2} = D_{\tau} \overline{\phi}_2 - D_{\tau} \overline{\psi} \, \frac{D_{y_n} \phi_2(0)}{D_{y_n} \psi(0)} = D_t (\overline{\phi}_2 - a_2 \, \overline{\psi})
\end{equation*}
and similarly compute $\overline{D_t u_k}$ for $k \neq 2$ and we compute $\overline{D_{x_i x_i} u_k}$ like in Subsection \ref{sec:checking_mse_subsec}, see \eqref{u_lin}.  Substituting into \eqref{mcf_lin0}, 
\begin{gather}
	D_t (\overline{\phi}_2^1 - a_1^1 \, \overline{\psi}) = \sum_{i=1}^{n-1} D_{y_i y_i} (\overline{\phi}_2^1 - a_1^1 \, \overline{\psi}) 
		+ \frac{|a_1-a_2|^2}{1+|a_1|^2} \, D_{y_n y_n} (\overline{\phi}_2^1 - a_1^1 \, \overline{\psi}), \label{mcf_lin} \\
	D_t (\overline{\phi}_k^{\kappa} - a_k^{\kappa} \, \overline{\psi}) = \sum_{i=1}^{n-1} D_{y_i y_i} (\overline{\phi}_k^{\kappa} - a_k^{\kappa} \, \overline{\psi}) 
		+ \frac{|a_1-a_2|^2}{1+|a_k|^2} \, D_{y_n y_n} (\overline{\phi}_k^{\kappa} - a_k^{\kappa} \, \overline{\psi}) 
		\text{ if } (k,\kappa) \neq (1,1), \, k \leq s, \nonumber \\
	D_t (\overline{\phi}_k^{\kappa} - a_k^{\kappa} \, \overline{\psi}) = \sum_{i=1}^{n-1} D_{y_i y_i} (\overline{\phi}_k^{\kappa} - a_k^{\kappa} \, \overline{\psi}) 
		+ \frac{|a_1 - a_2|^2}{(1 - C |a_1 - a_2|)^2 (1+|a_k|^2)} \, D_{y_n y_n} (\overline{\phi}_k^{\kappa} - a_k^{\kappa} \, \overline{\psi}) 
		\text{ if } k > s, \nonumber 
\end{gather}
in $\mathbb{R} \times \{y : y_n > 0\}$.  \eqref{mcf_lin} is obviously an parabolic system in $\overline{\phi}_2^1 - a_1^1 \, \overline{\psi}$ and $\overline{\phi}_k^{\kappa} - a_k^{\kappa} \, \overline{\psi}$ for $(k,\kappa) \neq (1,1)$.  In particular, if $\overline{\phi}_2^1 - a_1^1 \, \overline{\psi} = c_1^1 \, e^{\rho \tau + i\xi \cdot y}$ and $\overline{\phi}_k^{\kappa} - a_k^{\kappa} \, \overline{\psi} = c_k^{\kappa} \, e^{\rho \tau + i\xi \cdot y}$, where $(k,\kappa) \neq (1,1)$, solve \eqref{mcf_lin} for $\xi \in \mathbb{R}^n$, $\rho \in \mathbb{C}$, and $c_k^{\kappa} \in \mathbb{C}$ not all zero, then 
\begin{align*}
	&\rho = -\xi_1^2 - \cdots - \xi_{n-1}^2 - \frac{|a_1-a_2|^2}{1+|a_k|^2} \, \xi_n^2 \text{ for some } k \in \{1,2,\ldots,s\} \text{ or} \\
	&\rho = -\xi_1^2 - \cdots - \xi_{n-1}^2 - \frac{|a_1-a_2|^2}{(1 - C |a_1 - a_2|)^2 \, (1+|a_k|^2)} \, \xi_n^2 \text{ for some } k \in \{s+1,\ldots,q\}. 
\end{align*}
Thus in the definition of parabolic systems we may choose 
\begin{equation*}
	\delta < \min \{1\} \cup \left\{ \frac{|a_1-a_2|^2}{1+|a_k|^2} : k = 1,2,\ldots,s \right\} 
		\cup \left\{ \frac{|a_1-a_2|^2}{(1 - C |a_1 - a_2|)^2 \, (1+|a_k|^2)} : k = s+1,\ldots,q \right\} .
\end{equation*} 

To check the complementing condition, it suffices to consider solutions to \eqref{mcf_lin} of the form 
\begin{equation} \label{barpsi_mcf_eqn}
	\overline{\phi}_2^1 - a_1^1 \, \overline{\psi} = c_1^1 \, e^{\rho \tau + i\xi' \cdot y' - \lambda_1^1 y_n}, \hspace{10mm} 
	\overline{\phi}_k^{\kappa} - a_k^{\kappa} \, \overline{\psi} = c_k^{\kappa} \, e^{\rho \tau + i\xi' \cdot y' - \lambda_k^{\kappa} y_n} \text{ for } (k,\kappa) \neq (1,1), 
\end{equation}
where $\rho \in \mathbb{C}$ and $\xi' \in \mathbb{R}^{n-1}$ with $(\rho,\xi') \neq (0,0)$ and $\op{Re} \rho \geq -\delta_1 |\xi'|^2$ for some $\delta_1 \in (0,\delta)$, $c_k \in \mathbb{C}$, and $\op{Re} \lambda_k^{\kappa} > 0$.  It is readily computed using \eqref{mcf_lin} that  
\begin{align} \label{nu_mcf_eqn}
	\lambda_k^{\kappa} &= \frac{\sqrt{(1+|a_k|^2) \, (\rho + |\xi'|^2)}}{|a_1-a_2|} \text{ if } k \leq s, \\
	\lambda_k^{\kappa} &= \frac{(C \, |a_1 - a_2| - 1) \, \sqrt{(1+|a_k|^2) \, (\rho + |\xi'|^2)}}{|a_1-a_2|} \text{ if } k > s. \nonumber
\end{align}
Recall that the linearization of \eqref{coincide_trans} is \eqref{coincide_lin} and the linearization of \eqref{tangents_trans} is \eqref{tangents_lin}.  Substituting \eqref{barpsi_mcf_eqn}, where $\lambda_k^{\kappa}$ are given by \eqref{nu_mcf_eqn}, into \eqref{coincide_lin} and \eqref{tangents_lin} yields \eqref{coincide_lin2}, \eqref{tangents_lin2a}, and \eqref{tangents_lin2b}.  By the argument in Section \ref{sec:elliptic_section}, the only solution to \eqref{coincide_lin2}, \eqref{tangents_lin2a}, and \eqref{tangents_lin2b} is $c_k^{\kappa} = 0$ for all $k = 1,2,\ldots,q$ and $\kappa = 1,2,\ldots,m$.  Therefore, the system \eqref{mcf_lin}, \eqref{coincide_lin}, and \eqref{tangents_lin} satisfies the complementing condition in $\overline{\phi}_2^1 - a_1 \, \overline{\psi}$ and $\overline{\phi}_k^{\kappa} - a_k^{\kappa} \, \overline{\psi}$ for $(k,\kappa) \neq (1,1)$.  Consequently, the differential system in $\psi$ and $\phi_k^{\kappa}$ for $(k,\kappa) \neq (1,1)$ given by \eqref{mcf_trans} on $I \times U$, \eqref{tangents_trans} on $I \times S$, and \eqref{coincide_trans2} on $I \times S$ is parabolic and satisfies the complimenting condition. 

\begin{proof}[Proof of Theorem \ref{mainthm2} in the case $N$ is smooth]  
Recall that $\psi$ and $\phi_k^{\kappa}$ with $(k,\kappa) \neq (1,1)$ solve a system of the form \eqref{parasystem} that is parabolic and satisfies the complementing condition near the origin.  We can iteratively apply the estimates of~\cite[Theorem 4.11]{Parabolic} in a standard difference quotient argument to show that if $\psi, \phi_1, \phi_2,\ldots,\phi_q \in C^{k+\mu}_{\rm para}(I \times (U \cup S))$ for some integer $k \geq 2$ then $D\psi, D\phi_1, D\phi_2,\ldots,D\phi_q \in C^{k+\mu}_{\rm para}(I \times (U \cup S))$ and thus, noting the continuous embedding $C^{k+\mu}_{\rm para} \subset C^{k-1+\mu}_{\rm para}$, $\psi, \phi_1, \phi_2,\ldots,\phi_q \in C^{k+1+\mu}_{\rm para}(I \times (U \cup S))$.  Therefore, $\psi, \phi_1, \phi_2, \ldots, \phi_q$ are smooth in $I \times (U \cup S)$.  It follows that $u_1,\ldots,u_s$ are smooth near the origin in $\bigcup_{t \in I} \{t\} \times (\Omega_{t,+} \cup \gamma_t)$, $u_{s+1},u_{s+2},\ldots,u_q$ are smooth near the origin in $\bigcup_{t \in I} \{t\} \times (\Omega_{t,-} \cup \gamma_t)$, and $\Gamma_t = \{(t,y',\psi(t,y',0),\phi_2(t,y',0)) : t \in I, \, y' = (y_1,\ldots,y_{n-1}) \in S \}$ is smooth near the origin. 
\end{proof}

\section{Gevrey regularity of parabolic systems} \label{sec:gevrey_section}

We will now complete the proof of Theorem \ref{mainthm2} by showing that when $N$ is real-analytic, $\psi$ and $\phi_k^{\kappa}$ from Section \ref{sec:parabolic_section} are second Gevrey and real-analytic on each time-slice.  This will essentially follow the arguments of~\cite{Friedman}.  However, we must slight modify of the arguments of~\cite{Friedman} to account for the fact that the derivatives of $\tau$ and $y$ are weighted differently.  In particular, we treat the combinatorial aspects of the argument using \eqref{parareg_combo} below.  We in fact prove the following general result: 

\begin{theorem} \label{parareg_thm}
Let $I = (-\rho_0^2,\rho_0^2)$, $U = B_{\rho_0}(0) \cap \{y : y_n > 0\}$, and $S = B_{\rho_0}(0) \cap \{y : y_n = 0\}$ for $\rho_0 > 0$.  Suppose $v_1,v_2\ldots,v_Q \in C^{\infty}(I \times (U \cup S))$ are solutions to the differential system \eqref{parasystem} for some smooth real-valued functions $F_k$ and $\Phi_h$ and integers $b \geq 1$, $s_1,\ldots,s_Q$, $t_1,\ldots,t_Q$, and $r_1,\ldots,r_M$ such that $\max_k s_k = 0$, $\min_{k,j} (s_k+t_j) \geq 0$, $\max_h r_h \leq -1$, and $\min_{j,h} (r_h+t_j) \geq 0$.  Assume $b$ divides each $s_k$ and $t_j$.  Then for every $I' \subset \subset I$ and $U' \subset \subset U$ there exists constants $\widehat{H}_0, \, \widehat{H} \in (0,\infty)$ such that 
\begin{equation} \label{parareg_concl}
	\sup_{I' \times U'} |D_{\tau}^{\alpha} D_y^{\beta} v_j|  
	\leq (2b\alpha+|\beta|)! \, \widehat{H}_0 \, \widehat{H}^{2b\alpha+|\beta|} 
\end{equation}
for all $\alpha$, $\beta$, and $j = 1,2,\ldots,Q$ ($\widehat{H}_0$ and $\widehat{H}$ are independent of $\alpha$ and $\beta$.)  In particular, each $v_j$ is locally Gevrey class $2b$ in $I \times U$ and each $v_j(t,\,\cdot\,)$ is locally real-analytic in $U$ for each $t \in I$. 
\end{theorem}

By scaling, we may assume that $\rho_0 = 1$ and that $v_1,v_2\ldots,v_Q \in C^{\infty}(\overline{I \times U})$.  For each $\tau_0 \in \mathbb{R}$, $y_0 \in \mathbb{R}^n$, and $\rho > 0$, let 
\begin{equation*}
	\mathcal{Q}_{\rho}(\tau_0,y_0) = (\tau_0-\rho^{2b},\tau_0+\rho^{2b}) \times B_{\rho}(y_0), \quad 
	\mathcal{Q}^+_{\rho}(\tau_0,y_0) = \mathcal{Q}_{\rho}(\tau_0,y_0) \cap \{(\tau,y) : y_n \geq 0\}. 
\end{equation*}
By reordering the equations of \eqref{parasystem}, assume that $t_1 = \max_j t_j$ and note that $\max_k (-s_k) \leq t_1$ and $\max_h (-r_h) \leq t_1$.  Since $v_j \in C^{\infty}(\overline{I \times U})$, for some constant $1 \leq H_0 < \infty$, 
\begin{equation} \label{parareg_eqn1}
	\|D_{\tau}^{\alpha} D_y^{\beta} v_j\|_{L^{\infty}(\mathcal{Q}^+_1(0,0))} \leq H_0 
\end{equation}
for all $2b\alpha+|\beta| \leq t_j+2t_1+4+6b$.  We will show that there exists a constant $1 \leq H < \infty$ such that 
\begin{align} \label{parareg_eqn2}
	&\left(\frac{1-\vartheta}{2b\alpha+|\beta|-t_j}\right)^{-\frac{2b+n}{6bn}} \, \|D_{\tau}^{\alpha} D_y^{\beta} v_j
		\|_{L^{6bn}(\mathcal{Q}^+_{\frac{1-\vartheta}{2b\alpha+|\beta|-t_j}}(\tau_0,y_0))} 
	\\&\hspace{10mm} \leq (2b\alpha+|\beta|-t_j-2)! \, H_0 \, H^{2b\alpha+|\beta|-t_j-2-2b} \, (1-\vartheta)^{-2b\alpha-|\beta|+t_j} \nonumber 
\end{align}
for all $2b\alpha+|\beta| > t_j+2t_1+4+6b$, $\vartheta \in (0,1)$, and $(\tau_0,y_0) \in \mathcal{Q}^+_{\vartheta}(0,0)$.  ($H_0$ and $H$ are independent of $\alpha$, $\beta$, and $\vartheta$.)  We consider the $L^{6bn}$ norm with $6bn$ chosen large enough that we could use the Sobolev embedding $W^{1,6bn} \hookrightarrow L^{\infty}$ and the $L^{6bn}$ parabolic estimates \eqref{parareg_eqn5} and \eqref{parareg_eqn6} below.  We will prove \eqref{parareg_eqn2} by induction. 

Having proven \eqref{parareg_eqn2}, by the Sobolev embedding theorem we will obtain $L^{\infty}$ estimates on the derivatives of $v_j$.  In particular, there exists a constant $H_0 \leq \widetilde{H}_0 < \infty$ such that if $\nu > 2t_1+4+6b$ and $v_j$ satisfies \eqref{parareg_eqn1} and \eqref{parareg_eqn2} whenever $2b\alpha+|\beta| \leq t_j+\nu+2b$, then 
\begin{equation} \label{parareg_sobolev} 
	\|D_{\tau}^{\alpha} D_y^{\beta} v_j\|_{L^{\infty}(\mathcal{Q}^+_{\vartheta}(0,0))} \leq (\nu-2)! \, \widetilde{H}_0 \, H^{\nu-2} \, (1-\theta)^{-\nu}
\end{equation}
whenever $t_j+2 < 2b\alpha+|\beta| = t_j+\nu$ (where $\widetilde{H}_0$ is independent of $\nu$ and $\vartheta$).  To see this, it suffices to bound $\|D_{\tau}^{\alpha} D_y^{\beta} v_j\|_{L^{\infty}(\mathcal{Q}^+_{(1-\vartheta)/\nu}(\tau_0,y_0))}$ for each $(\tau_0,y_0) \in \mathcal{Q}^+_{\vartheta}(0,0)$ with either $y_0 \in S$ or $\op{dist}(y_0,S) \geq (1-\vartheta)/\nu$.  By the Sobolev embedding theorem, for $2b\alpha+|\beta| = t_j+\nu$, 
\begin{align*} 
	&\|D_{\tau}^{\alpha} D_y^{\beta} v_j\|_{L^{\infty}(\mathcal{Q}^+_{(1-\vartheta)/\nu}(\tau_0,y_0))} 
	\\&\hspace{10mm} \leq C(n,b) \, \sum_{\gamma+|\zeta| \leq 1} \left(\frac{1-\vartheta}{\nu}\right)^{2b\gamma+|\zeta|-\frac{2b+n}{6bn}} \, 
		\|D_{\tau}^{\alpha+\gamma} D_y^{\beta+\zeta} v_j\|_{L^{6bn}(\mathcal{Q}^+_{(1-\vartheta)/\nu}(\tau_0,y_0))} .
\end{align*}
Thus by \eqref{parareg_eqn1} and\eqref{parareg_eqn2} (together with a covering argument), 
\begin{align*} 
	|D_{\tau}^{\alpha} D_y^{\beta} v_j(\tau_0,y_0)| 
	&\leq C(n,b) \sum_{\gamma+|\zeta| \leq 1} \left(\frac{1-\vartheta}{\nu}\right)^{2b\gamma+|\zeta|} 
		(\nu+2b\gamma+|\zeta|-2)! \nonumber \\& \hspace{30mm} \cdot H_0 \, H^{\nu+2b\gamma+|\zeta|-2b-2} \, (1-\vartheta)^{-\nu-2b\gamma-|\zeta|} 
	\\&\leq (\nu-2)! \, \widetilde{H}_0 \, H^{\nu-2} \, (1-\theta)^{-\nu} ,
\end{align*}
choosing $\widetilde{H}_0 = C(n,b) \, H_0$.

Let $D_{\tau}^p D_y ^q v$ denote any derivative of the function $v$ of the form $D_{\tau}^p D_y^{\beta} v$ for $|\beta| = q$, as opposed to the matrix of all such derivatives.  By differentiating \eqref{parasystem} with respect to $\tau$, 
\begin{align} \label{parareg_eqn3}
	&\sum_{j=1}^Q L'_{kj} D_{\tau} v_j = \sum_{j=1}^Q \sum_{2b\alpha+|\beta| = s_k+t_j} a_{kj}^{\alpha \beta} D_{\tau}^{\alpha} D_y^{\beta} D_{\tau} v_j 
		= f_k \text{ in } I \times U \text{ for } k = 1,\ldots,Q, \\
	&\sum_{j=1}^Q B'_{hj} D_{\tau} v_j = \sum_{j=1}^Q \sum_{2b\alpha+|\beta| = r_h+t_j} b_{hj}^{\alpha \beta} D_{\tau}^{\alpha} D_y^{\beta} D_{\tau} v_j 
		= \varphi_h \text{ on } I \times S \text{ for } h = 1,\ldots,bR, \nonumber
\end{align}
where $L'_{kj}$ and $B'_{hr}$ are the operators for the principle part of the linearization of \eqref{parasystem}, and $a_{kj}^{\alpha \beta}$, $b_{kj}^{\alpha \beta}$, $f_k$, and $\varphi_h$ are smooth functions of $(\tau,y)$.  $a_{kj}^{\alpha \beta}$ and $f_k$ can be expressed in terms of $F_k$ and $D_{\tau}^{\alpha'} D_y^{\beta'} v_{j'}$ for $2b|\alpha'| + |\beta'| \leq s_k+t_{j'}$.  $b_{hj}^{\alpha \beta}$ and $\varphi_h$ can be expressed in terms of $\Phi_h$ and $D_{\tau}^{\alpha'} D_y^{\beta'} v_{j'}$ for $2b|\alpha'| + |\beta'| \leq r_h+t_{j'}$.  Extend $\Phi_h(\tau,y,\{z^{\alpha \beta}_j\})$ to a function of $\tau \in I$, $y \in U \cup S$, and $\{z^{\alpha \beta}_j\}$ that is independent of $y_n$ so that $b_{hj}^{\alpha \beta}$ and $\varphi_h$ extend to functions of $\tau \in I$ and $y \in U \cup S$.  By differentiating \eqref{parareg_eqn3} by $D_{\tau}^{p-1} D_y^q$, 
\begin{align} \label{parareg_eqn4}
	\sum_{j=1}^Q L'_{kj} D_{\tau}^p D_y^q v_j ={}& - \sum_{j=1}^Q \sum_{\lambda=2}^{p-1} \sum_{\sigma=0}^q \sum_{2b\alpha+|\beta| = s_k+t_j} 
		{p-1 \choose \lambda} {q \choose \sigma} D_{\tau}^{\lambda} D_y^{\sigma} a_{kj}^{\alpha \beta} 
		D_{\tau}^{p-1-\lambda+\alpha} D_y^{q-\sigma+|\beta|} v_j \\& + D_{\tau}^{p-1} D_y^q f_k 
	\equiv g_k \text{ in } I \times U \text{ for } k = 1,\ldots,Q, \nonumber \\
	\sum_{j=1}^Q B'_{hj} D_{\tau}^p D_y^q v_j ={}& - \sum_{j=1}^Q \sum_{\lambda=2}^{p-1} \sum_{\sigma=0}^q \sum_{2b\alpha+|\beta| = r_h+t_j} 
		{p-1 \choose \lambda} {q \choose \sigma} D_{\tau}^{\lambda} D_y^{\sigma} b_{hj}^{\alpha \beta} 
		D_{\tau}^{p-1-\lambda+\alpha} D_y^{q-\sigma+|\beta|} v_j \nonumber \\& + D_{\tau}^{p-1} D_y^q \varphi_h 
	\equiv \psi_h \text{ on } I \times S \text{ for } h = 1,\ldots,bR, \nonumber
\end{align}
for every integer $p \geq 1$ and $q \geq 0$, where recall that $D_{\tau}^{\lambda} D_y^{\sigma} a_{kj}^{\alpha \beta}$ for instance means any derivative of order $\lambda$ in $\tau$ and $\sigma$ in $y$ of $a_{kj}^{\alpha \beta}$ and thus the first sum of \eqref{parareg_eqn4} means sum over ${p-1 \choose \lambda} {q \choose \sigma}$ terms consisting of a derivative of order $\lambda$ in $\tau$ and $\sigma$ in $y$ of $a_{kj}^{\alpha \beta}$ times a derivative of order $p-1-\lambda+\alpha$ in $\tau$ and $q-\sigma+|\beta|$ in $y$ of $v_j$, with the particular derivatives possibly differing for each term.  (This is useful notation adopted from~\cite{Friedman}.)  Note that if instead $p = 0$ and $q \geq 1$, we obtain expressions similar to \eqref{parareg_eqn3} and \eqref{parareg_eqn4} by differentiating \eqref{parasystem} by $D_y$ and then by $D_y^{q-1}$.

We want to bound $D_{\tau}^p D_y^q v_j$ using the estimates of~\cite{Parabolic} for solutions to the linear parabolic systems corresponding to the operators $L'_{kj}$ and $B'_{hj}$.  Suppose $v_j$ are smooth solutions to $L'_{kj} v_j = g_k$ in $I \times U$ for $k = 1,\ldots,Q$ and $B'_{hj} v_j = \psi_h$ on $I \times S$ for $h = 1,\ldots,bR$ for some $g_k, \psi_h \in C^{\infty}(I \times U)$.  If $v_j = g_k = \psi_h = 0$ in an open neighborhood of $\partial (I \times U) \setminus (I \times S)$, then by~\cite[Theorems 5.1 and 5.7]{Parabolic}, 
\begin{align} \label{parareg_eqn5}
	&\sum_{j=1}^Q \sum_{2b\alpha+|\beta| = t_j} \|D_{\tau}^{\alpha} D_y^{\beta} v_j\|_{L^{6bn}(\mathcal{Q}^+_1)} 
	\leq C \left( \sum_{j=1}^Q \|v_j\|_{L^{6bn}(\mathcal{Q}_1)} \right. \\& \left. \hspace{10mm}
	+ \sum_{k=1}^Q \sum_{2b\alpha+|\beta| \leq -s_k} \|D_{\tau}^{\alpha} D_y^{\beta} g_k\|_{L^{6bn}(\mathcal{Q}^+_1)}
	+ \sum_{h=1}^{bR} \sum_{2b\alpha+|\beta| \leq -r_h} \|D_{\tau}^{\alpha} D_y^{\beta} \psi_h\|_{L^{6bn}(\mathcal{Q}^+_1)} \right) \nonumber 
\end{align}
for some constant $C \in (0,\infty)$ depending only on $n$, $b$, $Q$, $R$, and the operators $L_{kj},B_{hj}$ and their weights $r_h,s_k,t_j$.  Moreover, if $v_j = g_k = \psi_h = 0$ on an open neighborhood of $I \times S$, we may drop the last sum in \eqref{parareg_eqn5}. 

Now let $\tau_0 \in I$, $y_0 \in U \cup S$, and $\rho,\delta > 0$ such that $\mathcal{Q}^+_{\rho+\delta}(\tau_0,y_0) \subseteq \mathcal{Q}^+_1(0,0)$.  Let $\chi \in C^{\infty}(\mathbb{R}^{1+n})$ be a smooth function such that $0 \leq \chi \leq 1$, $\chi = 1$ on $\mathcal{Q}_{\rho}(\tau_0,y_0)$, $\chi = 0$ on $\mathbb{R}^{1+n} \setminus \mathcal{Q}_{\rho+\delta}(\tau_0,y_0)$, and $|D^{\alpha}_{\tau} D_y^{\beta} \chi| \leq C(n,b,\alpha,\beta) \, \delta^{-2b\alpha-|\beta|}$ for $2b\alpha+|\beta| \leq t_1$.  For instance, fix $\chi_0 \in C^{\infty}(\mathbb{R})$ such that $0 \leq \chi_0 \leq 1$, $\chi_0 = 1$ on $(-\infty,0]$, and $\chi_0 = 0$ on $[1,\infty)$ and let 
\begin{equation*}
	\chi(\tau,y) = \chi_0\left(\frac{|\tau-\tau_0| - \rho^{2b}}{(\rho+\delta)^{2b} - \rho^{2b}}\right) \chi_0\left(\frac{|x-x_0|-\rho}{\delta}\right) 
\end{equation*}
so that 
\begin{equation*}
	|D_{\tau}^{\alpha} D_y^{\beta} \chi(\tau,y)| \leq C(\chi_0,\alpha,\beta) \, ((\rho+\delta)^{2b} - \rho^{2b})^{-\alpha} \, \delta^{-|\beta|}
	\leq C(\chi_0,\alpha,\beta) \, \delta^{-2b\alpha-|\beta|}
\end{equation*}
since $(\rho+\delta)^{2b} - \rho^{2b} \geq \delta^{2b}$.  If $v_j$ are smooth solutions to $L'_{kj} v_j = g_k$ in $I \times U$ for $k = 1,\ldots,Q$ and $B'_{hj} v_j = \psi_h$ on $I \times S$ for $h = 1,\ldots,bR$, then 
\begin{align*}
	L'_{kj}(\chi v_j) &= \chi g_k + L'_{kj}(\chi v_j) - \chi L'_{kj} v_j \text{ in } I \times U \text{ for } k = 1,\ldots,Q, \\
	B'_{hj}(\chi v_j) &= \chi \psi_h + B'_{hj}(\chi v_j) - \chi B'_{hj} v_j \text{ on } I \times S \text{ for } h = 1,\ldots,bR, 
\end{align*}
where $L'_{kj}(\chi v_j) - \chi L'_{kj} v_j$ as an operator of $v_j$ has order $< s_k+t_j$ and $B'_{hj}(\chi v_j) - \chi B'_{hj} v_j$ has order $< r_h+t_j$.  Hence by \eqref{parareg_eqn5} with $\chi v_j$ in place of $v_j$, 
\begin{align} \label{parareg_eqn6}
	&\sum_{j=1}^Q \sum_{2b\alpha+|\beta| = t_j} 
		\|D_{\tau}^{\alpha} D_y^{\beta} v_j\|_{L^{6bn}(\mathcal{Q}^+_{\rho}(\tau_0,y_0))} 
	\\& \hspace{10mm} \leq C \left( \sum_{j=1}^Q \sum_{2b\alpha+|\beta| < t_j} \delta^{-t_j+2b\alpha+|\beta|} 
		\|D_{\tau}^{\alpha} D_y^{\beta} v_j\|_{L^{6bn}(\mathcal{Q}^+_{\rho+\delta}(\tau_0,y_0))} \right. \nonumber 
	\\& \hspace{20mm} + \sum_{k=1}^Q \sum_{2b\alpha+|\beta| \leq -s_k} \delta^{s_k+2b\alpha+|\beta|} 
		\|D_{\tau}^{\alpha} D_y^{\beta} g_k\|_{L^{6bn}(\mathcal{Q}^+_{\rho+\delta}(\tau_0,y_0))} \nonumber 
	\\& \hspace{30mm}  \left. + \sum_{h=1}^{bR} \sum_{2b\alpha+|\beta| \leq -r_h} \delta^{r_h+2b\alpha+|\beta|} 
		\|D_{\tau}^{\alpha} D_y^{\beta} \psi_h\|_{L^{6bn}(\mathcal{Q}^+_{\rho+\delta}(\tau_0,y_0))} \right) \nonumber 
\end{align}
for some constant $C \in (0,\infty)$ is a constant depending only on $n$, $b$, $Q$, $R$, and the operators $L_{kj},B_{hj}$ including their weights and the norms $\|a_{kj}^{\alpha \beta}\|_{C^{t_1}(\mathcal{Q}^+_1(0,0))}$ and $\|b_{hj}^{\alpha \beta}\|_{C^{t_1}(\mathcal{Q}^+_1(0,0))}$ of their coefficients.  Note that the last sum in \eqref{parareg_eqn6} can be dropped if $\mathcal{Q}_{\rho+\delta}(\tau_0,y_0) \subseteq \mathcal{Q}^+_1(0,0)$. 

In order to apply \eqref{parareg_eqn6} with $D_{\tau}^p D_y^q v_j$ in place of $v_j$, we need to bound the derivatives of $a^{\alpha \beta}_{kj}$, $b^{\alpha \beta}_{hj}$, $f_k$, and $\varphi_h$.  For this, we will need the following variant of~\cite[Lemma 1]{Friedman}: 

\begin{lemma} \label{majorants_lemma}
Let $p,q \geq 0$ and $s \in \{-t_1,\ldots,-1,0\}$ be integers and $\vartheta \in (0,1)$.  Consider the composition $X(\tau,y,\{D_{\tau}^{\alpha} D_y^{\beta} v_j\}_{2b\alpha+|\beta| \leq s+t_j})$ where $X$ is a real-analytic function and $v_1,v_2,\ldots,v_Q$ are smooth functions.  For some constant $C \in (0,\infty)$ and for $H \in [1,\infty)$ sufficiently large depending only on $n$, $b$, $Q$, $s$, $t_1,\ldots,t_Q$, $H_0$, and $X$ and independent of $p$, $q$, and $\vartheta$, the following hold true. 
\begin{enumerate}
\item[(i)]  If $2bp+q \leq 2-s+4b$ and $v_j$ satisfies \eqref{parareg_eqn1}, then 
\begin{equation} \label{majorants_con1}
	\left\| D_{\tau}^p D_y^q (X(\tau,y,\{D_{\tau}^{\alpha} D_y^{\beta} v_j\}_{2b\alpha+|\beta| \leq s+t_j})\right\|_{L^{\infty}(\mathcal{Q}^+_1(0,0))} 
	\leq C.  
\end{equation}

\item[(ii)]  If $2bp+q > 2-s+4b$ and $v_j$ satisfies \eqref{parareg_eqn1} and \eqref{parareg_eqn2} for $2b\alpha+|\beta| \leq 2bp+q+s+t_j+2b$, then 
\begin{align} \label{majorants_con2}
	&\left\| D_{\tau}^p D_y^q (X(\tau,y,\{D_{\tau}^{\alpha} D_y^{\beta} v_j\}_{2b\alpha+|\beta| \leq s+t_j}) \right\|_{L^{\infty}(\mathcal{Q}^+_{\vartheta}(0,0))} \\
	&\hspace{20mm} \leq C \, (2bp+q+s-2)! \, H^{2bp+q+s-2} \, (1-\vartheta)^{-2bp-q-s}. \nonumber 
\end{align}

\item[(iii)]  If $2bp+q > 2-s+4b$ and $v_j$ satisfies \eqref{parareg_eqn1} and \eqref{parareg_eqn2} for $2b\alpha+|\beta| \leq 2bp+q+s+t_j$, then 
\begin{align} \label{majorants_con3}
	&\left(\frac{1-\vartheta}{2bp+q}\right)^{-\frac{2b+n}{6bn}} \left\| D_{\tau}^p D_y^q (X(\tau,y,\{D_{\tau}^{\alpha} D_y^{\beta} v_j\}_{2b\alpha+|\beta| \leq s+t_j}) 
		\right\|_{L^{6bn}(\mathcal{Q}^+_{(1-\vartheta)/(2bp+q)}(\tau_0,y_0))} \\&\hspace{20mm} 
	\leq C \, (2bp+q+s-1)! \, H^{2bp+q+s-2-2b} \, (1-\vartheta)^{-2bp-q-s} \nonumber 
\end{align}
for all $(\tau_0,y_0) \in \mathcal{Q}^+_{\vartheta}(0,0)$. 
\end{enumerate} \end{lemma}

Cases (i), (ii), and (iii) in Lemma \ref{majorants_lemma} will allow us to bound both $L^{6bn}$ and $L^{\infty}$ norms of derivatives of compositions.  As we will see below, this will allow us to bound products of derivatives of $a^{\alpha \beta}_{kj}$ (or $b^{\alpha \beta}_{hj}$) and $v_j$ using Lemma \ref{majorants_lemma} and the H\"{o}lder inequality $\|fg\|_{L^{6bn}} \leq \|f\|_{L^{\infty}} \|g\|_{L^{6bn}}$. 

In order to prove and apply Lemma \ref{majorants_lemma}, it will be useful to first make the following preliminary observation.  For all integers $m,n \geq 0$ and $0 \leq k \leq m+n$ ($m$ and $n$ distinct from the dimensions $m$ and $n$ above), we can fill $m+n$ slots with $k$ items, ignoring order, by filling the first $m$ slots with $i$ items and filling the remaining $n$ slots the remaining items and thus we obtain the combinatorial identity 
\begin{equation} \label{parareg_combo0}
	\sum_{i=\max\{0,k-n\}}^{\min\{m,k\}} {m \choose i} {n \choose k-i} = {m+n \choose k}. 
\end{equation} 
We claim that using \eqref{parareg_combo0} we can show that for arbitrary integers $m,n \geq 0$ with $2bm+n \geq 4$, 
\begin{align} \label{parareg_combo}
	\sum_{0 \leq i \leq m, \,\, 0 \leq j \leq n, \,\, 2 \leq 2bi+j \leq 2bm+n-2} \frac{(2bi+j-2)! \, (2b(m-i)+n-j-2)!}{i! \, j! \, (m-i)! \, (n-j)!} &\\ 
	\leq 2\pi^2 \, \frac{(2bm+n-2)!}{m! \, n!}& \nonumber 
\end{align}
By multiplying both sides by $m! \, n!$, we see that we want to bound the quantity  
\begin{equation*}
	S = \sum_{0 \leq i \leq m, \,\, 0 \leq j \leq n, \,\, 2 \leq 2bi+j \leq 2bm+n-2} {m \choose i} {n \choose j} (2bi+j-2)! \, (2b(m-i)+n-j-2)! 
\end{equation*}
above by $2\pi^2 \, (2bm+n-2)!$.  Using the change of variable $k = 2bi+j$, 
\begin{equation*}
	S = \sum_{k=2}^{2bm+n-2} \sum_{\max\{0,(k-n)/2b\} \leq i \leq \min\{m,k/2b\}} {m \choose i} {n \choose k-2bi} (k-2)! \, (2bm+n-k-2)! \,.
\end{equation*}
By \eqref{parareg_combo0}, 
\begin{align} \label{parareg_combo1}
	S &\leq \sum_{k=2}^{2bm+n-2} \sum_{\max\{0,(k-n)/2b\} \leq i \leq \min\{m,k/2b\}} {2bm \choose 2bi} {n \choose k-2bi} (k-2)! \, (2bm+n-k-2)! 
	\\&\leq \sum_{k=2}^{2bm+n-2} {2bm+n \choose k} (k-2)! \, (2bm+n-k-2)! \,. \nonumber 
	\\&\leq \sum_{k=2}^{2bm+n-2} \frac{(2bm+n)!}{(k-1)^2 (2bm+n-k-1)^2}. \nonumber 
\end{align}
Using the identity $\sum_{k=2}^{\infty} (k-1)^{-2} = \pi^2/6$ and $\frac{N(N-1)}{(N-2)^2} \leq 3$ for $N \geq 4$, we obtain for every integer $N \geq 4$ that 
\begin{align} \label{parareg_combo2}
	\sum_{k=2}^{N-2} \frac{N!}{(k-1)^2 \, (N-k-1)^2}  
	&= \sum_{k=2}^{N-2} \frac{N! }{(N-2)^2} \left( \frac{1}{k-1} + \frac{1}{N-k-1} \right)^2  
	\\&\leq \frac{4 \, N!}{(N-2)^2} \sum_{k=2}^{N-2} \frac{1}{(k-1)^2}  
	\leq 2\pi^2 \, (N-2)! \,.  \nonumber 
\end{align}
By combining \eqref{parareg_combo1} and \eqref{parareg_combo2} with $N = 2bm+n$, we obtain \eqref{parareg_combo}.

\begin{proof}[Proof of Lemma \ref{majorants_lemma}]
Observe that for arbritrary smooth functions $\Psi(\tau,y,\{z^{\alpha \beta}_j\}_{2b\alpha+|\beta| \leq s+t_j})$ and $w^{\alpha \beta}_j(\tau,y)$, $D_{\tau}^p D_y^q (\Psi(\tau,y,\{w^{\alpha \beta}_j\}_{2b\alpha+|\beta| \leq s+t_j}))$ is a linear combination of terms of the form 
\begin{equation} \label{majorants_eqn1}
	D_{\tau}^{\gamma_0} D_y^{\zeta_0} D_{z^{\alpha_1 \beta_1}_{j_1}} \cdots D_{z^{\alpha_k \beta_k}_{j_k}} \Psi 
	\cdot \prod_{i=1}^k D_{\tau}^{\gamma_i} D_y^{\zeta_i} w^{\alpha_i \beta_i}_{j_i},  
\end{equation}
with nonnegative integer coefficients, where $2b\alpha_i+|\beta_i| \leq s+t_{j_i}$ for $i = 1,2,\ldots,k$, $\gamma_i+|\zeta_i| \geq 1$ for $i = 1,2,\ldots,k$, $\sum_{i=0}^k \gamma_i = p$, and $\sum_{i=0}^k |\zeta_i| = q$ and we allow $k = 0$ to include the terms $D_{\tau}^p D_y^q \Psi$. 

(i) follows from taking the $L^{\infty}$-norm of \eqref{majorants_eqn2} with $\Psi = X$ and $w^{\alpha \beta}_j = D_{\tau}^{\alpha} D_y^{\beta} v_j$ and applying \eqref{parareg_eqn1}.  It suffices to use \eqref{parareg_eqn1} here since we only consider $D_{\tau}^{\alpha} D_y^{\beta} v_j$ for $2b\alpha+|\beta| \leq s+t_j+2bp+q \leq t_j+2+4b$. 

To prove (ii) and (iii), we need more precise bounds, which we will obtain using majorants.  Let the function $\Psi(\tau,y,\{z^{\alpha \beta}_j\}_{2b\alpha+|\beta| \leq s+t_j})$ be the majorant corresponding to $X$ such that  
\begin{equation} \label{majorants_eqn2}
	\|D^{\gamma} X\|_{L^{\infty}} \leq D^{\gamma} \Psi(0) \text{ for } 1 \leq |\gamma| \leq p+q
\end{equation}
whenever $1 \leq |\gamma| \leq p+q$.  Select majorants $w^{\alpha \beta}_j(\tau,y)$ corresponding to $D_{\tau}^{\alpha} D_y^{\beta} v_j$ such that $w^{\alpha \beta}_j(0) = 0$ and 
\begin{equation} \label{majorants_eqn3}
	(1-\vartheta)^{\max\{2b\alpha+|\beta|+2b\gamma+|\zeta|-t_j,0\}} \, \|D_{\tau}^{\alpha+\gamma} D_y^{\beta+\zeta} v_j
		\|_{L^{\infty}(\mathcal{Q}^+_{\vartheta}(0,0))} 
	\leq D_{\tau}^{\gamma} D_y^{\zeta} w^{\alpha \beta}_j(0)   
\end{equation}
whenever $2b\alpha+|\beta| \leq s+t_j$ and $1 \leq 2b\gamma+|\zeta| \leq 2bp+q$.  By \eqref{majorants_eqn2}, \eqref{majorants_eqn3}, and $2bp+q \geq -s$, when $k \geq 1$, 
\begin{align*}
	&(1-\vartheta)^{2bp+q+s} \, \left\|D_{\tau}^{\gamma_0} D_y^{\zeta_0} D_{z^{\alpha_1 \beta_1}_{j_1}} \cdots 
		D_{z^{\alpha_k \beta_k}_{j_k}} X \cdot \prod_{i=1}^k D_{\tau}^{\gamma_i} D_y^{\zeta_i} v^{\alpha_i \beta_i}_{j_i} 
		\right\|_{L^{\infty}(\mathcal{Q}^+_{\vartheta})} \\
	&\leq (1-\vartheta)^{2bp+q+s} \, D_{\tau}^{\gamma_0} D_y^{\zeta_0} D_{z^{\alpha_1 \beta_1}_{j_1}} \cdots 
		D_{z^{\alpha_k \beta_k}_{j_k}} \Psi(0) \cdot \prod_{i=1}^k (1-\vartheta)^{-\max\{s+2b\gamma_i+|\zeta_i|,0\}} 
		D_{\tau}^{\gamma_i} D_y^{\zeta_i} w^{\alpha_i \beta_i}_{j_i}(0) \\
	&\leq D_{\tau}^{\gamma_0} D_y^{\zeta_0} D_{z^{\alpha_1 \beta_1}_{j_1}} \cdots D_{z^{\alpha_k \beta_k}_{j_k}} \Psi(0) \cdot 
		\prod_{i=1}^k D_{\tau}^{\gamma_i} D_y^{\zeta_i} w^{\alpha_i \beta_i}_{j_i}(0),
\end{align*}
where $2b\alpha_i+|\beta_i| \leq s+t_{j_i}$, $\gamma_i+|\zeta_i| \geq 1$, $\sum_{i=0}^k \gamma_i = p$, and $\sum_{i=0}^k |\zeta_i| = q$.  The case $k = 0$ is similar.  Therefore, by direct comparison via \eqref{majorants_eqn1}, 
\begin{equation} \label{majorants_eqn4}
	(1-\vartheta)^{2bp+q+s} \, \left\| D_{\tau}^p D_y^q (X(\tau,y,\{D_{\tau}^{\alpha} D_y^{\beta} v_j\})) 
	\right\|_{L^{\infty}(\mathcal{Q}^+_{\vartheta})} \leq \left. D_{\tau}^p D_y^q (\Psi(\tau,y,\{w^{\alpha \beta}_j\})) \right|_{\tau=0, \, y=0}.  
\end{equation}

We choose $\Psi$ and $w$ as follows.  Set $\xi = y_1+\cdots+y_n$.  By \eqref{parareg_eqn1} and \eqref{parareg_sobolev}, we can choose $w^{\alpha \beta}_j = w$ given by 
\begin{equation} \label{majorants_eqn5}
	w = \widetilde{H}_0 \xi + \sum_{2 \leq 2bi+j \leq 2bp+q} \frac{(2bi+j-2)!}{i! \, j!} \widetilde{H}_0 H^{\max\{2bi+j+s-2,0\}} \tau^i \xi^j 
\end{equation}
so that \eqref{majorants_eqn3} holds true.  We can choose $\Psi = \Psi\hspace{-.7mm}\left(\tau+y_1+\cdots+y_n,\sum_{j,\alpha,\beta} z_j^{\alpha \beta}\right)$.  Since $X$ is real-analytic, for some constants $1 \leq K_0,K < \infty$, 
\begin{equation*}
	\sup |D_{(\tau,y)}^{\gamma} D_z^{\zeta} X| \leq (\max\{|\gamma|-2,0\})! \, (\max\{|\zeta|-2,0\})! \, K_0 \, K^{|\gamma|+|\zeta|} \text{ for all } \gamma,\zeta, 
\end{equation*}
where $z = \{z^{\alpha \beta}_j\}_{2b\alpha+|\beta| \leq s+t_j}$, and thus we can choose 
\begin{align} \label{majorants_eqn6}
	\Psi = {}& K_0 K (\tau+\xi) + K_0 K N w + K_0 K^2 N (\tau+\xi) w + \sum_{i=2}^{p+q} \frac{K_0 K^i}{(i-1)^2} (\tau+\xi)^i (1+KNw) \\& 
		+ \sum_{j=2}^{p+q} \frac{K_0 K^j}{(j-1)^2} (1+K\tau+K\xi) (Nw)^j + \sum_{i,j=2}^{p+q} \frac{K_0 K^{i+j}}{(i-1)^2 (j-1)^2} (\tau+\xi)^i (N w)^j, \nonumber 
\end{align}
so that \eqref{majorants_eqn2} holds true, where $N$ equals the total numbers of entries of $\{z^{\alpha \beta}_j\}_{2b\alpha+|\beta| \leq s+t_j}$ and depends only on $n$, $s$, and $t_1,\ldots,t_Q$. 

By inductively multiplying $w$ by itself using \eqref{parareg_combo}, for some constant $1 \leq C_0 < \infty$, 
\begin{equation} \label{majorants_eqn7}
	w^k << \widetilde{H}_0^k \xi^k + \sum_{k+1 \leq 2bi+j \leq 2bp+q} \frac{(2bi+j-2)!}{i! \, j!} C_0^{k-1} \widetilde{H}_0^k H^{\max\{2bi+j+s-k-1,0\}} \tau^i \xi^j
\end{equation}
for $k = 1,2,\ldots,2bp+q$, where $f << g$ means $D_{\tau}^i D_{\xi}^j f(0) \leq D_{\tau}^i D_{\xi}^j g(0)$ for all $0 \leq i \leq p$ and $0 \leq j \leq q$ not both zero.  To see this, observe that \eqref{majorants_eqn7} clearly holds true when $k = 1$ and for the induction step, using \eqref{parareg_combo} and $\frac{j}{2bi+j-2} \leq 3$ if $i,j \geq 0$ and $2bi+j \geq 3$, we have for $k \geq 1$ 
\begin{align*}
	w^{k+1} <<{}& \left( \widetilde{H}_0 \xi + \sum_{2 \leq 2bi+j \leq 2bp+q} \frac{(2bi+j-2)!}{i! \, j!} \widetilde{H}_0 H^{\max\{2bi+j+s-2,0\}} \tau^i \xi^j \right) \\&
		\cdot \left( \widetilde{H}_0^k \xi^k + \sum_{k+1 \leq 2bi+j \leq 2bp+q} \frac{(2bi+j-2)!}{i! \, j!} C_0^{k-1} \widetilde{H}_0^k H^{\max\{2bi+j+s-k-1,0\}} 
		\tau^i \xi^j \right) \\
	<<{}& \widetilde{H}_0^{k+1} \xi^{k+1} + \sum_{k+2 \leq 2bi+j \leq 2bp+q} \frac{(2bi+j-3)!}{i! \, (j-1)!} C_0^{k-1} \widetilde{H}_0^{k+1} 
		H^{\max\{2bi+j+s-k-2,0\}} \tau^i \xi^j \\& + \sum_{k+2 \leq 2bi+j \leq 2bp+q} \frac{(2bi+j-k-2)!}{i! \, (j-k)!} \widetilde{H}_0^{k+1} H^{\max\{2bi+j+s-k-2,0\}} 
		\tau^i \xi^j \\& + \sum_{k+3 \leq 2bi+j \leq 2bp+q} \, \sum_{0 \leq r \leq i, \, 0 \leq l \leq j, \, 2 \leq 2br+l \leq 2bi+j-k-1} 
		\frac{(2br+l-2)!(2bi-2br+j-l-2)!}{r! \, l! \, (i-r)! \, (j-l)!} \\&\hspace{81mm} \cdot C_0^{k-1} \widetilde{H}_0^{k+1} H^{\max\{2bi+j+s-k-3,0\}} \tau^i \xi^j \\
	<<{}& \widetilde{H}_0^{k+1} \xi^{k+1} + \sum_{k+2 \leq 2bi+j \leq 2bp+q} \frac{(2bi+j-2)!}{i! \, j!} (3C_0^{k-1} + 3^k + 2\pi^2 C_0^{k-1})
		 \\&\hspace{70mm} \cdot \widetilde{H}_0^{k+1} H^{\max\{2bi+j+s-k-2,0\}} \tau^i \xi^j \\
	<<{}& \widetilde{H}_0^{k+1} \xi^{k+1} + \sum_{k+2 \leq 2bi+j \leq 2bp+q} \frac{(2bi+j-2)!}{i! \, j!} 
		C_0^k \widetilde{H}_0^{k+1} H^{\max\{2bi+j+s-k-2,0\}} \tau^i \xi^j
\end{align*}
if we choose $C_0 = 6 + 2\pi^2$.  By substituting \eqref{majorants_eqn7} into \eqref{majorants_eqn6},  
\begin{align*}
	\Psi << {}& K_0 K (\tau+\xi) + \sum_{i=2}^{p+q} \frac{K_0 K^i}{(i-1)^2} (\tau+\xi)^i 
		+ \left( K_0 + K_0 K (\tau+\xi) + \sum_{i=2}^{p+q} \frac{K_0 K^i}{(i-1)^2} (\tau+\xi)^i \right) \\&
		\cdot \left( K N \widetilde{H}_0 \xi + \sum_{2 \leq 2bk+l \leq 2bp+q} \frac{(2bk+l-2)!}{k! \, l!} K N \widetilde{H}_0 H^{\max\{s+2bk+l-2,0\}} \tau^k \xi^l 
		+ \sum_{j=2}^{p+q} \frac{K^j N^j \widetilde{H}_0^j}{(j-1)^2} \xi^j \right. 
		\\& \left. + \sum_{2 \leq j < 2bk+l \leq 2bp+q} \frac{(2bk+l-2)!}{(j-1)^2 \, k! \, l!} K^j N^j C_0^{j-1} \widetilde{H}_0^j H^{\max\{s+2bk+l-j-1,0\}} \tau^k \xi^l \right) 
\end{align*}
Choose $H$ such that $H \geq 2 N C_0 K \widetilde{H}_0$.  By expanding $(\tau+\xi)^i$ and using the choice of $H$, \eqref{parareg_combo}, and $\sum_{j=0}^{\infty} 2^{-j} = 2$, 
\begin{align} \label{majorants_eqn8}
	\Psi << {}& K_0 K (\tau+\xi) + 2 \sum_{2 \leq i+j \leq p+q} \frac{(i+j-2)!}{i! \, j!} K_0 K^{i+j} \tau^i \xi^j \\&
		+ \left( K_0 + K_0 K (\tau+\xi) + 2 \sum_{2 \leq i+j \leq p+q} \frac{(i+j-2)!}{i! \, j!} K_0 K^{i+j} \tau^i \xi^j \right) \nonumber \\& 
		\cdot \left( K N \widetilde{H}_0 \xi 
		+ C(s) \sum_{2 \leq 2bk+l \leq 2bp+q} \frac{(2bk+l-2)!}{k! \, l!} (K N \widetilde{H}_0)^{2-s} H^{\max\{s+2bk+l-2,0\}} \tau^k \xi^l \right) \nonumber \\ 
	<< {}& K_0 K \tau + 2 K_0 K N \widetilde{H}_0 \xi \nonumber\\ {}&
		+ C(s) \sum_{2 \leq 2bi+j \leq 2bp+q} \frac{(2bi+j-2)!}{i! \, j!} K_0 K^{2-s} (K N \widetilde{H}_0)^{2-s} H^{\max\{s+2bi+j-2,0\}} \tau^i \xi^j. \nonumber
\end{align}
Therefore, by \eqref{majorants_eqn4} we obtain \eqref{majorants_con2}. 

To prove (iii), we will modify the above argument.  Let $\Psi(\tau,y,\{z^{\alpha \beta}_j\}_{2b\alpha+|\beta| \leq s+t_j})$ be a majorant corresponding to $X$ such that \eqref{majorants_eqn2} holds true whenever $1 \leq |\gamma| \leq p+q$.  Select majorants $w^{\alpha \beta}_j(\tau,y)$ corresponding to $D_{\tau}^{\alpha} D_y^{\beta} v_j$ such that $w^{\alpha \beta}_j(0) = 0$, 
\begin{equation} \label{majorants_eqn9}
	(1-\vartheta)^{\max\{2b\alpha+|\beta|+2b\gamma+|\zeta|-t_j,0\}} \, 
		\|D_{\tau}^{\alpha+\gamma} D_y^{\beta+\zeta} v_j\|_{L^{\infty}(\mathcal{Q}^+_{(1-\vartheta)/(2bp+q)}(\tau_0,y_0))} 
	\leq D_{\tau}^{\gamma} D_y^{\zeta} w^{\alpha \beta}_j(0)   
\end{equation}
whenever $2b\alpha+|\beta| \leq s+t_j$ and $1 \leq 2b\gamma+|\zeta| \leq 2bp+q-2b$, and 
\begin{align} \label{majorants_eqn10}
	&(1-\vartheta)^{\max\{2b\alpha+|\beta|+2b\gamma+|\zeta|-t_j,0\}} \, \left(\frac{1-\vartheta}{2bp+q}\right)^{-\frac{2b+n}{6bn}} \, 
		\|D_{\tau}^{\alpha+\gamma} D_y^{\beta+\zeta} v_j\|_{L^{6bn}(\mathcal{Q}^+_{(1-\vartheta)/(2bp+q)}(\tau_0,y_0))} 
	\\&\leq (2bp+q) \, H^{-2b} \, D_{\tau}^{\gamma} D_y^{\zeta} w^{\alpha \beta}_j(0) \nonumber 
\end{align}
whenever $2b\alpha+|\beta| \leq s+t_j$ and $2-s+2b \leq 2b\gamma+|\zeta| \leq 2bp+q$.  Define the majorant $\widetilde{w}^{\alpha \beta}_j$ by 
\begin{equation*}
	\widetilde{w} = \widetilde{w}^{\alpha \beta}_j = e^{H_0(\tau+\xi)} - 1 
\end{equation*} 
so that $\widetilde{w}(0) = 0$ and 
\begin{equation} \label{majorants_eqn11}
	D_{\tau}^{\gamma} D_y^{\zeta} \widetilde{w}(0) \leq H_0 
\end{equation}
for $2b\gamma+|\zeta| \leq 2-s+2b$.  Consider \eqref{majorants_eqn1} with $k \geq 2$, noting that the cases $k = 0,1$ are similar.  Assume $2b\gamma_1+|\zeta_1| \geq 2b\gamma_i+|\zeta_i|$ for all $i \geq 2$.  Observe that $2b\gamma_i+|\zeta_i| > 2bp+q-2b$ for at most one $i$ since otherwise by summing over such $i$ we obtain $2bp+q > 2(2bp+q)-4b$, i.e. $2bp+q < 4b$, contradicting $2bp+q > 2-s+4b$.  If $2b\gamma_1+|\zeta_1| \geq 2-s+2b$, then by \eqref{majorants_eqn2}, \eqref{majorants_eqn9}, \eqref{majorants_eqn10}, and $2bp+q \geq -s$, when $k \geq 1$, 
\begin{align*}
	&(1-\vartheta)^{2bp+q+s} \, \left(\frac{1-\vartheta}{2bp+q}\right)^{-\frac{2b+n}{6bn}} \, 
		\left\|D_{\tau}^{\gamma_0} D_y^{\zeta_0} D_{z^{\alpha_1 \beta_1}_{j_1}} \cdots D_{z^{\alpha_k \beta_k}_{j_k}} X \cdot 
		\prod_{i=1}^k D_{\tau}^{\gamma_i} D_y^{\zeta_i} v^{\alpha_i \beta_i}_{j_i} \right\|_{L^{6bn}(\mathcal{Q}^+_{\frac{1-\vartheta}{2bp+q}}(\tau_0,y_0))} \\
	&\leq (1-\vartheta)^{2bp+q+s} \|D_{\tau}^{\gamma_0} D_y^{\zeta_0} D_{z^{\alpha_1 \beta_1}_{j_1}} \cdots D_{z^{\alpha_k \beta_k}_{j_k}} X\|_{L^{\infty}} \\
		&\hspace{10mm} \cdot \left(\frac{1-\vartheta}{2bp+q}\right)^{-\frac{2b+n}{6bn}} 
		\|D_{\tau}^{\gamma_1} D_y^{\zeta_1} v^{\alpha_1 \beta_1}_{j_1}\|_{L^{6bn}(\mathcal{Q}^+_{\frac{1-\vartheta}{2bp+q}}(\tau_0,y_0))}
		\cdot \prod_{i=2}^k \|D_{\tau}^{\gamma_i} D_y^{\zeta_i} v^{\alpha_i \beta_i}_{j_i}\|_{L^{\infty}(\mathcal{Q}^+_{\frac{1-\vartheta}{2bp+q}}(\tau_0,y_0))} \\
	&\leq (2bp+q) \, H^{-2b} \, (1-\vartheta)^{2bp+q+s} \, D_{\tau}^{\gamma_0} D_y^{\zeta_0} D_{z^{\alpha_1 \beta_1}_{j_1}} \cdots D_{z^{\alpha_k \beta_k}_{j_k}} 
		\Psi(0) \\&\hspace{10mm} \cdot \prod_{i=1}^k (1-\vartheta)^{-\max\{s+2b\gamma_i+|\zeta_j|,0\}} \, 
		D_{\tau}^{\gamma_i} D_y^{\zeta_i} w^{\alpha_i \beta_i}_{j_i}(0) \\
	&\leq (2bp+q) \, H^{-2b} \, D_{\tau}^{\gamma_0} D_y^{\zeta_0} D_{z^{\alpha_1 \beta_1}_{j_1}} \cdots D_{z^{\alpha_k \beta_k}_{j_k}} \Psi(0) \cdot 
		\prod_{i=1}^k D_{\tau}^{\gamma_i} D_y^{\zeta_i} w^{\alpha_i \beta_i}_{j_i}(0). 
\end{align*}
If instead $2b\gamma_i+|\zeta_i| \leq 2-s+2b$ for all $i$, then by \eqref{parareg_eqn1}, \eqref{majorants_eqn2}, and \eqref{majorants_eqn11}
\begin{align} \label{majorants_eqn12}
	&(1-\vartheta)^{2bp+q+s} \, \left\|D_{\tau}^{\gamma_0} D_y^{\zeta_0} D_{z^{\alpha_1 \beta_1}_{j_1}} \cdots 
		D_{z^{\alpha_k \beta_k}_{j_k}} X \cdot \prod_{i=1}^k D_{\tau}^{\gamma_i} D_y^{\zeta_i} v^{\alpha_i \beta_i}_{j_i} 
		\right\|_{L^{\infty}(\mathcal{Q}^+_{(1-\vartheta)/(2bp+q)}(\tau_0,y_0))} \\
	&\leq D_{\tau}^{\gamma_0} D_y^{\zeta_0} D_{z^{\alpha_1 \beta_1}_{j_1}} \cdots D_{z^{\alpha_k \beta_k}_{j_k}} \Psi(0) \cdot H_0^k \nonumber \\
	&\leq D_{\tau}^{\gamma_0} D_y^{\zeta_0} D_{z^{\alpha_1 \beta_1}_{j_1}} \cdots D_{z^{\alpha_k \beta_k}_{j_k}} \Psi(0) \cdot 
		\prod_{i=1}^k D_{\tau}^{\gamma_i} D_y^{\zeta_i} \widetilde{w}^{\alpha_i \beta_i}_{j_i}(0). \nonumber 
\end{align}
By direct comparison via \eqref{majorants_eqn1}, 
\begin{align} \label{majorants_eqn13}
	&(1-\vartheta)^{2bp+q+s} \, \left(\frac{1-\vartheta}{2bp+q}\right)^{-\frac{2b+n}{6bn}} \, \left\| D_{\tau}^p D_y^q (X(\tau,y,\{D_{\tau}^{\alpha} D_y^{\beta} v_j\})) 
		\right\|_{L^{6bn}(\mathcal{Q}^+_{(1-\vartheta)/(2bp+q)}(\tau_0,y_0))} 
	\\&\leq \left. (2bp+q) \, H^{-2b} \, D_{\tau}^p D_y^q (\Psi(\tau,y,\{w^{\alpha \beta}_j\})) \right|_{\tau=0, \, y=0} 
		+ \left. C(n,b) \, D_{\tau}^p D_y^q (\Psi(\tau,y,\{\widetilde{w}^{\alpha \beta}_j\})) \right|_{\tau=0, \, y=0}.  \nonumber 
\end{align}

We can choose $\Psi$ by \eqref{majorants_eqn6}.  Choose $w$ by \eqref{majorants_eqn5} with $e \widetilde{H}_0$ in place of $\widetilde{H}_0$.  Then by \eqref{parareg_sobolev} and $(1-1/k)^{1-k} \leq e$ for $k \geq 1$, 
\begin{align*}
	&\|D_{\tau}^{\alpha} D_y^{\beta} v_j\|_{L^{\infty}(\mathcal{Q}^+_{(1-\vartheta)/(2bp+q)}(\tau_0,y_0))} 
	\\&\leq (2b\alpha+|\beta|-t_j-2)! \, \widetilde{H}_0 \, H^{2b\alpha+|\beta|-t_j-2} \left(\frac{2bp+q-1}{2bp+q} \, (1-\vartheta)\right)^{-2b\alpha-|\beta|+t_j} 
	\\&\leq e \, (2b\alpha+|\beta|-t_j-2)! \, \widetilde{H}_0 \, H^{2b\alpha+|\beta|-t_j-2} (1-\vartheta)^{-2b\alpha-|\beta|+t_j} 
\end{align*}
if $t_j +2+6b < 2b\alpha+|\beta| \leq t_j+2bp+q-2b$, which together with \eqref{parareg_eqn1} gives us \eqref{majorants_eqn9}.  \eqref{parareg_eqn1} and \eqref{parareg_eqn2} imply that 
\begin{align*}
	&\left(\frac{1-\vartheta}{2bp+q}\right)^{-\frac{2b+n}{6bn}} \|D_{\tau}^{\alpha} D_y^{\beta} v_j\|_{L^{6bn}(\mathcal{Q}^+_{(1-\vartheta)/(2bp+q)}(\tau_0,y_0))} 
	\\&\leq \left(\frac{2bp+q}{2b\alpha+|\beta|-t_j}\right)^{\frac{2b+n}{6bn}} \left(\frac{1-\vartheta}{2b\alpha+|\beta|-t_j}\right)^{-\frac{2b+n}{6bn}} 
		\|D_{\tau}^{\alpha} D_y^{\beta} v_j\|_{L^{6bn}(\mathcal{Q}^+_{(1-\vartheta)/(2b\alpha+|\beta|-t_j)}(\tau_0,y_0))} 
	\\&\leq (2bp+q) \, (2b\alpha+|\beta|-t_j-2)! \, \widetilde{H}_0 \, H^{2b\alpha+|\beta|-t_j-2-2b} \, (1-\vartheta)^{-2b\alpha-|\beta|+t_j} 
\end{align*}
if $t_j +2+2b \leq 2b\alpha+|\beta| \leq t_j+2bp+q$, giving us \eqref{majorants_eqn10}. By the above computation of \eqref{majorants_eqn8}, we have the desired bound the first term on the right-hand side of \eqref{majorants_eqn13}, so it remains to bound the last term in \eqref{majorants_eqn13}. 

Observe that $\widetilde{w} << e^{H_0 (\tau+\xi)}$ implies $\widetilde{w}^j << e^{j H_0 (\tau+\xi)}$ and that, since $\widetilde{w}(0) = 0$, the derivatives of $\widetilde{w}^j$ of order $< j$ vanish, hence by the Taylor series of the exponential function 
\begin{equation*}
	\widetilde{w}^j << \sum_{k=j}^{p+q} \frac{j^k}{k!} \, H_0^k \, (\tau+\xi)^k
\end{equation*}
for $j \geq 1$.  Using $j^{k-2}/(k-2)! \leq e^{j-2}$ from the Taylor series of the exponential function, 
\begin{equation*}
	\widetilde{w}^j << \sum_{k=j}^{p+q} \frac{j^2 \, e^{j-2}}{(k-1)^2} \, H_0^k \, (\tau+\xi)^k 
\end{equation*}
for $j \geq 2$.  Using $\frac{j^2}{(j-1)^2} \leq 4$ for $j \geq 2$, 
\begin{align*}
	\sum_{j=2}^{p+q} \frac{N^j \widetilde{w}^j}{(j-1)^2} << {}& \sum_{j=2}^{p+q} \sum_{k=j}^{p+q} \frac{4e^{j-2} N^j}{(k-1)^2} H_0^k (\tau+\xi)^k 
	= \sum_{k=2}^{p+q} \sum_{j=2}^k \frac{4 e^{j-2} N^j}{(k-1)^2} H_0^k (\tau+\xi)^k \\
	<< {}& \sum_{k=2}^{p+q} \frac{4 N^2 (eN)^{k-1}}{(eN-1) (k-1)^2} H_0^k (\tau+\xi)^k
\end{align*}
Hence taking \eqref{majorants_eqn6} with $\widetilde{w}$ in place of $w$ and substituting for $\widetilde{w}$, 
\begin{align*}
	\Psi << {}& K_0 K (\tau+\xi) + \sum_{i=2}^{p+q} \frac{K_0 K^i}{(i-1)^2} (\tau+\xi)^i 
		+ \left( K_0 + K_0 K (\tau+\xi) + \sum_{i=2}^{p+q} \frac{K_0 K^i}{(i-1)^2} (\tau+\xi)^i \right) \\&
		\cdot \left( \sum_{k=1}^{p+q} \frac{H_0^k}{k!} (\tau+\xi)^k + \sum_{k=2}^{p+q} \frac{4 N^2 (e N H_0)^k}{(eN-1) (k-1)^2} (\tau+\xi)^k \right) .
\end{align*}
By expanding this expression and using \eqref{parareg_combo2}, 
\begin{equation*}
	\Psi <<  K_0 (K+H_0) (\tau+\xi) + C(N) \sum_{i=2}^{p+q} \frac{K_0 (K+eNH_0)^i}{(i-1)^2} (\tau+\xi)^i. 
\end{equation*}
By expanding $(\tau+\xi)^i$, 
\begin{equation} \label{majorants_eqn14}
	\Psi <<  K_0 (K+H_0) (\tau+\xi) + C(N) \sum_{2 \leq i+j \leq p+q} \frac{(i+j-2)!}{i! \, j!} K_0 (K+eNH_0)^{i+j} \tau^i \xi^j. 
\end{equation}
Choose $H$ so that $H \geq \max\{2 N C_0 K \widetilde{H}_0,K+eNH_0\}$.  By using \eqref{majorants_eqn8} and \eqref{majorants_eqn14} to compute the right-hand side of \eqref{majorants_eqn13}, we obtain \eqref{majorants_con3}. 
\end{proof}

\begin{proof}[Proof of Theorem \ref{parareg_thm}]  Let $v_j \in C^{\infty}(\overline{I \times U})$ be smooth solutions to the differential system \eqref{parasystem} for real-valued functions $F_k$ and $\Phi_h$ and weights $b,t_j,s_k,r_h$ as in the statement of Theorem \ref{parareg_thm}.  Assume the induction hypotheses that for some 
\begin{equation} \label{parareg_eqn8}
	\nu > 2t_1+4+6b, 
\end{equation}
\eqref{parareg_eqn1} and \eqref{parareg_eqn2} hold true whenever $2b\alpha+|\beta| < t_j+\nu$, where $H \geq 1$ is a large constant to be chosen below and in particular $H$ is large enough that Lemma \ref{majorants_lemma} applies when we bound derivatives of $f_k$, $a_{kj}^{\alpha \beta}$, $\varphi_h$, and $b_{hj}^{\alpha \beta}$ below.  We want to show that \eqref{parareg_eqn2} holds true when $2b\alpha+|\beta| = t_j+\nu$.  Let $p,q \geq 0$ be any integers such that $2bp+q = \nu$ and assume $p \geq 1$, noting that case $p = 0$ is similar.  Let $f_k$ and $\varphi_h$ be as in \eqref{parareg_eqn3} and $g_k$ and $\psi_h$ be as in \eqref{parareg_eqn4}.  Let $\vartheta \in (0,1)$, $(\tau_0,y_0) \in \mathcal{Q}^+_{\vartheta}(0,0)$, $\rho = 2\delta = (1-\vartheta)/\nu$. 

We can bound $\|g_k\|_{L^{6bn}(\mathcal{Q}^+_{3\rho/2}(\tau_0,y_0))}$ using the induction hypothesis and Lemma \ref{majorants_lemma}.  Using induction hypothesis, Lemma \ref{majorants_lemma} (together with a covering argument), and $(1-1/\nu)^{1-\nu} \leq e$, 
\begin{align*}
	&\sum_{2b\alpha+|\beta| \leq -s_k} \rho^{s_k+2b\alpha+|\beta|-\frac{2b+n}{6bn}} \, 
		\|D_{\tau}^{\alpha+p-1} D_y^{\beta} D_y^q f_k \|_{L^{6bn}(\mathcal{Q}^+_{3\rho/2}(\tau_0,y_0))} \\ 
	&\leq C \sum_{2b\alpha+|\beta| \leq -s_k} \left(\frac{1-\vartheta}{\nu}\right)^{s_k+2b\alpha+|\beta|} \, (\nu+2b\alpha+|\beta|+s_k-2)! \nonumber \\&
		\hspace{15mm} \cdot H^{\nu+2b\alpha+|\beta|+s_k-3-2b} \left(\frac{(\nu-1) (1-\vartheta)}{\nu}\right)^{-\nu-2b\alpha-|\beta|-s_k+1} \nonumber \\
	&\leq C \, e \, (\nu-2)! \, H^{\nu-3-2b} \, (1-\vartheta)^{1-\nu} \nonumber 
\end{align*}
for all $k = 1,2,\ldots,Q$ and some constants $C \in (0,\infty)$ depending on $n$, $b$, $Q$, $H_0$, $F_k$, $s_k$, and $t_j$.  

We want to similarly bound 
\begin{align*}
	S = {}& \sum_{2b\alpha+|\beta| \leq -s_k} \sum_{\lambda=0}^{\alpha+p-1} \, \sum_{\sigma=0}^{|\beta|+q} {\alpha+p-1 \choose \lambda} 
		{|\beta|+q \choose \sigma} \, \rho^{s_k+2b\alpha+|\beta|} \\& \hspace{15mm} \cdot 
		\rho^{-\frac{2b+n}{6bn}} \, \|D_{\tau}^{\lambda} D_y^{\sigma} a_{kj}^{\gamma \zeta} \cdot 
		D_{\tau}^{\alpha+\gamma+p-1-\lambda} D_y^{|\beta|+|\zeta|+q-\sigma} v_j\|_{L^{6bn}(\mathcal{Q}^+_{3\rho/2}(\tau_0,y_0))} \nonumber 
\end{align*}
for $j,k = 1,2,\ldots,Q$ and $2b\gamma+|\zeta| = s_k+t_j$.  We bound the terms as follows.  If $2b\lambda+\sigma \leq \nu+s_k-2-4b$, by Lemma \ref{majorants_lemma}(i)(ii) and the induction hypotheses,  
\begin{align*}
	&\|D_{\tau}^{\lambda} D_y^{\sigma} a_{kj}^{\gamma \zeta}\|_{L^{\infty}(\mathcal{Q}^+_{3\rho/2}(\tau_0,y_0))} \cdot 
		\nu^{\frac{2b+n}{6bn}} \, (1-\vartheta)^{-\frac{2b+n}{6bn}} \, \|D_{\tau}^{\alpha+\gamma+p-1-\lambda} D_y^{|\beta|+|\zeta|+q-\sigma} v_j
			\|_{L^{6bn}(\mathcal{Q}^+_{3\rho/2}(\tau_0,y_0))} 
	\\&\leq C \, \nu \, (\max\{2b\lambda+\sigma+s_k-2,0\})! \, (2b\alpha+|\beta|+\nu-2b-2b\lambda-\sigma+s_k-2)! \
	\\&\hspace{10mm} \cdot H^{2b\alpha+|\beta|+\nu+s_k-3-2b} \left(\frac{(\nu-1) (1-\vartheta)}{\nu}\right)^{-2b\alpha-|\beta|-\nu-s_k+1} 
\end{align*}
for some constant $C \in (0,\infty)$, noting that Lemma \ref{majorants_lemma}(i)(ii) applies since $2b\lambda+\sigma+s_k+t_j < t_j+\nu-2b$ and that $2b\alpha+|\beta|+\nu-2b\lambda-\sigma+s_k-2-4b \geq 0$ by \eqref{parareg_eqn8}.  If instead $2b\sigma+\lambda > \nu+s_k-2-4b$, by Lemma \ref{majorants_lemma}(i)(iii), the induction hypotheses, and \eqref{parareg_sobolev}
\begin{align*}
	&\rho^{-\frac{2b+n}{6bn}} \, \|D_{\tau}^{\lambda} D_y^{\sigma} a_{kj}^{\gamma \zeta}\|_{L^{6bn}(\mathcal{Q}^+_{3\rho/2}(\tau_0,y_0))} 
		\cdot \|D_{\tau}^{\alpha+\gamma+p-1-\lambda} D_y^{|\beta|+|\zeta|+q-\sigma} v_j\|_{L^{\infty}(\mathcal{Q}^+_{3\rho/2}(\tau_0,y_0))} 
	\\&\leq C \, (2b\lambda+\sigma+s_k-1)! \, (\max\{2b\alpha+|\beta|+\nu-2b-2b\lambda-\sigma+s_k-2,0\})! \
	\\&\hspace{10mm} \cdot H^{2b\alpha+|\beta|+\nu+s_k-3-2b} \left(\frac{(\nu-1) (1-\vartheta)}{\nu}\right)^{-2b\alpha-|\beta|-\nu-s_k+1} 
\end{align*}
for some constant $C \in (0,\infty)$, noting that Lemma \ref{majorants_lemma}(iii) applies since $2b\lambda+\sigma+s_k+t_j < t_j+\nu$, that \eqref{parareg_sobolev} applies since $t_j + \nu-2b\lambda-\sigma-2b < t_j-s_k+2+4b < t_j+\nu-2b$ by \eqref{parareg_eqn8}, and that $2b\lambda+\sigma+s_k-2-2b > \nu+2s_k-4-6b \geq 0$ by \eqref{parareg_eqn8}.  Putting this all together, 
\begin{align*}
	S \leq {}& C \, \nu \sum_{2b\alpha+|\beta| \leq -s_k} \sum_{\lambda=0}^{\alpha+p-1} \, \sum_{\sigma=0}^{|\beta|+q} {\alpha+p-1 \choose \lambda} 
		{|\beta|+q \choose \sigma} \left(\frac{1-\vartheta}{\nu}\right)^{s_k+2b\alpha+|\beta|} \nonumber \\& \hspace{15mm} \cdot 
		(\max\{2b\lambda+\sigma+s_k-2,0\})! \, (\max\{2b\alpha+|\beta|+s_k+\nu-2b\lambda-\sigma-2b-2,0\})! \nonumber \\& \hspace{15mm} \cdot 
		H^{2b\alpha+|\beta|+\nu+s_k-3-2b} \left(\frac{(\nu-1) (1-\vartheta)}{\nu}\right)^{-2b\alpha-|\beta|-\nu-s_k+1} \nonumber
\end{align*}
for some constant $C \in (0,\infty)$ depending only on $n$, $b$, $Q$, $H_0$, $F_k$, $s_k$, and $t_j$.  Notice that at least one of $2b\lambda+\sigma$ and $2b\alpha+|\beta|+\nu-2b\lambda-\sigma-2b$ is $\geq \nu/2-b$ and thus, recalling \eqref{parareg_eqn8}, 
\begin{align} \label{parareg_eqn10}
	&(\max\{2b\lambda+\sigma+s_k-2,0\})! \, (\max\{2b\alpha+|\beta|+s_k+\nu-2b\lambda-\sigma-2b-2,0\})!  
	\\&\leq C(s_k,b) \, \nu^{s_k} \, (\max\{2b\lambda+\sigma-2,0\})! \, (\max\{2b\alpha+|\beta|+\nu-2b\lambda-\sigma-2b-2,0\})! \nonumber
\end{align}
Using \eqref{parareg_eqn10}, \eqref{parareg_combo}, and $(1-1/\nu)^{1-\nu} \leq e$, 
\begin{align*}
	S &\leq C \, e \sum_{2b\alpha+|\beta| \leq -s_k} \nu^{1-2b\alpha-|\beta|} \, (2b\alpha+|\beta|+\nu-2b-2)! \, H^{\nu-3-2b} \, (1-\vartheta)^{1-\nu}
	\\&\leq C \, e \, (\nu-2)! \, H^{\nu-3-2b} \, (1-\vartheta)^{1-\nu}  
\end{align*}
for some constants $C \in (0,\infty)$ depending only on $n$, $b$, $Q$, $H_0$, $F_k$, $s_k$, and $t_j$.  Therefore,
\begin{align} \label{parareg_eqn11}
	&\sum_{2b\alpha+|\beta| < -s_h} \rho^{s_h+2b\alpha+|\beta|-\frac{2b+n}{6bn}} \, 
		\|D_{\tau}^{\alpha} D_y^{\beta} g_k\|_{L^{6bn}(\mathcal{Q}^+_{3\rho/2}(\tau_0,y_0))} 
	\\&\leq C \, (\nu-2)! \, H^{\nu-3-2b} \, (1-\vartheta)^{1-\nu} \nonumber 
\end{align}
for $k = 1,2,\ldots,Q$ for some constant $C \in (0,\infty)$ depending only on $n$, $b$, $Q$, $H_0$, $F_k$, $s_k$, and $t_j$.  Similarly, 
\begin{align} \label{parareg_eqn12}
	&\sum_{2b\alpha+|\beta| \leq -r_h} \rho^{r_h+2b\alpha+|\beta|-\frac{2b+n}{6bn}} \, 
		\|D_{\tau}^{\alpha} D_y^{\beta} \psi_h \|_{L^{\infty}(\mathcal{Q}^+_{3\rho/2}(\tau_0,y_0))} 
	\\&\leq C \, (\nu-2)! \, H^{\nu-3-2b} \, (1-\vartheta)^{1-\nu} \nonumber 
\end{align}
for $h = 1,2,\ldots,bR$ for some constant $C \in (0,\infty)$ depending only on $n$, $b$, $Q$, $H_0$, $\Phi_h$, $s_k$, and $r_h$. 

By induction hypothesis and $(1-1/\nu)^{1-\nu} \leq e$, 
\begin{align} \label{parareg_eqn13}
	&\sum_{2b\alpha+|\beta| < t_j} \rho^{-t_j+2b\alpha+|\beta|-\frac{2b+n}{6bn}} \, 
		\|D_{\tau}^{\alpha+p} D_y^{\beta} D_y^q v_j \|_{L^{6bn}(\mathcal{Q}^+_{3\rho/2}(\tau_0,y_0))} \\ 
	&\leq C(n,b) \sum_{2b\alpha+|\beta| < t_j} \left(\frac{1-\vartheta}{\nu}\right)^{-t_j+2b\alpha+|\beta|} \, (\nu+2b\alpha+|\beta|-t_j-2)! \nonumber \\&
		\hspace{15mm} \cdot H_0 \, H^{\nu+2b\alpha+|\beta|-t_j-2-2b} \left(\frac{(\nu-1) (1-\vartheta)}{\nu}\right)^{-\nu-2b\alpha-|\beta|+t_j} \nonumber \\
	&\leq C(n,b,t_j) \, (\nu-2)! \, H_0 \, H^{\nu-3-2b} \, (1-\vartheta)^{-\nu}. \nonumber 
\end{align}

By substituting \eqref{parareg_eqn11}, \eqref{parareg_eqn12}, and \eqref{parareg_eqn13} into \eqref{parareg_eqn6} with $D_{\tau}^p D_y^q v_j$ in place of $v_j$,  
\begin{equation*}
	\sum_{j=1}^Q \, \sum_{2b\alpha+|\beta| = t_j+\nu} \rho^{-\frac{2b+n}{6bn}} \, \|D_{\tau}^{\alpha} D_y^{\beta} v_j\|_{L^{6bn}(\mathcal{Q}^+_{\rho}(\tau_0,y_0))} 
	\leq C \, (\nu-2)! \, H^{\nu-3-2b} \, (1-\vartheta)^{-\nu}
\end{equation*}
if $2b\alpha+|\beta| = t_j+\nu$ for some constant $C \in (0,\infty)$ depending only on $n$, $b$, $Q$, $R$, $H_0$, and the nonlinear operators $F_k$ and $\Phi_h$, and their weights $t_j$, $s_k$, and $r_h$.  Choosing $H \geq C/H_0$, we obtain \eqref{parareg_eqn2} for $2b\alpha+|\beta| = t_j+\nu$. 

Finally, having shown \eqref{parareg_eqn2} for all $\alpha,\beta$, by the Sobolev embedding theorem, \eqref{parareg_concl} holds true. 
\end{proof}

\begin{proof}[Proof of Theorem \ref{mainthm2} in the case $N$ is real-analytic]  
By Theorem \ref{mainthm2} in the case that $N$ is smooth, $\psi, \phi_1, \phi_2, \ldots, \phi_q$ are smooth.  Since $N$ is real-analytic, by Theorem \ref{parareg_thm}, $\psi, \phi_1, \phi_2, \ldots, \phi_q$ are second Gevrey with real-analytic time-slices.  It follows that $u_1,\ldots,u_q$ are second Gevrey with real-analytic time slices near the origin and $\Gamma_t = \{(t,y',\psi(t,y',0),\phi_2(t,y',0)) : t \in I, \, y' = (y_1,\ldots,y_{n-1}) \in S \}$ is second Gevrey with real-analytic time slices near the origin.
\end{proof}

\end{document}